\documentclass[11pt,reqno]{amsart}

\usepackage{amssymb, amsmath, latexsym, amsfonts, amsthm, mathrsfs, mathtools} 
\usepackage{verbatim} 

\usepackage{tikz-cd}

\usepackage{geometry}
\usepackage{mathabx}

\usepackage{hyperref} 
\usepackage{color}
\usepackage{bbm} 
\usepackage{appendix}
\usepackage{tikz-cd}
\usepackage{enumitem}
\usepackage{etoolbox}

\theoremstyle{plain}
\newtheorem{theorem}{Theorem}[section]
\newtheorem{proposition}[theorem]{Proposition}
\newtheorem{corollary}[theorem]{Corollary}
\newtheorem{lemma}[theorem]{Lemma}

\newtheorem{assumptionalt}{Assumption}[theorem]

\newenvironment{assumptionp}[1]{
	\renewcommand\theassumptionalt{#1}
	\assumptionalt
}{\endassumptionalt}

\newtheorem{remark}[theorem]{Remark}

\newtheorem{example}[theorem]{Example}

\newcommand{\N}{\mathbb{N}}
\newcommand{\R}{\mathbb{R}}
\newcommand{\PP}{\mathbb{P}}
\renewcommand{\P}{\PP}
\newcommand{\F}{\mathcal{F}}

\newcommand{\E}{\mathbb{E}}

\newcommand{\e}[1]{\mathrm{e}^{#1}}

\newcommand{\loc}{\mathrm{loc}}

\newcommand{\dd}{\mathrm{d}}
\renewcommand{\d}{\dd}

\DeclarePairedDelimiter{\norm}{\|}{\|}

\newcommand{\1}{\mathbbm{1}}

\DeclarePairedDelimiterX{\vertiii}[1]
{\vvvert}
{\vvvert}
{\ifblank{#1}{\:\cdot\:}{#1}}

\title[Limit theorems for stochastic Volterra processes]{Limit theorems for stochastic Volterra processes}

\author{Luigi Amedeo Bianchi}
\address[Luigi Amedeo Bianchi]{Department of Mathematics
	\\ University of Trento
	\\ Trento, Italy}
\email{luigiamedeo.bianchi@unitn.it}

\author{Stefano Bonaccorsi}
\address[Stefano Bonaccorsi]{Department of Mathematics
	\\ University of Trento
	\\ Trento, Italy}
\email{stefano.bonaccorsi@unitn.it}

\author{Ole Ca\~nadas}
\address[Ole Ca\~nadas]{School of Mathematical Sciences, Dublin City University, Dublin, Ireland}
\curraddr{}
\email{ole.canadas2@mail.dcu.ie}

\author{Martin Friesen}
\address[Martin Friesen]{School of Mathematical Sciences, Dublin City University, Dublin, Ireland}
\email{martin.friesen@dcu.ie}

\date{\today}

\thanks{O.C.~is funded by Research Ireland grant GOIPG/2023/3129. L.A.B.~is partially funded by INdAM - GNAMPA Grant. L.A.B.~and S.B.~have been partially funded by the European Union under NextGenerationEU PRIN 2022 PNRR Prot P2022TX4FE. Financial support by ECIU is gratefully recognised.}

\begin{document}

\begin{abstract}
		We introduce an abstract Hilbert space-valued framework of Markovian lifts for stochastic Volterra equations with operator-valued Volterra kernels. Our main results address the existence and characterisation of possibly multiple limit distributions and stationary processes, a law of large numbers including a convergence rate, and the central limit theorem for time averages of the process within the Gaussian domain of attraction. As particular examples, we study Markovian lifts based on Laplace transforms in a weighted Hilbert space of densities and Markovian lifts based on the shift semigroup on the Filipovi\'c space. We illustrate our results for the case of fractional stochastic Volterra equations with additive or multiplicative Gaussian noise.
\end{abstract}

\maketitle

\section{Introduction}

\subsection{Overview}

The class of \emph{stochastic Volterra processes} provides a flexible and popular way to introduce path-dependence, but also allows for modelling the regularity of sample paths on small time-scales. The applications of Volterra processes extend across a diverse range of fields, including boundary-value problems for Partial Differential Equations or age-structured population dynamics \cite{grippenberg}, measure-valued Markov processes and superprocesses \cite{AbiJaber21,MytnikSalisbury15}, and Stochastic Partial Differential Equations of Volterra type for modeling materials with memory \cite{MR1658451, MR1454409, Pruss93}, but also may be introduced to obtain improved fits to empirical data exhibiting long- or short-range dependence~\cite{BayerFrizGatheral16, Bennedsen17, ElEuchFukasawRosenbaum18, GatheralJaissonRosenbaum18}.

Below, we introduce the general form of the stochastic Volterra equations studied in this work. Let $H, V, H_b, H_{\sigma}$ be separable Hilbert spaces with continuous embedding
\begin{align}\label{eq: inclusion}
        (V, \| \cdot \|_V) \hookrightarrow (H, \| \cdot \|_H).
\end{align}
Let $U$ be another separable Hilbert space, and $(W_t)_{t \geq 0}$ a cylindrical Wiener process on $U$. Given an $\F_0$-measurable $G \in C(\R_+; V)$, a drift $b\colon H \longrightarrow H_b$ and diffusion operator $\sigma\colon H \longrightarrow L(U,H_{\sigma})$, we study limit distributions, stationary solutions, and limit theorems for the stochastic Volterra equation
\begin{equation}\label{eq: mild formulation}
		u(t) = G(t) + \int_0^t E_b(t-s)b(u(s))\, \mathrm{d}s + \int_0^t E_{\sigma}(t-s)\sigma(u(s))\, \mathrm{d}W_s.
\end{equation}
The operators $E_b, E_{\sigma}$ satisfy at least $E_b \in L_{\mathrm{loc}}^1(\R_+; L(H_b, V))$ and $E_{\sigma} \in L^2_{\mathrm{loc}}(\R_+; L(H_{\sigma}, V))$. A solution of \eqref{eq: mild formulation} is an $(\F_t)_{t \geq 0}$-adapted process $u$ with continuous sample paths that satisfies \eqref{eq: mild formulation} a.s., where it is implicitly assumed that all integrals are well-defined. In this formulation, the space $V$ allows for additional spatial regularity inherited from the operators $E_b, E_{\sigma}$. Note that solutions of \eqref{eq: mild formulation} are typically neither Markov processes nor semimartingales.

For many models, formulation \eqref{eq: mild formulation} appears as the mild formulation of the stochastic Volterra equation
\begin{align}\label{VSPDE}
		u(t) = g(t) + \int_0^t k(t-s) \left( Au(s) + b(u(s)) \right) \mathrm{d}s + \int_0^t h(t-s)\sigma(u(s))\, \mathrm{d}W_s
\end{align}
where $(A,D(A))$ is a closed and densely defined linear operator on $H$, $k \in L_{\loc}^1(\R_+)$, and $h\in L^2_\loc(\R_+)$. Following \cite{BBF23}, the relation between $g$ and $G$ is given by
\begin{equation*}\label{eq:G}
        G(t) = g(t) + A \int_0^t E_b(t-s)g(s)\, \mathrm{d}s,
\end{equation*}
and $E_b, E_{\sigma}$ are resolvent operators given as unique solutions to the linear deterministic Volterra equation
\begin{align}\label{eq: resolvent equation}
        E_{\rho}(t) = \rho(t) + A \int_0^t k(t-s)E_{\rho}(s)\, \mathrm{d}s, \qquad \rho \in \{k,h\}.
\end{align}
Stochastic Volterra equations have been studied through their mild formulations in various settings; see, e.g., \cite{MR3274768, MR1658451, MR1454409, FHK23, MR3223334, MR3605714, MR2565842}. If $(A,D(A))$ is the generator of a $C_0$-semigroup $(\e{tA})_{t \geq 0}$ on $H$, $k(t) = h(t) = 1$, and $g(t) = u_0$, then $E_b(t) = E_{\sigma}(t) = \e{tA}$ and $G(t) = \e{tA}u_0$, and \eqref{eq: mild formulation} reduces to the mild formulation of the classical stochastic evolution equation
\[
        \mathrm{d}u(t) = \left( Au(t) + b(u(t)) \right)\, \mathrm{d}t + \sigma(u(t))\, \mathrm{d}W_t
\]
which, under appropriate uniqueness assumptions, determines a Markov process.

The study of long-time behaviour, invariant measures, and limit theorems for stochastic evolution equations forms a central part of the mathematical analysis of stochastic models. In this work, we focus on the characterisation of stationary processes, their associated invariant measures, and investigate convergence towards limit distributions in the Wasserstein distance. In addition, we establish a Law of Large Numbers with an explicit convergence rate, and we derive a Central Limit Theorem within the Gaussian domain of attraction. These limit theorems play a key role in the statistical estimation of model parameters, see \cite{MR2144185} for the general theory of Markov diffusion processes, and \cite{BFK24} for the case of stochastic Volterra processes.

For Markovian models, the long-time behaviour is a classical topic with many powerful techniques and results available, see e.g.~\cite{MR4068305, MR1417491, Kulik}. However, for stochastic Volterra processes \eqref{eq: mild formulation}, results concerning the long-time behaviour are much less developed. Limit distributions and stationary processes have been studied in \cite{FJ24} for multivariate affine Volterra processes on $\R_+^m$, in \cite{MR4503737} for regular kernel with sufficient decay at infinity, in \cite{huber2024} for completely monotone kernels with $H = \R^d$, while \cite{BBF23} addresses limit distributions for a general class of stochastic Volterra processes of the form \eqref{eq: mild formulation}. Concerning limit theorems, the Law of Large Numbers without a convergence rate was recently established in \cite{BFK24a} for an affine process on $\R_+$, while other models, a convergence rate for the Law of Large Numbers, or the Central Limit Theorem for stochastic Volterra processes have, up to our knowledge, not been considered in any meaningful general setting.

\subsection{Methodology and results}

When investigating the long-time behaviour for stochastic Volterra equations, new obstacles arise firstly from the absence of the semimartingale property and secondly from the failure of the Markov property. As a consequence, we cannot apply the (mild) It\^{o} formula to study limit distributions via well-established contraction methods as done, e.g., in \cite{MR4241464, MR2675108, MR2211714}. Due to the path dependence introduced via the Volterra kernels, such processes typically do not possess the Markov property. Thus, the one-dimensional time marginals do not determine the law of the process, which rules out methods based on Kolmogorov equations and successful couplings for Markov processes. Such path-dependence is, for example, reflected in the observation that limit distributions are generally not unique, see \cite{BBF23, FJ24}. Finally, while for Markov processes subgeometric convergence rates are a consequence of nonlinear drifts, see \cite{MR3178490, Kulik}; in the setting of Volterra processes, we observe the emergence of subgeometric (often polynomial) convergence rates even for linear models.

To overcome these obstacles, it is natural to capture the dynamics of stochastic Volterra processes along the whole trajectory, including their past evolution. The latter often allows us to recover the Markov property on an enlarged state space. We call a Markov process $X$ obtained by such a procedure \textit{Markovian lift}. Markovian lifts have been studied, e.g., in \cite{MR3057145, MR2511555, MR1658690, MR1982156,  H23, MR3926553, huber2024} for the context of completely monotone Volterra kernels, \cite{MR4503737} for regular kernels based on the shift semigroup, and \cite{MR3911660, MR4181950} in the context of affine Volterra processes, and in \cite{DG23} for an application of such lifts towards optimal control of Volterra processes. In any case, the choice of such a Markovian lift is certainly not unique and depends on the class of Volterra kernels. We propose an abstract functional analytic framework for Hilbert-space valued Markovian lifts with Volterra kernels that have a weak singularity of order $t^{-\rho}$ with $\rho \in [0,1/2)$ as $t \to 0$. In this framework, we study the long-time behaviour for the corresponding Markovian lift, while results for the stochastic Volterra process are obtained by projection with the operator $\Xi$ given below.

Let $\mathcal{H}, \mathcal{V}$ be separable Hilbert spaces, and let $(S(t))_{t \geq 0}$ be a strongly continuous semigroup on $\mathcal{H}$ which leaves $\mathcal{V}$ invariant, satisfies $S(t) \in L(\mathcal{H}, \mathcal{V})$ for $t > 0$, and
\begin{align}\label{eq: rho}
        \| S(t)\|_{L(\mathcal{H}, \mathcal{V})} \lesssim 1 + t^{- \rho}, \qquad t \in (0,1)
\end{align}
holds for some $\rho \in [0,1/2)$. Let us suppose that the Volterra kernels $E_b, E_{\sigma}$ in \eqref{eq: mild formulation} have the following representation with respect to the semigroup
\begin{align}\label{eq: Volterra kernel lift}
        E_b(t) = \Xi S(t)\xi_b \ \text{ and } \ E_{\sigma}(t) = \Xi S(t)\xi_{\sigma}, \qquad t > 0.
\end{align}
Here $\Xi\colon \mathcal{V} \longrightarrow V$ is a bounded linear operator, $\xi_b \in L(H_b, \mathcal{H})$ and $\xi_{\sigma} \in L(U, \mathcal{H})$ denote \textit{abstract Markovian lifts} of the Volterra kernels. The Hilbert space $\mathcal{H}$ encodes the small-time regularity of the kernels, while $\mathcal{V}$ is chosen in such a way that $\Xi$ is bounded on $\mathcal{V}$. In particular, for regular Volterra kernels, one may take $\mathcal{V} = \mathcal{H}$ with $\rho = 0$, while for the more interesting and challenging case of singular kernels, the operator $\Xi$ is not bounded on $\mathcal{H}$. To treat such cases, additional regularisation properties of the semigroup reflected by \eqref{eq: rho} are essential. Finally, the parameter $\rho$ allows us to prove that the Markovian lift introduced below has continuous sample paths.

For given drift $b\colon H \longrightarrow H_b$ and diffusion operator $\sigma\colon H \longrightarrow L(U, H_{\sigma})$ such that $\xi_{\sigma} \sigma\colon H \longrightarrow L_2(U, \mathcal{H})$, we study the abstract stochastic evolution equation in its mild formulation
\begin{align}\label{eq: abstract mild formulation Markovian lift}
        X_t = S(t)\xi + \int_0^t S(t-s)\xi_b\, b(\Xi X_s)\, \mathrm{d}s + \int_0^t S(t-s)\xi_{\sigma}\, \sigma(\Xi X_s)\, \mathrm{d}W_s.
\end{align}
Such an equation is closely linked to the original stochastic Volterra equation \eqref{eq: mild formulation} via the operator $\Xi$. Namely, if $X$ is a solution of \eqref{eq: abstract mild formulation Markovian lift}, $u(t) \coloneqq \Xi X_t$ solves \eqref{eq: mild formulation}. On the other side, if $u$ is a solution of \eqref{eq: mild formulation}, then defining
\[
        X_t = S(t)\xi + \int_0^t S(t-s)\xi_b\, b(u(s))\, \mathrm{d}s + \int_0^t S(t-s)\xi_{\sigma}\, \sigma(u(s))\, \mathrm{d}W_s
\]
one can verify that $\Xi X = u$ and hence is a solution of \eqref{eq: abstract mild formulation Markovian lift}. Details on this construction, the existence and uniqueness of solutions for \eqref{eq: abstract mild formulation Markovian lift}, and the Markov property are discussed in Section~\ref{sec:markovian_lift}.

One natural and flexible choice of Markovian lift is based on the representation of completely monotone operators in terms of their Bernstein measures. Such types of lifts have been used, e.g., in \cite{H23}. Below, we provide a modified version of this Markovian lift for our infinite-dimensional setting that also allows us to study the long-time behaviour for models that exhibit polynomial rates of convergence.

\begin{example}\label{example 1}
    Let $\mu$ be a $\sigma$-finite Borel measure on $\R_+$, and suppose that $E_b, E_{\sigma}$ have representation $E_b(t) = \int_{\R_+} \e{-xt}\xi_b(x)\, \mu(\mathrm{d}x)$ and $E_{\sigma}(t) = \int_{\R_+} \e{-xt}\xi_{\sigma}(x)\, \mu(\mathrm{d}x)$. For given $\delta, \eta \in \R$, let $\mathcal{H}_{\delta, \eta}$ be the Hilbert space of functions $y\colon \R_+ \longrightarrow V$ with finite norm
    \[
                \vertiii{y}_{\delta, \eta}^2 = \int_{\R_+} \| y(x)\|_V^2 \left( \1_{\{0\}}(x) + \1_{(0,1]}(x)x^{-\delta} + \1_{(1,\infty)}(x)x^{\eta} \right)\, \mu(\mathrm{d}x).
    \]
    Then, we may choose $\mathcal{V}, \mathcal{H}$ as realizations of $\mathcal{H}_{\delta, \eta}$ with suitable $\delta, \eta$, define the $C_0$-semigroup by $S(t)y(x) = \e{-xt}y(x)$, and let $\Xi y = \int_{\R_+} y(x)\, \mu(\mathrm{d}x)$. Further details on this construction are given in Section~\ref{section markovia lift cm}.
\end{example}

Another approach that does not rely on complete monotonicity, but directly works with the kernels $E_b, E_{\sigma}$, is based on the shift-semigroup in the so-called Filipovi\'c space, see~\cite{MR4503737}. Below, we state a modification of their lift that allows for weakly singular kernels and a polynomial rate of convergence to equilibrium.

\begin{example}\label{example 2}
    For given $\delta, \eta \geq 0$, let $\mathcal{H}_{\delta, \eta}$ be the Hilbert space of absolutely continuous functions $y\colon (0,\infty) \longrightarrow V$ with finite norm
    \[
                \vertiii{y}_{\delta, \eta}^2 = \| y(1)\|_V^2 + \int_{\R_+} \| y'(x)\|_V^2 \left( \1_{(0,1]}(x)x^{\eta} + \1_{(1,\infty)}(x)x^{\delta} \right)\, \mathrm{d}x.
    \]
    Let $\mathcal{V}, \mathcal{H}$ be realizations of $\mathcal{H}_{\delta, \eta}$ with suitable values for $\delta, \eta \geq 0$, define the $C_0$-semigroup by $S(t)y(x) = y(x+t)$, and projection operator by $\Xi y = y(0) = y(1) - \int_0^1 y'(x)\, \mathrm{d}x$. Then \eqref{eq: Volterra kernel lift} holds for $\xi_b = E_b$ and $\xi_{\sigma} = E_{\sigma}$. Details are discussed in Section~\ref{section HJM lift}.
\end{example}

Remark that, although the lift discussed in Example \ref{example 2} appears to be more general, the lift described in Example \ref{example 1} is not redundant. For instance, the semigroup defined in Example \ref{example 1} is analytic, which is not the case for Example \ref{example 2}. Moreover, the lift discussed in Example \ref{example 1} can be used to develop a flexible framework for finite-dimensional Markovian approximations, see \cite{MR3934104, MR1658690, MR3926553}.

Since equation \eqref{eq: abstract mild formulation Markovian lift} determines a Markov process, we may use tools from the Markovian framework to study its long-time behaviour. However, such methods typically rely on spectral conditions on the operator $(A, D(A))$ and generally require the associated semigroup $(S(t))_{t \geq 0}$ to exhibit exponential uniform stability, or at least exponential uniform ergodicity in the sense that
\[
    \|S(t) - S_{\infty}\|_{L(\mathcal{V})} \lesssim \e{-\lambda t}, \qquad t > 0,
\]
see \cite{MR1417491, MR4503737, MR4241464}. In both cases, one obtains geometric convergence to equilibrium, which excludes several interesting examples arising in the context of stochastic Volterra processes. Moreover, due to the strong continuity of the semigroup $(S(t))_{t \geq 0}$, any form of uniform stability or ergodicity on $\mathcal{V}$ necessarily entails an exponential rate.

To obtain results for subgeometric convergence rates, in Section~\ref{sec:limit_dist_inv_meas}, we develop a contraction method for \eqref{eq: abstract mild formulation Markovian lift} under the weaker condition
\[
    \| S(t) - S_{\infty}\|_{L(\mathcal{V}, \mathcal{V}_0)} \lesssim (1 \vee t)^{- \lambda}
\]
where $\lambda > 0$, and $\mathcal{V}_0$ is another separable Hilbert space such that $\mathcal{V} \hookrightarrow \mathcal{V}_0$. The operator $S_{\infty}$ determines the structure of all invariant measures and hence encodes memory effects that persist in the large time asymptotics. In this framework, we study in Section~\ref{sec:limit_dist_inv_meas} limit distributions in the Wasserstein distance for the case of small nonlinearities. In our first main result, Theorem~\ref{theorem_limit_distribution}, we establish for each initial condition $\xi$ the existence of a unique limit distribution $\pi_{\xi}$ (which is also an invariant measure) with the property
\begin{align}\label{eq: Wasserstein intro}
    \mathcal{W}_{p, \mathcal{V}_0}( X^{\xi}_t, \pi_{\xi} ) \lesssim (1 \vee t)^{- \chi}
\end{align}
where $\chi > 0$ denotes the polynomial rate of convergence. The dependence of $\pi_{\xi}$ on $\xi$ reflects the presence of memory. In particular, we find $\pi_{\xi} = \pi_{S_{\infty}\xi}$, i.e.~all limit distributions are fully parameterised by the range of $S_{\infty}$. In Corollary \ref{corollary Pi operator}, we provide another characterisation of limit distributions/invariant measures in terms of the operator $\Pi$ formally given as the limit of the transition semigroup $\lim_{t \to \infty}P_t = \Pi$. Finally, we show that this limit $\Pi$ is an integral operator with respect to a subclass of invariant measures $\pi_{\xi}$ with $\xi$ deterministic.

Under the same conditions, in Section~\ref{sec:limit_theorems}, we proceed to study corresponding limit theorems. Firstly, in Theorem \ref{thm: LLN abstract} we prove an abstract Law of Large Numbers for Markov processes that admit a unique invariant measure and satisfy \eqref{eq: Wasserstein intro}. Afterwards, in Theorem \ref{theorem LLN lift} we derive the desired Law of Large Numbers for \eqref{eq: abstract mild formulation Markovian lift}, i.e.~
\begin{equation*}\label{eq: LLN intro}
    \E\left[ \left |\frac{1}{t}\int_0^t f(X^{\xi}_s)\, \mathrm{d}s - (\Pi f)(\overline{\xi}) \right|^2 \right] \lesssim t^{- \vartheta}
\end{equation*}
where $\vartheta > 0$ denotes the rate of convergence, and $\overline{\xi}$ denotes the stationary process associated with the limit distribution of $\pi_{\xi}$. Finally, in Theorem \ref{theorem CLT general}, we show that, if \eqref{eq: Wasserstein intro} holds with $\chi > 1$, then the time averages lie in the Gaussian domain of attraction, i.e.~
\begin{equation*}\label{eq: CLT intro}
    \sqrt{t}\left( \frac{1}{t}\int_0^t f(X_s)\, \mathrm{d}s - (\Pi f)(\xi) \right) \Longrightarrow \sigma(\xi)Z
\end{equation*}
where $Z \sim \mathcal{N}(0,1)$ and $\sigma \geq 0$ denotes the standard deviation. Remark that, due to the occurrence of multiple invariant measures, $\Pi f$ and $\sigma$ are not constants but functions evaluated at the stationary process $\overline{\xi}$. Finally, in Sections \ref{section markovia lift cm} and \ref{section HJM lift}, we illustrate our results for the case of Markovian lifts based on fractional Volterra kernels.

\subsection{Structure of the work}

In Section~\ref{sec:markovian_lift}, we introduce the abstract Markovian lift framework, discuss properties of stochastic convolutions, and show that under Lipschitz conditions, equation \eqref{eq: abstract mild formulation Markovian lift} has a unique solution that determines a Markov process. Based on the contraction method, limit distributions and a characterisation of invariant measures are then given in Section~\ref{sec:limit_dist_inv_meas}. The Law of Large Numbers, including a convergence rate, and the Central Limit Theorem for abstract Markovian lifts are studied in Section~\ref{sec:limit_theorems}. Examples of Markovian lifts are subsequently discussed in Sections~\ref{section markovia lift cm} and~\ref{section HJM lift}. Finally, some auxiliary results are collected in the appendix.

\subsection{Notation}
	We write $|\cdot|$ for the standard Euclidean norm on $\R^d$. Moreover, we denote for a Banach space $Y$ by $C([0,T];Y)$, $C^\theta([0,T];Y)$, $\theta\in(0,1]$, the spaces consisting of functions $f\colon[0,T]\longrightarrow Y$ that are continuous, $\theta$-H\"older continuous, respectively. For a measure space $(A,\mathcal{A},\mu)$, we write $L^p(A,\mathcal{A},\mu;Y) = L^p(A, \mu ; Y) = L^p(A; Y)$ for the space of all equivalence classes of Bochner $p$-integrable functions $f\colon(A,\mathcal{A},\mu)\longrightarrow(Y,\mathcal{B}(Y))$, where $\mathcal{B}(Y)$ denotes the Borel-$\sigma$-algebra over $Y$. If $Y = \R$, we also write $L^p(A;\R) = L^p(A)$. For separable Hilbert spaces $H_1, H_2$ and $p\in[1,\infty]$, we let $L_p(H_1, H_2)$ be the $p$-th Schatten class where $L_\infty(H_1, H_2)\coloneqq L(H_1, H_2)$ is the space of all bounded and linear operators equipped with the operator norm, and
    \[
    L_p(H_1,H_2) \coloneqq \left\{A\colon H_1\longrightarrow H_2 \text{ compact}: \norm{A}_{L_p(H_1,H_2)}^p \coloneqq \sum_{\lambda\in\sigma(A^\ast A)}\lambda^{p/2}<\infty\right\}
    \]
    when $p \in [1,\infty)$. Here $\sigma(A^\ast A)$ denotes the spectrum of the (compact and positive) operator $A^\ast A$. In particular, $L_2(H_1,H_2)$ denotes the space of Hilbert-Schmidt operators from $H_1$ to~$H_2$. Let $H_3$ be another separable Hilbert space. Then, the Schatten norms obey H\"older's inequality, i.e.,
	for $p,q,r\in[1,\infty]$ with $\frac{1}{r}=\frac{1}{p}+\frac{1}{q}$ and $A\in L_p(H_2,H_3)$, $B\in L_q(H_1,H_2)$ holds
		\[
		\norm{AB}_{L_r(H_1,H_3)}\leq \norm{A}_{L_p(H_2,H_3)}\norm{B}_{L_q(H_1,H_2)}.
		\]
    Finally, to avoid the introduction of several (irrelevant) multiplicative constants, we use the symbol $\lesssim$, which stands for inequality up to a multiplicative constant, i.e., $x\lesssim y$ if $x\leq cy$.

\section{Abstract Markovian lift}\label{sec:markovian_lift}

\subsection{Functional analytic framework}

	Let $V, H$ be separable Hilbert spaces such that $V \hookrightarrow H$ continuously. On a filtered probability space $(\Omega,\F,(\F_t)_{t\geq0},\P)$ let $(W_t)_{t\geq0}$ be a cylindrical $(\F_t)_{t\geq0}$-Wiener process on another separable Hilbert space $U$. We study the abstract Markovian lift \eqref{eq: abstract mild formulation Markovian lift} as a stochastic evolutionary equation under the following set of assumptions.

	\begin{assumptionp}{A}\label{assumption SEE}
		There exist separable Hilbert spaces $\mathcal{H}, \mathcal{V}$, a bounded linear operator $\Xi\colon \mathcal{V} \longrightarrow V$, and a $C_0$-semigroup $(S(t))_{t\geq0}$ on $\mathcal{H}$ such that $S(t) \in L(\mathcal{H}, \mathcal{V})$ for $t > 0$, and there exists $\rho \in[0,1/2)$ and for each $T > 0$ a constant $C_0(\mathcal{H},\mathcal{V}, T)>0$ such that
		\begin{align}\label{eq: semigroup regularization}
			\norm{S(t)}_{L(\mathcal{H},\mathcal{V})}\leq C_0(\mathcal{H},\mathcal{V},T)(1+t^{-\rho}), \qquad t \in [0,T].
		\end{align}
        Moreover, $\mathcal{H} \cap \mathcal{V} \subset \mathcal{V}$ is dense, and $(S(t)|_{\mathcal{H} \cap \mathcal{V}})_{t \geq 0}$ extends to a $C_0$-semigroup on $\mathcal{V}$ which we again denote by $(S(t))_{t \geq 0}$.
	\end{assumptionp}

	Assumption \ref{assumption SEE} provides a minimal set of assumptions under which we can study \eqref{eq: abstract mild formulation Markovian lift}. Namely, we impose the existence of an abstract projection operator $\Xi\colon \mathcal{V} \longrightarrow V$ that relates the abstract Markovian lift with the original stochastic Volterra equation. For regular kernels, we may take $\mathcal{V} = \mathcal{H}$, while $\mathcal{V} \neq \mathcal{H}$ allows for singular kernels. Secondly, for the strongly continuous semigroup $(S(t))_{t \geq 0}$ on $\mathcal{H}$, we suppose that it is regularizing in the sense that $S(t) \in L(\mathcal{H}, \mathcal{V})$ for $t > 0$. The latter is necessary for singular kernels, in which case $\rho > 0$, while for regular kernels we have $\mathcal{V} = \mathcal{H}$ and hence $\rho = 0$.

    The composition $\Xi S(\cdot)$ plays a central role in the study of Markovian lifts. Its mapping properties are summarised in the following remark.
    \begin{remark}
        Under Assumption \ref{assumption SEE}, let $y \in \mathcal{H}$ and set $g(t) = \Xi S(t)y$. Then
        \[
            \|g(t)\|_{V} \leq \| \Xi\|_{L(\mathcal{V}, V)}C(\mathcal{H}, \mathcal{V}, T)(1 + t^{-\rho})\norm{y}_{\mathcal{H}}, \qquad t \in [0,T].
        \]
        Thus, for regular Volterra kernels (where $\rho = 0$), the function $g$ is bounded, while for singular kernels with $\rho > 0$, it has a singularity at $t = 0$. Moreover, if $y \in \mathcal{V}$, then
        \[
            \| g(t)\|_V \leq \| \Xi\|_{L(\mathcal{V}, V)} \sup_{t \in [0,T]}\|S(t)\|_{L(\mathcal{V})} \| y \|_{\mathcal{V}}.
        \]
    \end{remark}

	As a first step, we study the mapping properties of the convolution operators
	\begin{align}\label{eq: convolutions}
		t \longmapsto I_b(t) \coloneqq \int_0^t S(t-s)\xi_b(s)\, \mathrm{d}s \ \text{ and } \ t \longmapsto I_{\sigma}(t) \coloneqq \int_0^t S(t-s)\xi_{\sigma}(s)\, \mathrm{d}W_s
	\end{align}
	as $\mathcal{V}$-valued processes.

	\begin{proposition}\label{proposition_properties_lift}
		Suppose that Assumption \ref{assumption SEE} is satisfied. Then for each $p\in[1,\infty)$ with $\frac{1}{p} + \rho < 1$ and each $T > 0$ we obtain
		\[
		      \norm{I_b}_{L^p(\Omega;C([0,T]);\mathcal{V})}
              \leq \left( \int_0^T \|S(r)\|_{L(\mathcal{H}, \mathcal{V})}^{\frac{p}{p-1}}\, \mathrm{d}r \right)^{p-1} \| \xi_{b} \|_{L^{p}([0,T]; L^p(\Omega; \mathcal{H}))}
		\]
		for each predictable process $\xi_b \in L^p(\Omega,\P;L^p([0,T]; \mathcal{H}))$. Furthermore, for each $p\in(2,\infty)$ with $\frac{1}{p} + \rho < \frac{1}{2}$ and each $0 < \alpha < \frac{1}{2} - \rho$ there exists a constant $A(p,\rho, \alpha, T)>0$ such that
		\[
		\norm{I_\sigma}_{L^p(\Omega;C([0,T];\mathcal{V}))} \leq A(p,\rho, \alpha, T) \left( \int_0^T r^{-2\alpha} \|S(r)\|_{L(\mathcal{H}, \mathcal{V})}^2\, \mathrm{d}r \right)^{\frac{1}{2}} \| \xi_{\sigma} \|_{L^{\infty}([0,T]; L^p (\Omega; L_2(U,\mathcal{H})))}
		\]
		holds for each predictable process $\xi_{\sigma} \in L^{\infty}([0,T]; L^p(\Omega,\P; L_2(U,\mathcal{H})))$.
	\end{proposition}
	\begin{proof}
		Firstly, $I_b$ has a continuous version by \cite[Lemma A.2]{BBF23}. The inequality given therein also yields for $\frac{1}{p} + \frac{1}{q} = 1$
		\begin{align}\nonumber
			\E\left[ \sup_{t \in [0,T]}\|I_b(t)\|_{\mathcal{V}}^p \right]
			&\leq \left(\int_0^T \|S(r)\|_{L(\mathcal{H}, \mathcal{V})}^q\, \mathrm{d}r\right)^{\frac{p}{q}}  \int_0^T \E\left[\|\xi_b(r) \|_{\mathcal{H}}^p\right] \, \mathrm{d}r
		\end{align}
        This proves the first assertion since $q = p/(p-1)$ and $p > 1$ due to $1/p + \rho < 1$.

		We apply the factorisation method from~\cite[Theorem 5.10]{DaPrato_Zabczyk_2014} for the second inequality. Fix $T > 0$ and since $\rho \in (0,1/2)$, we may take $0 < \alpha < \frac{1}{2} - \rho$. Let us show that
		\[
		Y_{\alpha}(t) = \int_0^t (t-s)^{-\alpha}S(t-s)\xi_{\sigma}(s)\, \mathrm{d}W_s, \qquad t \in [0,T]
		\]
		satisfies $Y_{\alpha} \in L^{\infty}([0,T]; L^p(\Omega,\P; \mathcal{V}))$. Indeed, by an application of \cite[Theorem 4.36]{DaPrato_Zabczyk_2014} and then Jensen inequality we obtain
		\begin{align*}
			& \E\left[\|Y_{\alpha}(t)\|_{\mathcal{V}}^p \right]
			\\ &\quad\lesssim \E\left[ \left( \int_0^t (t-s)^{-2\alpha} \|S(t-s)\xi_\sigma(s)\|_{L_2(U, \mathcal{V})}^2\, \mathrm{d}s\right)^{p/2} \right]\nonumber
			\\ &\quad\lesssim \E\left[ \left( \int_0^t (t-s)^{-2\alpha}\|S(t-s)\|_{L(\mathcal{H}, \mathcal{V})}^2 \|\xi_{\sigma}(s)\|_{L_2(U,\mathcal{H})}^2\, \mathrm{d}s\right)^{p/2} \right] \nonumber
            \\ &\quad\lesssim \left( \int_0^t r^{-2\alpha} \|S(r)\|_{L(\mathcal{H}, \mathcal{V})}^2\, \mathrm{d}r \right)^{\frac{p}{2} - 1} \int_0^t (t-s)^{-2\alpha} \|S(t-s)\|_{L(\mathcal{H}, \mathcal{V})}^2 \E\left[ \|\xi_{\sigma}(s)\|_{L_2(U,\mathcal{H})}^p\right] \, \mathrm{d}s \notag
            \\ &\quad\lesssim \left( \int_0^T r^{-2\alpha} \|S(r)\|_{L(\mathcal{H}, \mathcal{V})}^2\, \mathrm{d}r \right)^{\frac{p}{2}} \| \xi_{\sigma} \|_{L^{\infty}([0,T]; L^p(\Omega; L_2(U,\mathcal{H})))}^p
		\end{align*}
		for $t \in [0, T]$. Since $2(\alpha + \rho) < 1$, the integral on the right-hand side is well-defined, and hence $Y_{\alpha} \in L^{\infty}([0, T]; L^p(\Omega,\P; \mathcal{V}))$. Similarly, we show that
		\begin{align*}
			& \int_0^s (s-r)^{-2\alpha} \E\left[ \|S(t-r)\xi_{\sigma}(r) \|_{L_2(U, \mathcal{V})}^2 \right] \mathrm{d}r
            \\ &\quad\lesssim \| \xi_{\sigma} \|_{L^{\infty}([0,T]; L^2(\Omega; L_2(U,\mathcal{H})))}^2 \int_0^s (s-r)^{-2\alpha} \|S(t-r)\|_{L(\mathcal{H}, \mathcal{V})}^2 \mathrm{d}r
            \\ &\quad\lesssim \| \xi_{\sigma} \|_{L^{\infty}([0,T]; L^2(\Omega; L_2(U,\mathcal{H})))}^2 \int_0^s r^{-2\alpha} \|S(t-s + r)\|_{L(\mathcal{H}, \mathcal{V})}^2 \mathrm{d}r
		\end{align*}
		which implies for each $t \in [0,T]$
		\begin{align*}
			& \int_0^t (t-s)^{\alpha - 1}\left( \int_0^s (s-r)^{-2\alpha} \E\left[ \|S(t-r)\xi_{\sigma}(r) \|_{L_2(U, \mathcal{V})}^2 \right]\, \mathrm{d}r \right)^{1/2}\, \mathrm{d}s
            \\ &\quad\lesssim \| \xi_{\sigma} \|_{L^{\infty}([0,T]; L^2(\Omega; L_2(U,\mathcal{H})))} \int_0^t (t-s)^{\alpha-1} \left( \int_0^s r^{-2\alpha} \|S(t-s + r)\|_{L(\mathcal{H}, \mathcal{V})}^2 \mathrm{d}r \right)^{1/2}\, \mathrm{d}s
			\\ &\quad\lesssim \| \xi_{\sigma} \|_{L^{\infty}([0,T]; L^2(\Omega; L_2(U,\mathcal{H})))}  \int_0^t (t-s)^{\alpha - 1}(s^{\frac{1}{2}-\alpha}+s^{\frac{1}{2} - (\alpha + \rho)})\, \mathrm{d}s
			\\ &\quad\lesssim \| \xi_{\sigma} \|_{L^{\infty}([0,T]; L^2(\Omega; L_2(U,\mathcal{H})))} (t^{\frac{1}{2}}+t^{\frac{1}{2} - \rho}) < \infty.
		\end{align*}
		Hence the factorization formula \cite[Theorem 5.10]{DaPrato_Zabczyk_2014} yields
		\[
		\int_0^t S(t-s)\xi_{\sigma} \sigma_s\, \mathrm{d}W_s = \frac{\sin(\alpha \pi)}{\pi}\int_0^t (t-s)^{\alpha - 1}S(t-s)Y_{\alpha}(s)\, \mathrm{d}s.
		\]
		An application of \cite[Proposition 5.9]{DaPrato_Zabczyk_2014} for $E_1 = E_2 = \mathcal{V}$ and $r = 0$ shows that the right-hand side is continuous in $t$, which provides the desired continuous modification. For the inequality, let us note that
		\begin{align*}
			\norm{I_\sigma}_{L^p(\Omega;C([0,T];\mathcal{V}))}^p
			&= \E\left[ \sup_{t \in [0,T]}\left\| \frac{\sin(\alpha \pi)}{\pi}\int_0^t (t-s)^{\alpha - 1}S(t-s)Y_{\alpha}(s)\, \mathrm{d}s \right\|_{\mathcal{V}}^p \right] \notag
			\\ \notag &\lesssim \E\left[ \|Y_{\alpha}\|_{L^p([0,T]; \mathcal{V})}^p \right]
			\\ &\lesssim \left( \int_0^T r^{-2\alpha} \|S(r)\|_{L(\mathcal{H}, \mathcal{V})}^2\, \mathrm{d}r \right)^{\frac{p}{2}} \| \xi_{\sigma} \|_{L^{\infty}([0,T]; L^p(\Omega; L_2(U,\mathcal{H})))}^p
		\end{align*}
		where the first inequality follows from \cite[Proposition 5.9]{DaPrato_Zabczyk_2014}. This proves the assertion.
	\end{proof}

	The next proposition strengthens the bounds to H\"older continuous sample paths provided that the semigroup has additional regularity. Such a condition is similar to the factorisation lemma used in \cite{MR1417491} for analytic semigroups.

	\begin{proposition}\label{prop: Hoelder}
		Suppose that, additionally to Assumption \ref{assumption SEE}, $(0,\infty) \ni t \longmapsto S(t) \in L(\mathcal{H}, \mathcal{V})$ is differentiable, and for each $T > 0$ there exists a constant $C_1(\mathcal{H}, \mathcal{V}, T) > 0$ such that
		\[
		\norm{\Dot{S}(t)}_{L(\mathcal{H},\mathcal{V})}\leq C_1(\mathcal{H},\mathcal{V}, T)(1+t^{-1-\rho}), \qquad t \in [0,T].
		\]
		Then for each $p\in[1,\infty)$ with $\frac{1}{p} + \rho < 1$, and $\theta \in (0,1-\rho)$ there exists some constant $A_0=A_0(\theta,\mathcal{H},\mathcal{V},\rho,T)>0$ such that
		\[
		\norm{I_b}_{L^p(\Omega;C^\theta([0,T],\mathcal{V}))}\leq A_0\norm{\xi_b}_{L^p(\Omega;L^\infty([0,T]; \mathcal{H}))}
		\]
		holds for each predictable process $b\in L^p(\Omega,\P;L^\infty([0,T]; \mathcal{H}))$. Furthermore, for each $p\in(2 ,\infty)$ and $\theta$ satisfying  $0<\theta<\frac{1}{2}-\frac{1}{p}-\rho$ there exists some constant $A_1=A_1(\theta,\mathcal{H},\mathcal{V},\rho,T)>0$ such that
		\[
		\norm{I_\sigma}_{L^p(\Omega;C^\theta([0,T],\mathcal{V}))} \leq A_1\norm{\xi_{\sigma}}_{L^p(\Omega;L^\infty([0,T];L_2(U,\mathcal{H})))}
		\]
		holds for each predictable process $\xi_{\sigma} \in L^p(\Omega,\P;L^\infty([0,T];L_2(U,\mathcal{H})))$.
	\end{proposition}
	\begin{proof}
		Using Assumption \ref{assumption SEE} we arrive for $x\in \mathcal{H}$ at the inequality
		\begin{equation*}
			t^\theta\norm{\Dot{S}(t)x}_{\mathcal{V}}
			+ t^{\theta-1}\norm{S(t)x}_{\mathcal{V}}
			\leq \widetilde{C}(\mathcal{H},\mathcal{V},T)\left( t^{\theta} + t^{\theta-1} + t^{\theta - 1 - \rho}\right)\norm{x}_{\mathcal{H}}
		\end{equation*}
		where $\theta \in (0,1)$, and $\widetilde{C}(\mathcal{H},\mathcal{V},T)$ can be computed from the constants in Assumption \ref{assumption SEE}. In particular, the right-hand side is locally integrable for $\theta>\rho$. Hence, the first inequality follows from \cite[Lemma 2.7]{FHK23} applied to $Y_1 = \mathcal{H}$, $Y_2=\mathcal{V}$, $\Psi(t)=S(t)$, $\Phi(t)= \xi_b(t)$. Similarly, the second inequality follows from \cite[Lemma 2.8]{FHK23} applied to $Y_1 = \mathcal{H}$, $Y_2=\mathcal{V}$, $\Psi(t)=S(t)$, $\Phi_t(t)= \xi_{\sigma}(t)$.
	\end{proof}

These propositions form the central tool to construct a Markovian lift with continuous sample paths from a given solution $u$ of the stochastic Volterra equation \eqref{eq: mild formulation}. Let us remark that the regularisation property \eqref{eq: semigroup regularization} can be replaced by the weaker assumption
\[
    \int_0^T \|S(t)\|_{L(\mathcal{H}, \mathcal{V})}^2\, \mathrm{d}t < \infty.
\]
However, in this case, $I_b, I_{\sigma}$ are only defined as elements in $L^2( [0, T]; \mathcal{V})$ a.s., see also \cite{H23} for the case of completely monotone Volterra kernels and finite-dimensional stochastic Volterra equations. Thus, our slightly stronger assumption \eqref{eq: semigroup regularization} allows us to study Markovian lifts with continuous sample paths.

\subsection{Markov solutions}

In this section, we study the existence and uniqueness of solutions to the abstract Markovian lift with time-inhomogeneous coefficients given by
\begin{align}\label{eq: abstract Markov lift time inhomogeneous}
            X(t;s,\xi) &= S(t-s)\xi + \int_s^t S(t-r)\xi_b(r,\Xi X(r; s,\xi))\, \mathrm{d}r
            \\ &\qquad \qquad \qquad \qquad  + \int_s^t S(t-r)\xi_{\sigma}(r, \Xi X(r; s,\xi))\, \mathrm{d}W_r, \qquad 0 \leq s \leq t \leq T. \nonumber
\end{align}
Remark that, formally, $u(t) = \Xi X_t$ satisfies the stochastic Volterra equation
\[
    u(t) = G(t) + \int_0^t E_b(t,s, u(s))\, \mathrm{d}s + \int_0^t E_{\sigma}(t,s,u(s))\, \mathrm{d}W_s
\]
where we have set $G(t) = \Xi S(t)\xi$, $E_b(t,s,u) = \Xi S(t-s)\xi_b(s,u)$, and $E_{\sigma}(t,s,u) = \Xi S(t-s)\xi_{\sigma}(s,u)$. For the rigorous treatment of such equations, let us suppose, in addition to Assumption \ref{assumption SEE}, the following set of conditions on a fixed interval $[0, T]$ with $T>0$:
\begin{assumptionp}{B}\label{assumption Lipschitz}
		There exist measurable functions $[0,T] \times H \ni (t,u) \longmapsto \xi_b(t,u) \in \mathcal{H}$ and $[0,T] \times H \longmapsto \xi_{\sigma}(t,u) \in L_2(U, \mathcal{H})$ such that for some $C_{\mathrm{lip}}(T) \geq 0$
        \[
            \| \xi_b(t,u) - \xi_b(t,v)\|_{\mathcal{H}} + \| \xi_{\sigma}(t,u) - \xi_{\sigma}(t,v)\|_{L_2(U, \mathcal{H})} \leq C_{\mathrm{lip}}(T)\|u-v\|_H
        \]
        holds for all $u,v \in H$ and $t \in [0,T]$.
	\end{assumptionp}

Based on the bounds for the convolutions \eqref{eq: convolutions}, the existence and uniqueness for the Markovian lift given by the stochastic evolution equation \eqref{eq: abstract mild formulation Markovian lift} can be obtained by the usual fixed-point procedure.

	\begin{theorem}\label{label:existence_uniqueness_VSPDE}
		Suppose that Assumptions \ref{assumption SEE} and \ref{assumption Lipschitz} are satisfied. Then for each $s \in [0,T)$ and each $\xi \in L^p(\Omega, \F_s, \P; \mathcal{V})$ with $p \in (2,\infty)$ satisfying
		\begin{align}\label{eq: continuity}
			\rho + \frac{1}{p} < \frac{1}{2},
		\end{align}
		there exists a unique solution $X(\cdot; s, \xi) \in L^p(\Omega,\P; C([s,T]; \mathcal{V}))$ of \eqref{eq: abstract Markov lift time inhomogeneous}. Moreover, there exists a constant $C_T > 0$ independent of $\xi$ and $s$, such that
		\begin{equation}\label{eq:H-SVE lift initial dependence}
			\E\left[ \sup_{t \in [s,T]}\| X(t; s, \xi) - X(t; s, \widetilde{\xi})\|_{\mathcal{V}}^p \right] \leq C_T \E\left[ \| \xi - \widetilde{\xi} \|_{\mathcal{V}}^p \right].
		\end{equation}
	\end{theorem}

    A detailed proof is given in the appendix. Using stronger conditions on the semigroup and additional conditions on the initial condition, we also obtain H\"older continuous sample paths as stated in Proposition \ref{prop: Hoelder}. As usual for unique solutions of (differential) stochastic equations (see e.g.~\cite[Theorem 9.14]{DaPrato_Zabczyk_2014}), also the process given by \eqref{eq: abstract Markov lift time inhomogeneous} determines a (time-inhomogeneous) Markov process. Denote by $B_b(\mathcal{V}), C_b(\mathcal{V})$ the Banach space of bounded (respectively continuous and bounded) functions $f\colon \mathcal{V} \longrightarrow \R$.

	\begin{corollary}\label{cor: Markov}
		Suppose that Assumptions \ref{assumption SEE} and \ref{assumption Lipschitz} are satisfied. Then \eqref{eq: abstract Markov lift time inhomogeneous} determines a time-inhomogeneous Markov process with transition family $P_{s,t}\colon B_b(\mathcal{V}) \longrightarrow B_b(\mathcal{V})$
	      \begin{equation*}\label{eq: transition semigroup}
    		P_{s,t} f(\eta) \coloneqq \E[f(X(t; s,\eta))], \qquad  0 \leq s \leq t \leq T,\ \eta\in\mathcal{V}.
    	\end{equation*}
        This transition family is $C_b$-Feller in the sense that $P_{s,t}$ leaves $C_b(\mathcal{V})$ invariant.
	\end{corollary}

    The proof of this statement is postponed to the appendix. If the coefficients $\xi_b, \xi_{\sigma}$ appearing in \eqref{eq: abstract Markov lift time inhomogeneous} are time-homogeneous, then Assumption \ref{assumption Lipschitz} holds for each $T > 0$. In particular, for each $\xi$ there exists a unique global solution $X(\cdot; s, \xi) \in L^p(\Omega,\P; C([s, \infty); \mathcal{V}))$ which forms a time-homogeneous Markov process.

    \begin{corollary}\label{cor: time homogeneous Markov}
        Suppose that Assumptions \ref{assumption SEE} and \ref{assumption Lipschitz} are satisfied, and assume that $\xi_b, \xi_{\sigma}$ are time-homogeneous. Then \eqref{eq: abstract Markov lift time inhomogeneous} determines a time-homogeneous Markov process with $C_b(\mathcal{V})$-Feller transition semigroup $(P_t)_{t \geq 0}$ given by $P_tf(\eta) = P_{0,t}f(\eta)$.
    \end{corollary}
    \begin{proof}
        Note that the unique solution of \eqref{eq: abstract Markov lift time inhomogeneous} satisfies
		\begin{multline*}
			X(t+s; t, \eta) =  S(s)\eta + \int_{0}^s S(s-r)\xi_bb(\Xi X(r+t; s, \eta))\,\d r\\ + \int_{0}^s S(s-r)\xi_\sigma\sigma(\Xi X(r+t; s, \eta))\,\d W^t_r
		\end{multline*}
		where $W^t_r=W_{t+r}-W_t$ denotes the restarted Wiener process with respect to the shifted filtration $(\F_{t+r})_{r\geq0}$. Consequently, by Theorem \ref{label:existence_uniqueness_VSPDE} and \cite[Theorem 2]{MartinOndreját2004}, $X(t+s; t, \eta)$ and $X(s; 0, \eta)$ have the same law. This shows that $P_{0,s}f(\eta) = P_{s, s+t}f(\eta)$ holds for all $t,s \geq 0$ and $\eta \in \mathcal{V}$. In particular, \eqref{eq: inhom. markov} yields the time-homogeneous Markov property.
    \end{proof}

    As usual for time-homogeneous Markov processes, all distributional properties are already captured by $X(t;\xi) \coloneqq X(t; 0,\xi)$ and hence the associated transition semigroup $(P_t)_{t \geq 0}$.


	\section{Limit distributions and invariant measures}\label{sec:limit_dist_inv_meas}

	In this section, we study the long-time behaviour of the Markovian lift with time-homogeneous coefficients under the additional structural assumption
    \begin{align}\label{eq: timehomogeneous}
        \xi_b(t,u) = \xi_b\, b(u) \ \text{ and } \ \xi_{\sigma}(t,u) = \xi_{\sigma}\, \sigma(u)
    \end{align}
    where $b$ denotes the drift and $\sigma$ the diffusion operator. Moreover, we suppose that the semigroup $(S(t))_{t \geq 0}$ has additional regularisation properties for large time $t \gg 1$ as stated below.

	\begin{assumptionp}{C}\label{assumption long-time}
		The following conditions hold:
		\begin{enumerate}
			\item[(a)] There exist separable Hilbert spaces $H_b, H_{\sigma}, V$ satisfying \eqref{eq: inclusion} such that $\xi_b \in L(H_b, \mathcal{H})$ and $\xi_{\sigma} \in L_q(H_{\sigma}, \mathcal{H})$, where $\frac{1}{q} + \frac{1}{q'} = \frac{1}{2}$. Moreover, there exist Lipschitz constants $C_{b,\mathrm{lip}}, C_{\sigma, \mathrm{lip}} \geq 0$ such that for all $u,v \in H$
            \begin{align*}
                \| b(u) - b(v)\|_{H_b} &\leq C_{b,\mathrm{lip}}\|u - v\|_H,
                \\ \| \sigma(u) - \sigma(v)\|_{L_{q'}(U, H_{\sigma})} &\leq C_{\sigma, \mathrm{lip}}\|u-v\|_H.
            \end{align*}

            \item[(b)] There exists a projection operator $S_\infty \in L(\mathcal{V})$ with $S(t) \longrightarrow S_\infty$ strongly as $t \to \infty$. This projection operator satisfies for each $t > 0$
            \[
                S_\infty S(t)\xi_b = 0 \ \text{ and } \ S_\infty S(t) \xi_{\sigma} = 0.
            \]
            Finally, the semigroup satisfies the integrability condition
            \begin{align}\label{eq:assumption long-timme - integrability}
                \int_0^{\infty} \left( \|S(t)\xi_b \|_{L(H_b, \mathcal{V})} + \|S(t)\xi_{\sigma}\|_{L_q(H_{\sigma}, \mathcal{V})}^2 \right)\, \mathrm{d}t < \infty.
            \end{align}

			\item[(c)] There exists a separable Hilbert space $\mathcal{V}_0$ such that $\mathcal{V} \hookrightarrow \mathcal{V}_0$ dense and it holds that $\mathcal{V} = \{ y \in \mathcal{V}_0 \ : \ \|y \|_{\mathcal{V}} < \infty \}$. Furthermore, there are constants $\lambda, C(\mathcal{V}, \mathcal{V}_0) > 0$ such that
			\begin{equation}\label{eq: rate S to P}
			\| S(t) - S_\infty \|_{L(\mathcal{V}, \mathcal{V}_0)} \leq C(\mathcal{V}, \mathcal{V}_0) (1 \lor t)^{-\lambda}, \qquad t > 0.
			\end{equation}
            The operator $\Xi\colon \mathcal{V} \longrightarrow V$ admits a unique continuous extension $\Xi \in L(\mathcal{V}_0, V)$.
		\end{enumerate}
	\end{assumptionp}

    Condition (a) is a slight modification of Assumption \ref{assumption Lipschitz} for time-homogeneous coefficients under the structural condition \eqref{eq: timehomogeneous}. Condition (b) replaces the dissipativity conditions from \cite[Section 6.3]{MR1417491} and \cite{MR4241464} in terms of the integrability condition \eqref{eq:assumption long-timme - integrability}. For Markov models it was shown in \cite{MR4241464} that multiple limit distributions appear whenever the semigroup $(S(t))_{t \geq 0}$ is not exponentially stable on the full space, but is instead \textit{exponentially ergodic} with limit $S_\infty \coloneqq \lim_{t \to \infty}S(t)$. Thus, in this section, we extend \cite{MR4241464} towards polynomial rates of convergence. Markovian lifts of stochastic Volterra processes constitute an interesting class of Markov processes with multiple limit distributions characterised by the range of $S_\infty$ that falls into this class of processes.

    Let us remark that the convergence rate for $S(t)y \longrightarrow S_\infty y$ depends on the choice of $y \in \mathcal{V}$. Thus, to obtain a rate of convergence uniformly in all $y \in \mathcal{V}$, in condition (c) we introduce the larger space $\mathcal{V} \hookrightarrow \mathcal{V}_0$ with a \textit{polynomial rate of convergence} determined by \eqref{eq: rate S to P} which is a characteristic feature for many stochastic Volterra equations.

    Finally, let $C_{b,\mathrm{lin}},C_{\sigma,\mathrm{lin}}>0$ be the linear growth constants given by
	\[
	C_{b,\mathrm{lin}} \coloneqq \sup_{u\in H}\frac{\norm{b(u)}_{H_b}}{1+\norm{u}_H} \ \text{ and } \ C_{\sigma,\mathrm{lin}}\coloneqq \sup_{u\in H}\frac{\norm{\sigma(u)}_{L_{q'}(U, H_{\sigma})}}{1+\norm{u}_H}.
	\]
    Note that these constants are finite due to Assumption \ref{assumption long-time}.(a). Moreover, by Assumption \ref{assumption long-time}.(b), $S(t) \longrightarrow S_{\infty}$ strongly on $\mathcal{V}$, and hence the uniform boundedness principle gives $\sup_{t \geq 0}\|S(t)\|_{L(\mathcal{V})} < \infty$. Finally, under Assumptions \ref{assumption SEE} and \ref{assumption long-time}, there exists a unique time-homogeneous Markov process $X(t;\xi) \coloneqq X(t;0,\xi)$ obtained from \eqref{eq: abstract Markov lift time inhomogeneous} with $s = 0$, see Corollary \ref{cor: time homogeneous Markov}.

	\subsection{Uniform contraction estimates}

	In this section, we prove uniform bounds on the $L^p$-norm of the unique solution of \eqref{eq: abstract mild formulation Markovian lift} and subsequently derive a contraction estimate in the spirit of \eqref{eq:H-SVE lift initial dependence} but with a constant that decays polynomially in time. Below, we start with the general case, and later on, show how the result can be strengthened for the case of additive noise or when the drift $b$ vanishes.

	For this purpose, let us define a constant $c_p$ by $c_2 = 1$, and
	\begin{equation}\label{eq: constant BDG}
	c_p=\left(\frac{p(p-1)}{2}\right)^p\left(\frac{p}{p-1}\right)^{\frac{p^2}{2}}, \qquad p \in (2,\infty).
	\end{equation}
	Note that this constant appears in the BDG-inequality for the stochastic integral against the cylindrical Wiener process $W$, see e.g.~\cite[Section 4.6]{DaPrato_Zabczyk_2014}. Let us define
	\begin{multline*}
		\rho_{\mathrm{gen}}(t) = 3^{p-1}\norm{\Xi}_{L(\mathcal{V}_0,V)}^p \bigg(C_{b,\text{lip}}^p\norm{S(\cdot)\xi_b}_{L^1(\R_+;L(H_b,\mathcal{V}_0))}^{p-1}\norm{S(t)\xi_b}_{L(H_b,\mathcal{V}_0)}\\
		+ c_pC_{\sigma,\text{lip}}^p \norm{S(\cdot)\xi_\sigma}_{L^2(\R_+;L_q(H_{\sigma},\mathcal{V}_0))}^{p-2}\norm{S(t)\xi_\sigma}_{L_q(H_{\sigma},\mathcal{V}_0)}^2 \bigg).
	\end{multline*}
	Since the inclusion $\iota\colon \mathcal{V} \hookrightarrow \mathcal{V}_0$ is bounded, we get $\|S(t)\xi_b \|_{L(H_b, \mathcal{V}_0)} \leq \| \iota \|_{L(\mathcal{V}, \mathcal{V}_0)}  \|S(t)\xi_b \|_{L(H_b, \mathcal{V})}$ and $\|S(t)\xi_{\sigma}\|_{L_q(H_{\sigma}, \mathcal{V}_0)} \leq \| \iota \|_{L(\mathcal{V}, \mathcal{V}_0)}  \|S(t)\xi_{\sigma}\|_{L_q(H_{\sigma}, \mathcal{V})}$, which implies $\| S(\cdot)\xi_b\|_{L(H_b, \mathcal{V}_0)} \in L^1(\R_+)$ and $\| S(\cdot)\xi_{\sigma}\|_{L_q(H_{\sigma}, \mathcal{V}_0)} \in L^2(\R_+)$ by Assumption \ref{assumption long-time}. In particular, $\rho_{\mathrm{gen}} \in L^1(\R_+)$ is well-defined. Denote by $r_{\mathrm{gen}} \in L_{\loc}^1(\R_+)$ the unique solution of the Volterra convolution equation
	\begin{equation}\label{eq: r general}
		r_{\mathrm{gen}}(t) = \rho_{\mathrm{gen}}(t) + \int_0^t r_{\mathrm{gen}}(t-s)\rho_{\mathrm{gen}}(s)\, \mathrm{d}s.
	\end{equation}
    Remark that $\rho_{\mathrm{gen}} = \rho_{\mathrm{gen}}^{(p)}, r_{\mathrm{gen}} = r_{\mathrm{gen}}^{(p)}$ implicitly depend on $p > 2$. Let us define for $p > 2$
	\[
	\mathcal{R}^{p}_{\mathrm{gen}}(t) = (1 \vee t)^{- \lambda p} + \int_0^t r_{\mathrm{gen}}^{(p)}(t-s)(1 \vee s)^{- \lambda p} \, \mathrm{d}s.
	\]
    Then we obtain the following sufficient conditions for uniform boundedness of moments and global contraction estimates.

	\begin{lemma}\label{lemma: general case}
		Suppose that Assumptions \ref{assumption SEE} and \ref{assumption long-time} hold. Let $p\in(2,\infty)$ satisfy \eqref{eq: continuity}. If
		\begin{align}\label{eq:bounded_moment_condi_multi}
				6^{p-1}\norm{\Xi}_{L(\mathcal{V},V)}^p \left( C_{b,\mathrm{lin}}^p \| S(\cdot)\xi_b\|_{L^1(\R_+;L(H_b,\mathcal{V}))}^p + c_p C_{\sigma,\mathrm{lin}}^p \| S(\cdot)\xi_{\sigma}\|_{L^2(\R_+;L_q(H_\sigma,\mathcal{V}))}^{p} \right) < 1,
		\end{align}
		then for each $\xi\in L^p(\Omega, \F_0, \P;\mathcal{V})$
		\begin{align}\label{eq: uniform norm bound}
				\sup_{t\geq 0}\E\left[\norm{X(t; \xi)}_{\mathcal{V}}^p\right]\lesssim 1 + \E\left[\norm{\xi}_{\mathcal{V}}^p\right].
		\end{align}
		Likewise, if
		\begin{align}\label{eq:contraction_estimate_condi_multi}
				3^{p-1}\norm{\Xi}_{L(\mathcal{V}_0,V)}^p \left( C_{b,\mathrm{lip}}^p \| S(\cdot)\xi_b\|_{L^1(\R_+;L(H_b,\mathcal{V}))}^p + c_pC_{\sigma,\mathrm{lip}}^p \| S(\cdot)\xi_{\sigma}\|_{L^2(\R_+;L_q(H_\sigma,\mathcal{V}))}^{p} \right)<1,
		\end{align}
		then for all $\xi,\widetilde{\xi}\in L^p(\Omega, \F_0, \P;\mathcal{V})$ one has
		\begin{equation}\label{eq: uniform contraction estimate}
				\norm{X(t; \xi)-X(t; \widetilde{\xi})}_{L^p(\Omega;\mathcal{V}_0)}^p \lesssim \norm{\xi-\widetilde{\xi}}_{L^p(\Omega;\mathcal{V})}^p \mathcal{R}_{\mathrm{gen}}^{p}(t) + \|S_\infty\xi - S_\infty\widetilde{\xi}\|_{L^p(\Omega;\mathcal{V}_0)}^p.
		\end{equation}
	\end{lemma}
	\begin{proof}
		Let us first prove \eqref{eq: uniform norm bound} under assumption \eqref{eq:bounded_moment_condi_multi}. Using \eqref{eq: abstract Markov lift time inhomogeneous} we obtain for $t \geq 0$
		\begin{align*}
			\norm{X(t; \xi)}_{L^p(\Omega;\mathcal{V})}^p
			&\leq 3^{p-1}\norm{S(t)\xi}_{L^p(\Omega;\mathcal{V})}
			+ 3^{p-1}\E\left[\left(\int_0^t \norm*{S(t-s)\xi_b\, b(\Xi X(s; \xi))}_{\mathcal{V}}\d s\right)^p\right]
			\\ &\quad+ 3^{p-1}\E\left[ \left\| \int_0^t S(t- s)\xi_\sigma\, \sigma(\Xi X(s; \xi))\,\d W_s \right\|_{\mathcal{V}}^p\right].
		\end{align*}
		For the first term we obtain $\|S(t)\xi\|_{L^p(\Omega; \mathcal{V})} \leq (\sup_{t \geq 0}\|S(t)\|_{L(\mathcal{V})})\|\xi\|_{L^p(\Omega; \mathcal{V})}$. For the second term, we use Jensen's inequality twice to find that
		\begin{align*}
			\E&\left[\left(\int_0^t \norm*{S(t-s)\xi_b\, b(\Xi X(s; \xi))}_{\mathcal{V}}\d s\right)^p\right]
			\\ &\leq C_{b,\text{lin}}^p\E\left[\left(\int_0^t \norm{S(t - s)\xi_b}_{L(H_b,\mathcal{V})}\left(1+\norm{\Xi}_{L(\mathcal{V},V)}\norm{X(s; \xi)}_{\mathcal{V}}\right)\d s\right) ^p\right]
			\\ &\leq C_{b,\text{lin}}^p \left(\int_0^t \norm{S(t-s)\xi_b}_{L(H_b,\mathcal{V})}\,\d s\right) ^{p-1}\int_0^t \norm{S(t-s)\xi_b}_{L(H_b,\mathcal{V})}
			\\ &\hskip75mm \cdot \E\left[\left(1+\norm{\Xi}_{L(\mathcal{V},V)}\norm{X(s; \xi)}_{\mathcal{V}}\right)^p\right]\,\d s
			\\ &\leq 2^{p-1}C_{b,\text{lin}}^p \left(\int_0^\infty \norm{S(\tau)\xi_b}_{L(H_b,\mathcal{V})}\,\d \tau \right)^p
			\\ &\qquad+ 2^{p-1} C_{b,\text{lin}}^p\left( \int_0^\infty \norm{S(\tau)\xi_b}_{L(H_b,\mathcal{V})}\,\d \tau \right)^{p-1}
			\\ &\hskip45mm \cdot\norm{\Xi}_{L(\mathcal{V},V)}^p\int_0^t \norm{S(t-s)\xi_b}_{L(H_b,\mathcal{V})}\norm{X(s; \xi)}_{L^p(\Omega;\mathcal{V})}^p\,\d s.
		\end{align*}
		For the third term, we use the BDG inequality and Jensen's inequality to find that
		\begin{align*}
			\E&\left[ \left\| \int_0^t S(t-s)\xi_\sigma\, \sigma(\Xi X(s; \xi))\,\d W_s \right\|_{\mathcal{V}}^p\right]
			\\ &\leq c_p C_{\sigma,\text{lin}}^p  \E\left[\left(\int_0^t \norm{S(t-s)\xi_\sigma}_{L_q(H_\sigma,\mathcal{V})}^2\left(1+\norm{\Xi}_{L(\mathcal{V},V)}\norm{X(s; \xi)}_{\mathcal{V}}\right)^2\d s\right)^{p/2}\right]
			\\ &\leq c_pC_{\sigma,\text{lin}}^p  \left(\int_0^t \norm{S(t-s)\xi_\sigma}_{L_q(H_\sigma,\mathcal{V})}^2\,\d s  \right)^{p/2-1}
			\\ &\qquad\cdot\int_0^t \norm{S(t-s)\xi_\sigma}_{L_q(H_\sigma,\mathcal{V})}^2 \E\left[\left(1+\norm{\Xi}_{L(\mathcal{V},V)}\norm{X(s; \xi)}_{\mathcal{V}}\right)^p\right]\d s
			\\ &\leq 2^{p-1}c_pC_{\sigma,\text{lin}}^p \left(\int_0^\infty \norm{S(\tau)\xi_\sigma}_{L_q(H_\sigma,\mathcal{V})}^2\,\d \tau  \right)^{p/2}
			\\ &\qquad+ 2^{p-1}c_pC_{\sigma,\text{lin}}^p  \left(\int_0^\infty \norm{S(\tau)\xi_\sigma}_{L_q(H_\sigma,\mathcal{V})}^2\,\d \tau  \right)^{p/2-1}
			\\ &\hskip45mm \cdot\norm{\Xi}_{L(\mathcal{V},V)}^p \int_0^t \norm{S(t-s)\xi_\sigma}_{L_q(H_\sigma,\mathcal{V})}^2\norm{X(s; \xi)}_{L^p(\Omega;\mathcal{V})}^p\,\d s.
		\end{align*}
		Hence, we arrive at the inequality
		\begin{align*}
			\norm{X(t; \xi)}_{L^p(\Omega;\mathcal{V})}^p
			\leq 3^{p-1}\sup_{t \geq 0}\|S(t)\|_{L(\mathcal{V})}\norm{\xi}_{L^p(\Omega;\mathcal{V})}^p + A + \int_0^t \rho_{\text{lin}}(t-s)\norm{X(s; \xi)}_{L^p(\Omega;\mathcal{V})}^p\,\d s
		\end{align*}
		where we have set
		\begin{align*}
			A &= 6^{p-1}C_{b,\text{lin}}^p \norm{S(\cdot)\xi_b}_{L^1(\R_+;L(H_b,\mathcal{V}))}^p + 6^{p-1}c_pC_{\sigma,\text{lin}}^p  \norm{S(\cdot)\xi_\sigma}_{L^2(\R_+;L_q(H_\sigma,\mathcal{V}))}^p,
			\\ \rho_{\text{lin}}(t) &= 6^{p-1}C_{b,\text{lin}}^p \norm{S(\cdot)\xi_b}_{L^1(\R_+;L(H_b,\mathcal{V}))}^{p-1}\norm{\Xi}_{L(\mathcal{V},V)}^p \norm{S(t)\xi_b}_{L(H_b,\mathcal{V})}
			\\ &\quad+ 6^{p-1}c_p C_{\sigma,\text{lin}}^p \norm{S(\cdot)\xi_\sigma}_{L^2(\R_+;L_q(H_\sigma,\mathcal{V}))}^{p-2}\norm{\Xi}_{L(\mathcal{V},V)}^p \norm{S(t)\xi_\sigma}_{L_q(H_\sigma,\mathcal{V})}^2.
		\end{align*}
		Note that $A$ is finite due to Assumption \ref{assumption long-time}. Let $r_{\mathrm{lin}} \in L_{\loc}^1(\R_+)$ be the unique nonnegative solution of $r_{\mathrm{lin}} = \rho_{\text{lin}} + \rho_{\text{lin}} \ast r_{\mathrm{lin}}$. Since $\int_{0}^{\infty}\rho_{\text{lin}}(t)\, \mathrm{d}t < 1$ by assumption \eqref{eq:bounded_moment_condi_multi}, the Paley-Wiener theorem implies that $r_{\mathrm{lin}} \in L^1(\R_+)$. An application of the Volterra Gronwall inequality (see e.g.~\cite[Lemma A.1]{BBF23}) yields
		\[
		\norm{X(t; \xi)}_{L^p(\Omega;\mathcal{V})}^p \leq \left(3^{p-1}\sup_{t \geq 0}\|S(t)\|_{L(\mathcal{V})}\norm{\xi}_{L^p(\Omega;\mathcal{V})}^p + A \right)\left(1 +  \int_0^{\infty}r_{\mathrm{lin}}(\tau)\, \mathrm{d}\tau\right) < \infty
		\]
        for $t \geq 0$. This proves the desired uniform moment bound.

		Next, we prove \eqref{eq: uniform contraction estimate} under the assumption \eqref{eq:contraction_estimate_condi_multi}. Here, using the BDG-inequality and the Lipschitz continuity of $b,\sigma$, we find
		\begin{align*}
			&\norm{X(t; \xi)-X(t; \widetilde{\xi})}_{L^p(\Omega;\mathcal{V}_0)}^p
			\\ &\quad\leq 3^{p-1}\norm{S(t)(\xi-\widetilde{\xi})}_{L^p(\Omega;\mathcal{V}_0)}^p
			\\ &\qquad + 3^{p-1}C^p_{b,\text{lip}}\norm{\Xi}^p_{L(\mathcal{V}_0,V)}\E\left[\left(\int_0^t \norm{S(t-s)\xi_b}_{L(H_b,\mathcal{V}_0)}\norm{X(s; \xi)-X(s; \widetilde{\xi})}_{\mathcal{V}_0}\,\d s\right)^p\right]
			\\ &\qquad +3^{p-1}c_pC^p_{\sigma,\text{lip}}\norm{\Xi}^p_{L(\mathcal{V}_0,V)} \E\left[\left(\int_0^t \norm{S(t-s)\xi_\sigma}^2_{L_q(H_\sigma,\mathcal{V}_0)}\norm{X(s; \xi) - X(s; \widetilde{\xi})}_{\mathcal{V}_0}^2\,\d s\right)^{p/2}\right].
		\end{align*}
		Using Jensen's inequality and performing a similar substitution to the one above, we obtain
		\begin{align*}
			&\ \norm{X(t; \xi) - X(t; \widetilde{\xi})}_{L^p(\Omega;\mathcal{V}_0)}^p
			\\ &\qquad \leq 3^{p-1}\norm{S(t)(\xi-\widetilde{\xi})}_{L^p(\Omega;\mathcal{V}_0)}^p
			+ \int_0^t \rho_{\mathrm{gen}}(t - s)\norm{X(s; \xi) - X(s; \widetilde{\xi})}_{L^p(\Omega;\mathcal{V}_0)}^p\,\d s.
		\end{align*}
		Moreover, since $\norm{\rho_{\mathrm{gen}}}_{L^1(\R_+)}<1$ by assumption \eqref{eq:contraction_estimate_condi_multi}, we obtain $r_{\mathrm{gen}} \in L^1(\R_+)$. An application of the Volterra version of the Gronwall lemma yields
		\begin{align*}
			&\norm{X(t; \xi)-X(t; \widetilde{\xi})}_{L^p(\Omega;\mathcal{V}_0)}^p
            \\ &\quad\leq 3^{p-1}\norm{S(t)(\xi-\widetilde{\xi})}_{L^p(\Omega;\mathcal{V}_0)}^p
            + 3^{p-1}\int_0^t r_{\mathrm{gen}}(t - \tau)\norm{S(\tau)(\xi-\widetilde{\xi})}_{L^p(\Omega;\mathcal{V}_0)}^p\,\d \tau
			\\ &\quad\leq 6^{p-1} C(\mathcal{V}, \mathcal{V}_0) \norm{\xi-\widetilde{\xi}}_{L^p(\Omega;\mathcal{V})}^p \left( (1\lor t)^{-\lambda p} + \int_0^t r_{\mathrm{gen}}(t-\tau)(1\lor \tau)^{-\lambda p} \,\d \tau \right)
			\\ &\qquad + 6^{p-1}\| S_\infty\xi - S_\infty\widetilde{\xi}\|^p_{\mathcal{V}_0}\left( 1 + \int_0^{\infty} r_{\mathrm{gen}}(\tau)\, \mathrm{d}\tau \right)
		\end{align*}
		where we have used
		\begin{align*}
			\| S(\tau)(\xi - \widetilde{\xi})\|_{\mathcal{V}_0}
			&\leq \| (S(\tau) - S_\infty)(\xi - \widetilde{\xi})\|_{\mathcal{V}_0} + \| S_\infty\xi - S_\infty\widetilde{\xi}\|_{\mathcal{V}_0}
			\\ &\lesssim (1\lor \tau)^{-\lambda} \| \xi - \widetilde{\xi}\|_{\mathcal{V}} + \| S_\infty\xi - S_\infty\widetilde{\xi}\|_{\mathcal{V}_0}.
		\end{align*}
		This proves the assertion.
	\end{proof}
	Remark that, in contrast to dissipative systems with a unique invariant measure, here an additive term $\|S_\infty\xi - S_\infty\widetilde{\xi}\|_{\mathcal{V}_0}^p$ is present, which characterises the occurrence of multiple limit distributions.

    For the case $b \equiv 0$, all estimates can be improved since then we may choose $\xi_b = 0$ and hence \eqref{eq:assumption long-timme - integrability} reduces to an integral solely against $S(t)\xi_{\sigma}$. For the precise statement, let us define
    \begin{align}\label{eq: rho b zero}
		          \rho_{\mathrm{b=0}}(t) = 2^{p-1}\norm{\Xi}_{L(\mathcal{V}_0, V)}^p c_pC_{\sigma,\text{lip}}^p \norm{S(\cdot)\xi_\sigma}_{L^2(\R_+;L_q(H_{\sigma},\mathcal{V}_0))}^{p-2}\norm{S(t)\xi_\sigma}_{L_q(H_{\sigma},\mathcal{V}_0)}^2,
	\end{align}
    and let $r_{\mathrm{b=0}} = r_{\mathrm{b=0}}^{(p)}$ be the unique solution of \eqref{eq: r general} with $\rho_{\mathrm{gen}} = \rho_{\mathrm{gen}}^{(p)}$ replaced by $\rho_{\mathrm{b=0}} = \rho_{\mathrm{b=0}}^{(p)}$, i.e.~$r_{\mathrm{b=0}}(t) = \rho_{\mathrm{b=0}}(t) + \int_0^t r_{\mathrm{b=0}}(t-s)\rho_{\mathrm{b=0}}(s)\, \mathrm{d}s$. Finally, define
    \[
        \mathcal{R}_{\mathrm{b=0}}^{p}(t) = (1\vee t)^{- \lambda p} + \int_0^t r_{\mathrm{b=0}}^{(p)}(t-s)(1\vee s)^{- \lambda p}\, \mathrm{d}s.
    \]
    Then we obtain the following uniform contraction estimate.
    \begin{lemma}\label{lemma: general case b zero}
		Suppose that Assumptions \ref{assumption SEE} and \ref{assumption long-time} hold with $b = 0$ and $\xi_b = 0$. Let $p\in(2,\infty)$ satisfy \eqref{eq: continuity}. If
		\begin{align}\label{eq:bounded_moment_condi_multi b zero}
				4^{1 - \frac{1}{p}} C_{\sigma, \mathrm{lin}} \norm{\Xi}_{L(\mathcal{V},V)} \left( \int_0^{\infty} \| S(t)\xi_\sigma \|_{L_q(H_\sigma,\mathcal{V})}^2\, \mathrm{d}t \right)^{\frac{1}{2}} c_p^{\frac{1}{p}} < 1,
		\end{align}
		then \eqref{eq: uniform norm bound} holds for each $\xi\in L^p(\Omega, \F_0, \P;\mathcal{V})$. Likewise, if
		\begin{align}\label{eq:contraction_estimate_condi_multi b zero}
				2^{1-\frac{1}{p}}C_{\sigma,\mathrm{lip}} \norm{\Xi}_{L(\mathcal{V}_0,V)} \left( \int_0^{\infty} \| S(t)\xi_\sigma\|_{L_q(H_{\sigma},\mathcal{V}_0))}^{2}\, \mathrm{d}t \right)^{\frac{1}{2}} c_p^{\frac{1}{p}} < 1,
		\end{align}
		 then \eqref{eq: uniform contraction estimate} holds for all $\xi,\widetilde{\xi}\in L^p(\Omega, \F_0, \P;\mathcal{V})$ with $\mathcal{R}_{\mathrm{gen}}^{p}$ is replaced by $\mathcal{R}_{\mathrm{b=0}}^{p}$.
	\end{lemma}
	\begin{proof}
		For the first assertion, we use again \eqref{eq: abstract mild formulation Markovian lift} and argue as in the proof of Lemma \ref{lemma: general case} to obtain for $t \geq 0$
		\begin{align*}
			\norm{X(t; \xi)}_{L^p(\Omega;\mathcal{V})}^p
			&\leq 2^{p-1}\norm{S(t)\xi}_{L^p(\Omega;\mathcal{V})} ^p
			 + 2^{p-1} \E\left[ \left\| \int_0^t S(t-s)\xi_\sigma\, \sigma(\Xi X(s; \xi))\,\d W_{s} \right\|_{\mathcal{V}}^p\right]
            \\ &\leq 2^{p-1}\sup_{t \geq 0}\|S(t)\|_{L(\mathcal{V})}\norm{\xi}_{L^p(\Omega;\mathcal{V})}^p + A + \int_0^t \rho_{\text{lin}}(t-s)\norm{X(s; \xi)}_{L^p(\Omega;\mathcal{V})}^p\,\d s
		\end{align*}
		where we have set $A = 4^{p-1}c_pC_{\sigma,\text{lin}}^p  \norm{S(\cdot)\xi_\sigma}_{L^2(\R_+;L_q(H_\sigma,\mathcal{V}))}^p$ and
		\begin{align*}
		      \rho_{\text{lin}}(t) = 4^{p-1}c_p C_{\sigma,\mathrm{lin}}^p \norm{S(\cdot)\xi_\sigma}_{L^2(\R_+;L_q(H_\sigma,\mathcal{V}))}^{p-2}\norm{\Xi}_{L(\mathcal{V}, V)}^p \norm{S(t)\xi_\sigma}_{L_q(H_\sigma,\mathcal{V})}^2.
		\end{align*}
		Note that $A$ is finite due to Assumption \ref{assumption long-time}. The assertion can now be deduced as in the proof of Lemma \ref{lemma: general case}. For the second assertion, we use the BDG-inequality and the Lipschitz continuity of $b,\sigma$ to find
		\begin{align*}
			&\norm{X(t; \xi) - X(t; \widetilde{\xi})}_{L^p(\Omega;\mathcal{V}_0)}^p
			\\ &\quad\leq 2^{p-1}\norm{S(t)(\xi-\widetilde{\xi})}_{L^p(\Omega;\mathcal{V}_0)}^p
			\\ &\qquad + 2^{p-1}c_pC^p_{\sigma,\mathrm{lip}}\norm{\Xi}^p_{L(\mathcal{V}_0,V)} \E\left[\left(\int_0^t \norm{S(t-s)\xi_\sigma}^2_{L_q(H_\sigma,\mathcal{V}_0)}\norm{X(s; \xi) - X(s; \widetilde{\xi})}_{\mathcal{V}_0}^2\,\d s\right)^{p/2}\right]
            \\ &\quad\leq 2^{p-1}\norm{S(t)(\xi-\widetilde{\xi})}_{L^p(\Omega;\mathcal{V}_0)}^p + \int_0^t \rho_{\mathrm{gen}}(t-s)\norm{X(s; \xi) - X(s; \widetilde{\xi})}_{L^p(\Omega;\mathcal{V}_0)}^p\,\d s.
		\end{align*}
        Arguing similarly to the proof of Lemma \ref{lemma: general case} proves the assertion.
	\end{proof}

	Finally, let us outline how, for additive noise where $\sigma \equiv \sigma_0$ is constant, Lemma \ref{lemma: general case} can be strengthened with respect to conditions \eqref{eq:bounded_moment_condi_multi} and \eqref{eq:contraction_estimate_condi_multi}. In this case, we define
	\begin{align*}
		\rho_{\text{add}}(t) = C_{b,\text{lip}} \norm{\Xi}_{L(\mathcal{V}_0, V)}\|S(t)\xi_b\|_{L(H_b,\mathcal{V}_0)},
	\end{align*}
	and let $r_{\mathrm{add}}$ be given by \eqref{eq: r general} with $\rho_{\mathrm{gen}}$ replaced by $\rho_{\mathrm{add}}$. Finally, let us define
	\[
	\mathcal{R}_{\mathrm{add}}(t) = (1 \vee t)^{- \lambda} + \int_0^t r_{\mathrm{add}}(t-s)(1 \vee s)^{- \lambda} \, \mathrm{d}s.
	\]
	Then we obtain the following analogue for the case of additive noise.

	\begin{lemma}\label{lemma contraction estimate}
		Suppose that Assumptions \ref{assumption SEE} and \ref{assumption long-time} are satisfied, and that $\sigma = \sigma_0 \in L_{q'}(U,H_{\sigma})$ does not depend on $u \in H$. Fix $p \in (2,\infty)$ such that \eqref{eq: continuity} holds. If
		\[
		      C_{b,\mathrm{lin}}\norm{\Xi}_{L(\mathcal{V},V)} \int_0^\infty \norm{S(t)\xi_b}_{L(H_b, \mathcal{V})}\,\d t < 1,
		\]
		then \eqref{eq: uniform norm bound} holds for each $\xi\in L^p(\Omega, \F_0, \P;\mathcal{V})$. Likewise, if
		\begin{align}\label{eq: contraction estimate additive}
			C_{b,\mathrm{lip}}\norm{\Xi}_{L(\mathcal{V}_0,V)} \int_0^\infty \norm{S(t)\xi_b}_{L(H_b, \mathcal{V}_0)}\,\d t < 1,
		\end{align}
        then for all $\xi,\widetilde{\xi}\in L^p(\Omega, \F_0, \P;\mathcal{V})$ we get
        \[
            \norm{X(t; \xi)-X(t; \widetilde{\xi})}_{L^p(\Omega;\mathcal{V}_0)} \lesssim \norm{\xi-\widetilde{\xi}}_{L^p(\Omega;\mathcal{V})} \mathcal{R}_{\mathrm{add}}(t) + \|S_\infty\xi - S_\infty\widetilde{\xi}\|_{L^p(\Omega;\mathcal{V}_0)}.
        \]
	\end{lemma}
	\begin{proof}
		  For the first assertion, we argue similarly to the proof of Lemma \ref{lemma: general case}, which gives
		\begin{multline*}
			\norm{X(t; \xi)}_{L^p(\Omega;\mathcal{V})}\leq \norm{S(t)\xi}_{L^p(\Omega;\mathcal{V})} + \int_0^t \norm{S(t-s)\xi_b\, b(\Xi X(s; \xi))}_{L^p(\Omega;\mathcal{V})}\,\d s
            \\ + c_p^{1/p}\left(\int_0^t \norm{S(t-s)\xi_\sigma\sigma_0}_{L_2(U,\mathcal{V})}^2\,\d s\right)^{1/2}
		\end{multline*}
		and, consequently,
		\[
		\norm{X(t; \xi)}_{L^p(\Omega;\mathcal{V})} \lesssim A + \|\xi\|_{L^p(\Omega; \mathcal{V})} + \int_0^t \rho_{\mathrm{lin}}(t-s)\norm{X(s;\xi)}_{L^p(\Omega;\mathcal{V})}\, \mathrm{d}s
		\]
		with $A = C_{b,\mathrm{lin}}\norm{S(\cdot)\xi_b}_{L^1(\R_+;L(H_b,\mathcal{V}))} + c_p^{1/p} C_{\sigma, \mathrm{lin}}\norm{S(\cdot)\xi_\sigma}_{L^2(\R_+;L_q(H_\sigma,\mathcal{V}))}$ and
        \[
            \rho_{\text{lin}}(t)=C_{b,\text{lin}}\norm{\Xi}_{L(\mathcal{V}_0,V)}\norm{S(t)\xi_b}_{L(H_b,\mathcal{V})}.
        \]
        Since again $\int_0^{\infty} \rho_{\mathrm{lin}}(t)\, \mathrm{d}t < 1$, we can argue as in Lemma \ref{lemma: general case}. For the second assertion, we obtain
		\begin{align*}
			&\norm{X(t; \xi)-X(t; \widetilde{\xi})}_{L^p(\Omega;\mathcal{V}_0)}
            \\ &\quad\leq \norm{S(t)(\xi-\widetilde{\xi})}_{L^p(\Omega;\mathcal{V}_0)} + \int_0^t \norm{S(t-s)\xi_b\, (b(\Xi X(s; \xi)) - b(\Xi X(s; \widetilde{\xi}))}_{L^p(\Omega;\mathcal{V}_0)}\,\d s
            \\ &\quad\lesssim \norm{S(t)(\xi-\widetilde{\xi})}_{L^p(\Omega;\mathcal{V}_0)} + \int_0^t \rho_{\text{add}}(t-s)\norm{X(s; \xi) - X(s; \widetilde{\xi})}_{L^p(\Omega;\mathcal{V}_0)}\,\d s.
		\end{align*}
		By assumption $\int_0^\infty\rho_{\text{add}}(t)\,\d t <\infty $ and we can argue as in Lemma \ref{lemma: general case}.
	\end{proof}

	In the next section, we prove that the functions $\mathcal{R}_{\mathrm{gen}}^{\varepsilon}, \mathcal{R}_{\mathrm{b=0}}^{\varepsilon}, \mathcal{R}_{\mathrm{add}}^{\varepsilon}$ provide an estimate on the rate of convergence towards the limiting distribution. From this perspective, the next lemma provides an explicit pointwise bound for such a convergence rate.

	\begin{lemma}\label{lemma: convergence rate bound}
		Suppose that assumptions \ref{assumption SEE} and \ref{assumption long-time} are satisfied. Then for each $\kappa \in (0,1)$ there exist a constant $C_{\kappa} > 0$ such that
		\begin{align*}
			\mathcal{R}_{\mathrm{gen}}^{p}(t) &\lesssim (1\vee t)^{-\lambda p} + (1\lor t)^{-\log\left(1/\norm{\rho_{\mathrm{gen}}^{(p)}}_{L^1(\R_+)}\right)} + C_{\kappa} (1\lor t)^{- \lambda(1 - \kappa)},
            \\ \mathcal{R}_{\mathrm{b=0}}^{p}(t) &\lesssim (1\vee t)^{-\lambda p} + (1\lor t)^{-\log\left(1/\norm{\rho_{\mathrm{b=0}}^{(p)}}_{L^1(\R_+)}\right)} + C_{\kappa}(1\lor t)^{-2\lambda (1-\kappa)},
            \\ \mathcal{R}_{\mathrm{add}}(t) &\lesssim (1\vee t)^{-\lambda} + (1\lor t)^{-\log\left(1/\norm{\rho_{\mathrm{add}}}_{L^1(\R_+)}\right)} + C_{\kappa} (1\lor t)^{-\lambda(1 - \kappa)}.
		\end{align*}
	\end{lemma}
	\begin{proof}
		Let $r \in \{ r_{\mathrm{gen}}, r_{\mathrm{b=0}}, r_{\mathrm{add}}\}$, $\varepsilon = \lambda p$ for the first two cases, and $\varepsilon = \lambda$ for the additive case. Then we obtain
		\begin{align*}
			\int_0^t r(t-s)(1 \vee s)^{-\varepsilon}\,\d s
			&= \int_0^t r(s)(1 \vee (t-s))^{-\varepsilon}\,\d s
			\\ &= \int_0^{t/2} r(s)(1 \vee (t-s))^{-\varepsilon}\,\d s + \int_{t/2}^t r(s)(1 \vee (t-s))^{-\varepsilon}\, \d s
			\\ &\leq (1 \vee (t/2))^{-\varepsilon}\left(\int_0^\infty r(s)\,\d s\right) +\int_{t/2}^\infty r(s)\,\d s.
		\end{align*}
        Hence an application of Lemma \ref{lemma convergence rate tail r} yields for $\kappa \in (0,1)$
        \begin{align*}
			\mathcal{R}^{\varepsilon}(t) &\lesssim (1 \vee t)^{-\varepsilon} + \int_{t/2}^{\infty}r(s)\, \d s
			\\ &\lesssim (1 \vee t)^{-\varepsilon} + t^{-\log\left(1/\norm{\rho}_{L^1(\R_+)}\right)} + \int_{\kappa t^{1+\log(1-\kappa)}}^\infty \rho(s)\,\d s.
		\end{align*}
        To estimate the last integral, let us first note that
        \begin{align*}
            \|S(t)\xi_b \|_{L(H_b, \mathcal{V}_0)}
            &= \|(S(t/2) - S_\infty)S(t/2)\xi_b\|_{L(H_b, \mathcal{V}_0)}
            \\ &\leq C(\mathcal{V},\mathcal{V}_0)(1 \vee (t/2))^{-\lambda}\|S(t/2)\xi_b\|_{L(H_b, \mathcal{V})}
        \end{align*}
		and similarly
        \begin{align}\label{eq: S bound}
            \|S(t)\xi_{\sigma}\|_{L_q(H_{\sigma}, \mathcal{V}_0)} \leq C(\mathcal{V}, \mathcal{V}_0)(1 \vee (t/2))^{-\lambda}\|S(t/2)\xi_{\sigma}\|_{L_q(H_{\sigma}, \mathcal{V})}.
        \end{align}
        Hence we obtain for $\rho \in \{\rho_{\mathrm{gen}}, \rho_{\mathrm{add}}\}$
        \begin{align*}
            \int_{\kappa t^{1+ \log(1-\kappa)}}^\infty \rho(s)\, \mathrm{d}s
            &\lesssim \kappa^{-\lambda} t^{-\lambda(1 + \log(1-\kappa))} \int_{0}^{\infty} \left(\|S(s)\xi_b\|_{L(H_b, \mathcal{V})} + \|S(s)\xi_{\sigma}\|_{L_q(H_{\sigma}, \mathcal{V})}^2 \right)\, \mathrm{d}s.
        \end{align*}
        For $\rho = \rho_{\mathrm{b=0}}$, we obtain from \eqref{eq: rho b zero} combined with \eqref{eq: S bound}
        \[
            \int_{\kappa t^{1+ \log(1-\kappa)}}^\infty \rho_{\mathrm{b=0}}(s)\, \mathrm{d}s
            \lesssim \kappa^{-2\lambda} t^{-2\lambda(1 + \log(1-\kappa))} \int_{0}^{\infty} \|S(s)\xi_{\sigma}\|_{L_q(H_{\sigma}, \mathcal{V})}^2 \, \mathrm{d}s.
        \]
        Since $1 + \log(1-\kappa) \in (0,1)$, the assertion is proved.
	\end{proof}

    Finally, let us remark that similar contraction estimates also hold for the case where $\mathcal{V}_0 = \mathcal{V}$. More precisely, assuming that either \eqref{eq:contraction_estimate_condi_multi} or \eqref{eq:contraction_estimate_condi_multi b zero} with $b = 0$ holds for $\mathcal{V} = \mathcal{V}_0$, we obtain
    \[
        \| X(t;\xi) - X(t;\widetilde{\xi})\|_{L^p(\Omega; \mathcal{V})}^p \lesssim \|S(t)(\xi - \widetilde{\xi})\|_{L^p(\Omega; \mathcal{V})}^p + \int_0^t r(t-s) \|S(s)(\xi - \widetilde{\xi})\|_{L^p(\Omega; \mathcal{V})}^p\, \mathrm{d}s
    \]
    where $r \in \{ r_{\mathrm{gen}}^{(p)}, r_{\mathrm{b=0}}^{(p)}\}$. Likewise, if \eqref{eq: contraction estimate additive} holds for $\mathcal{V} = \mathcal{V}$ and $\sigma \equiv \sigma_0$, then we obtain
    \[
        \| X(t;\xi) - X(t;\widetilde{\xi})\|_{L^p(\Omega; \mathcal{V})} \lesssim \|S(t)(\xi - \widetilde{\xi})\|_{L^p(\Omega; \mathcal{V})} + \int_0^t r(t-s) \|S(s)(\xi - \widetilde{\xi})\|_{L^p(\Omega; \mathcal{V})}\, \mathrm{d}s.
    \]
    In particular, if $S_\infty\xi = S_\infty\widetilde{\xi}$, then the right-hand sides converge to zero, but a rate of convergence is not available unless we study convergence on the larger space $\mathcal{V} \hookrightarrow \mathcal{V}_0$.

	\subsection{Limit distributions and invariant measures}

	Let $\mathcal{P}(\mathcal{V})$ be the convex space of all Borel probability measures over $\mathcal{V}$ and let $\mathcal{P}_p(\mathcal{V})$ be the subspace of all probability measures with finite $p$-th moment, i.e.
    \[
        m_p(\rho) \coloneqq \left(\int_{\mathcal{V}}\|\xi\|_{\mathcal{V}}^p\, \rho(\mathrm{d}\xi)\right)^{\frac{1}{p}} < \infty.
    \]
    Similarly, we introduce $\mathcal{P}_p(\mathcal{V}_0) \subset \mathcal{P}(\mathcal{V}_0)$. Note that the embedding $\mathcal{V} \hookrightarrow \mathcal{V}_0$ induces an embedding $\mathcal{P}_p(\mathcal{V}) \hookrightarrow \mathcal{P}_p(\mathcal{V}_0)$. The space $\mathcal{P}_p(\mathcal{V}_0)$ is a polish space when equipped with the $p$-Wasserstein distance
	\[
	\mathcal{W}_{p,\mathcal{V}_0}(\pi,\widetilde{\pi}) = \left(\inf_{\nu\in\mathcal{C}_0(\pi,\widetilde{\pi})}\int_{ \mathcal{V}_0 \times \mathcal{V}_0} \norm{x-y}_{\mathcal{V}_0}^p\,\nu(\d x,\d y)\right)^{\frac{1}{p}}
	\]
	where $\mathcal{C}_0(\pi,\widetilde{\pi})$ denotes the set of all couplings of $\pi$ and $\widetilde{\pi}$ on $\mathcal{V}_0 \times \mathcal{V}_0$. Likewise, let $\mathcal{C}(\rho, \widetilde{\rho})$ denote the collection of all couplings on $\mathcal{V} \times \mathcal{V}$, whenever $\rho, \widetilde{\rho} \in \mathcal{P}(\mathcal{V})$.

	Recall that $(P_{t})_{t \geq 0}$ denotes the transition semigroup of the process given by Corollary \ref{cor: time homogeneous Markov}. Denote by $p_{t}(\xi,\mathrm{d}y)$ its transition probability kernel on $\mathcal{V}$. Then the action of the transition semigroup on probability measures $\rho \in \mathcal{P}_p(\mathcal{V})$ is given by $P_{t}^* \rho(\d y) = \int_{\mathcal{V}} p_{t}(\xi, \mathrm{d}y)\, \rho(\mathrm{d}\xi)$ and because of previously established global moment bounds, it leaves $\mathcal{P}_p(\mathcal{V})$ invariant. Here $P_t^* \rho$ is the distribution of $X(t;\xi)$ where $\xi \in L^p(\Omega, \F_0, \P; \mathcal{V})$ satisfies $\rho \sim \xi$. Below, we need the following observation that the dynamics leaves the null-space of $S_{\infty}$ invariant.

	\begin{lemma}\label{lemma: projection}
		Suppose that Assumptions \ref{assumption SEE} and \ref{assumption long-time} are satisfied, and let $p \in (2,\infty)$ satisfy \eqref{eq: continuity}. Let $\xi \in L^p(\Omega, \F_0, \P; \mathcal{V})$, and $X(\cdot; \xi)$ be the corresponding unique solution of \eqref{eq: abstract Markov lift time inhomogeneous} with $s = 0$. Then $S_\infty X(\cdot; \xi) = S_\infty \xi$ holds a.s.~for each $t \geq 0$. In particular, for each $\rho \in \mathcal{P}_p(\mathcal{V})$ it holds that $P_t^* \rho \circ S_\infty^{-1} = \rho \circ S_\infty^{-1}$ for each $t \geq 0$.
	\end{lemma}
	\begin{proof}
		Firstly, since $S_\infty \in L(\mathcal{V})$, we can compute $S_\infty X(\cdot; \xi)$ by pulling the projection operator inside the integrals in \eqref{eq: abstract Markov lift time inhomogeneous}. Then using $S_\infty S(t)\xi_b = 0$ and $S_\infty S(t)\xi_{\sigma} = 0$, the latter gives $S_\infty X(t; \xi) = S_\infty S(t)\xi$. Since $S_\infty$ is, by definition, the projection operator onto the fixed space of the semigroup, we get $S_\infty S(t)\xi = S_\infty \xi$ which proves the assertion.
	\end{proof}

    Since \eqref{eq: abstract mild formulation Markovian lift} has a probabilistically strong, analytically mild solution, we have the freedom to choose the filtration $(\mathcal{F}_t)_{t \geq 0}$ such that $W$ is an $(\mathcal{F}_t)_{t \geq 0}$-cylindrical Wiener process. Thus, by enlargement of $\mathcal{F}_0$ if necessary, let us suppose that $\mathcal{F}_0$ is rich enough such that for each $\rho \in \mathcal{P}_p(\mathcal{V})$ there exists $\xi \in L^p(\Omega, \mathcal{F}_0, \P; \mathcal{V})$ with $\xi \sim \rho$.

    The following is our main result on the existence and characterisation of limit distributions of the process $X(t;\xi)$ obtained from \eqref{eq: abstract Markov lift time inhomogeneous} with $s = 0$.

	\begin{theorem}\label{theorem_limit_distribution}
		Suppose that Assumptions \ref{assumption SEE} and \ref{assumption long-time} are satisfied. Let $p \in (2,\infty)$ such that \eqref{eq: continuity} holds. Then the following assertions hold:
		\begin{enumerate}
			\item (general case) If \eqref{eq:bounded_moment_condi_multi} and \eqref{eq:contraction_estimate_condi_multi} hold, then for each $\rho \in \mathcal{P}_p(\mathcal{V})$, there exists a unique probability measure $\pi_{\rho} \in \mathcal{P}_p(\mathcal{V})$ such that
			\begin{equation}\label{eq: convergence Wasserstein}
				\mathcal{W}_{p,\mathcal{V}_0}( P_t^* \rho,\pi_{\rho}) \lesssim \left(1 + m_p(\rho) \right) \left(\mathcal{R}_{\mathrm{gen}}^{p}(t)\right)^{1/p}.
			\end{equation}

            \item (no drift) If $b = 0$ and $\xi_b = 0$, and \eqref{eq:bounded_moment_condi_multi b zero} and \eqref{eq:contraction_estimate_condi_multi b zero} hold, then for each $\rho \in \mathcal{P}_p(\mathcal{V})$ there exists a unique probability measure $\pi_{\rho} \in \mathcal{P}_p(\mathcal{V})$ such that \eqref{eq: convergence Wasserstein} holds with $\mathcal{R}_{\mathrm{gen}}^{p}$ replaced by $\mathcal{R}_{\mathrm{b=0}}^{p}$.

			\item (additive noise) If $\sigma \equiv \sigma_0$ does not depend on $u \in H$, and
			\begin{equation}\label{eq:small_nonlinearity}
				\max\{C_{b, \mathrm{lip}} \norm{\Xi}_{L(\mathcal{V}_0,V)} , C_{b,\mathrm{lin}}\norm{\Xi}_{L(\mathcal{V},V)} \} \int_0^{\infty}\| S(t)\xi_b\|_{L(H_b, \mathcal{V})}\, \mathrm{d}t < 1,
			\end{equation}
			then for each $\rho \in \mathcal{P}_p(\mathcal{V})$, there exists a unique $\pi_{\rho} \in \mathcal{P}_p(\mathcal{V})$ such that
			\begin{equation*}\label{eq:lim_distrib_rate}
				\mathcal{W}_{p,\mathcal{V}_0}( P_t^* \rho, \pi_{\rho})\lesssim
				\left(1+ m_p(\rho) \right) \mathcal{R}_{\mathrm{add}}(t).
			\end{equation*}
		\end{enumerate}
		In all cases the limit distribution $\pi_{\rho}$ satisfies the disintegration property
        \begin{align}\label{eq: limit distribution decomposition}
            \pi_{\rho} = \int_{\mathcal{V}} \pi(\xi, \cdot) \rho(\mathrm{d}\xi) \ \text{ where }\ \pi(\xi,\cdot) = \pi_{\delta_{\xi}}.
        \end{align}
        Moreover, for given $\rho, \widetilde{\rho} \in \mathcal{P}_p(\mathcal{V})$ and limit distributions $\pi_{\rho},\pi_{\widetilde{\rho}}\in\mathcal{P}_p(\mathcal{V})$ it holds that
		\begin{align}\label{eq: continuity pi}
		\mathcal{W}_{p, \mathcal{V}_0}(\pi_{\rho}, \pi_{\widetilde{\rho}}) \lesssim \inf_{H \in \mathcal{C}(\rho, \widetilde{\rho})} \left(\int_{\mathcal{V} \times \mathcal{V}} \| S_\infty\xi - S_\infty \widetilde{\xi}\|_{\mathcal{V}_0}^p \, H(\mathrm{d}\xi, \mathrm{d}\widetilde{\xi}) \right)^{1/p}.
		\end{align}
	\end{theorem}
	\begin{proof}
        \textit{Step 1.} Let $\rho, \widetilde{\rho} \in \mathcal{P}_p(\mathcal{V})$. We first prove a contraction estimate for $(P_t^* \rho, P_t^* \widetilde{\rho})$ in the Wasserstein distance. Let $H \in \mathcal{C}(\rho, \widetilde{\rho})$ be arbitrary. Using the convexity of the Wasserstein distance and then Lemma \ref{lemma: general case} yields
		\begin{align*}
			\mathcal{W}_{p,\mathcal{V}_0}( P^*_t \rho, P^*_t \widetilde{\rho})^p
            &\leq \int_{\mathcal{V} \times \mathcal{V}} \mathcal{W}_{p, \mathcal{V}_0}(P_t^* \delta_{\xi}, P_t^* \delta_{\widetilde{\xi}})^p \, H(\mathrm{d}\xi, \mathrm{d}\widetilde{\xi})
            \\ &\leq \int_{\mathcal{V} \times \mathcal{V}} \norm{X(t;\xi) - X(t; \widetilde{\xi})}_{L^p(\Omega;\mathcal{V}_0)}^p \, H(\mathrm{d}\xi, \mathrm{d}\widetilde{\xi})
            \\ &\lesssim \int_{\mathcal{V} \times \mathcal{V}} \left( \norm{\xi- \widetilde{\xi}}_{\mathcal{V}}^p \mathcal{R}_{\mathrm{gen}}^{p}(t) + \| S_\infty \xi - S_\infty \widetilde{\xi}\|_{\mathcal{V}_0}^p \right)\, H(\mathrm{d}\xi, \mathrm{d}\widetilde{\xi})
            \\ &\lesssim \left( m_p(\rho)^p + m_p(\widetilde{\rho})^p \right) \mathcal{R}_{\mathrm{gen}}^{p}(t) + \int_{\mathcal{V} \times \mathcal{V}} \| S_\infty \xi - S_\infty \widetilde{\xi}\|_{\mathcal{V}_0}^p \, H(\mathrm{d}\xi, \mathrm{d}\widetilde{\xi})
		\end{align*}
        where the last inequality is satisfied since $H$ is a coupling of $\rho, \widetilde{\rho}$. Secondly, let us show that if $\rho \circ S_\infty^{-1} = \widetilde{\rho} \circ S_\infty^{-1}$, then there exists $H \in \mathcal{C}(\rho, \widetilde{\rho})$ such that
        \[
        \int_{\mathcal{V} \times \mathcal{V}} \| S_\infty \xi - S_\infty \widetilde{\xi}\|_{\mathcal{V}_0}^p \, H(\mathrm{d}\xi, \mathrm{d}\widetilde{\xi}) = 0.
        \]
        By disintegration let us write $\rho(\mathrm{d}\xi) = \rho(\xi', \mathrm{d}\xi) (\rho \circ S_\infty^{-1})(\mathrm{d}\xi')$ and $\widetilde{\rho}(\mathrm{d}\xi) = \widetilde{\rho}(\xi', \mathrm{d}\xi) (\widetilde{\rho}\circ S_\infty^{-1})(\mathrm{d}\xi') = \widetilde{\rho} (\xi', \mathrm{d}\xi) (\rho \circ S_\infty^{-1})(\mathrm{d}\xi')$. Define
        \[
            H(A \times B) = \int_{\mathcal{V} \times \mathcal{V}} \int_{V \times V}\1_{A}(\xi) \1_{B}(\widetilde{\xi}) \rho(\xi', \mathrm{d}\xi) \widetilde{\rho}(\widetilde{\xi}', \mathrm{d}\widetilde{\xi}) \widetilde{H}(\mathrm{d}\xi', \mathrm{d}\widetilde{\xi}')
        \]
        where $\widetilde{H}$ is a probability measure on $V \times V$ given by $\widetilde{H}(A' \times B') = \rho \circ S_\infty^{-1}(A' \cap B')$. For this choice of coupling, we find
        \begin{align*}
            \int_{\mathcal{V} \times \mathcal{V}} \| S_\infty \xi - S_\infty \widetilde{\xi}\|_{\mathcal{V}_0}^p \, H(\mathrm{d}\xi, \mathrm{d}\widetilde{\xi})
            &= \int_{\mathcal{V} \times \mathcal{V}} \int_{V \times V} \| S_\infty \xi - S_\infty \widetilde{\xi}\|_{\mathcal{V}_0}^p \rho(\xi', \mathrm{d}\xi) \widetilde{\rho}(\widetilde{\xi}', \mathrm{d}\widetilde{\xi}) \widetilde{H}(\mathrm{d}\xi', \mathrm{d}\widetilde{\xi}')
            \\ &= \int_{V \times V} \| \xi' - \widetilde{\xi}'\|_{\mathcal{V}_0} \widetilde{H}(\mathrm{d}\xi', \mathrm{d}\widetilde{\xi}')
            \\ &= 0
        \end{align*}
        since $\rho(\xi', \mathrm{d}\xi)$ is supported on $\{y  :  S_\infty y = \xi'\}$, $\widetilde{\rho}(\widetilde{\xi}', \mathrm{d}\widetilde{\xi})$ is supported on $\{y \ : \ Py = \widetilde{\xi}' \}$, and $\widetilde{H}$ is by definition supported on the diagonal.

        \textit{Step 2.} Let us now show that for each $\rho \in \mathcal{P}_p(\mathcal{V})$, $(P^*_t\rho)_{t \geq 0}$ is a Cauchy sequence in $\mathcal{P}_p(\mathcal{V}_0)$. Let $t, \tau \geq 0$. Then $P_t^* \rho \circ S_\infty^{-1} = \rho \circ S_\infty^{-1} = P_{t+\tau}^* \rho \circ S_\infty^{-1}$ by Lemma \ref{lemma: projection}. Hence, we obtain from step 1
        \begin{align*}
            \mathcal{W}_{p,\mathcal{V}_0}( P^*_t \rho, P^*_{t+\tau}\rho)
            &=  \mathcal{W}_{p,\mathcal{V}_0}( P^*_t \rho, P^*_{t} P_{\tau}^*\rho)
            \\ &\lesssim \left( m_p(\rho) + m_p(P_{\tau}^*\rho) \right) \left(\mathcal{R}_{\mathrm{gen}}^{p}(t)\right)^{1/p}
            \lesssim \left( 1 + m_p(\rho) \right) \left(\mathcal{R}_{\mathrm{gen}}^{p}(t) \right)^{1/p}
        \end{align*}
        where the last inequality holds uniformly in $\tau$ and follows from the uniform moment bounds provided in Lemma \ref{lemma: general case}. Since $\mathcal{R}_{\mathrm{gen}}^{p}(t) \to 0$ as $t\to \infty$ by Lemma \ref{lemma: convergence rate bound}, we conclude $\mathcal{W}_{p,\mathcal{V}_0}( P^*_t \rho, P^*_{t+\tau}\rho)\longrightarrow0$ as $t\longrightarrow0$ uniformly in $\tau$. Consequently, $(P^*_t\rho )_{t\geq0}$ is a Cauchy sequence in $\mathcal{P}_p(\mathcal{V}_0)$ with respect to the $p$-th Wasserstein distance. Hence, it has a limit denoted by $\pi_{\rho}\in\mathcal{P}_p(\mathcal{V}_0)$. Furthermore, we obtain
		\begin{align*}
			\mathcal{W}_{p,\mathcal{V}_0}( P_t^*\rho, \pi_{\rho})
			\leq \limsup_{\tau\to\infty}\mathcal{W}_{p,\mathcal{V}_0}( P_t^* \rho, P^*_{t+\tau}\rho )
			\lesssim \left( 1 + m_p(\rho) \right) \left(\mathcal{R}_{\mathrm{gen}}^{p}(t)\right)^{1/p}
		\end{align*}
		which proves the desired convergence rate. The disintegration property \eqref{eq: limit distribution decomposition} follows from the weak convergence $p_t(\xi, \cdot) = P_t^* \delta_{\xi} \Longrightarrow \pi(\xi, \cdot)$ on $\mathcal{V}_0$ and
        \begin{align*}
            \int_{\mathcal{V}_0}f(y) \left( \int_{\mathcal{V}_0} \pi(\xi, \mathrm{d}y) \rho(\mathrm{d}\xi) \right)
            &=  \int_{\mathcal{V}_0} \int_{\mathcal{V}_0}f(y) \pi(\xi, \mathrm{d}y) \rho(\mathrm{d}\xi)
            \\ &= \lim_{t \to \infty} \int_{\mathcal{V}_0} \int_{\mathcal{V}_0}f(y) p_t(\xi, \mathrm{d}y) \rho(\mathrm{d}\xi)
            \\ &= \lim_{t \to \infty} \int_{\mathcal{V}_0}f(y) (P_t^* \rho)(\mathrm{d}y)
            = \int_{\mathcal{V}_0}f(y) \pi_{\rho}(\mathrm{d}y)
        \end{align*}
        where $f \in C_b(\mathcal{V}_0)$.

       Note that $\|\cdot\|_{\mathcal{V}}^{p}\colon\mathcal{V}_0\longrightarrow[0,+\infty]$ is lower semi-continuous and bounded from below. Using the Portmanteau theorem and Lemma \ref{lemma: general case}, we have
		\begin{align}\label{eq: moment bound limit distribution}
		\int_{\mathcal{V}_{0}}\|y\|_{\mathcal{V}}^{p}\,\pi_{\rho}(\d y)
		\leq \liminf_{t\to\infty} \E\left[\|X(t;\xi)\|_{\mathcal{V}}^p\right]
		\lesssim 1+\E\left[\|\xi\|_{\mathcal{V}}^p\right]<\infty
		\end{align}
        where $\xi \in L^p(\Omega, \F_0, \P; \mathcal{V})$ is such that $\xi \sim \rho$. Consequently, since $\mathcal{V} = \{y \in \mathcal{V}_0  :  \|y\|_{\mathcal{V}} < \infty\}$, we conclude $\pi_{\rho}(\mathcal{V}_0\setminus\mathcal{V})=0$ and hence $\pi_{\rho}\in\mathcal{P}_p(\mathcal{V})$.

        Finally, let $\pi_{\rho},\pi_{\widetilde{\rho}}$ be the limit distributions for $\rho,\widetilde{\rho}\in \mathcal{P}_p(\mathcal{V})$. Then
		\begin{align*}
			\mathcal{W}_{p, \mathcal{V}_0}(\pi_{\rho},\pi_{\widetilde{\rho}})
			\leq \mathcal{W}_{p,\mathcal{V}_0}(\pi_{\rho}, P_t^* \rho) + \mathcal{W}_{p, \mathcal{V}_0}( P_t^* \rho, P_t^* \widetilde{\rho}) + \mathcal{W}_{p,\mathcal{V}_0}( P_t^* \widetilde{\rho},\pi_{\widetilde{\rho}}).
		\end{align*}
		Let $H$ be any coupling of $(\rho, \widetilde{\rho})$ supported on $\mathcal{V} \times \mathcal{V}$. Then, by passing to the limit $t\longrightarrow \infty$ and using step 1, we find
		\begin{align*}
			\mathcal{W}_{p,\mathcal{V}_0}(\pi_{\rho},\pi_{\widetilde{\rho}})
			\leq \limsup_{t\to\infty}\, \mathcal{W}_{p,\mathcal{V}_0}( P_t^* \rho, P_t^* \widetilde{\rho})
			\lesssim \left(\int_{\mathcal{V} \times \mathcal{V}} \| S_\infty\xi - S_\infty \widetilde{\xi}\|_{\mathcal{V}_0}^p \, H(\mathrm{d}\xi, \mathrm{d}\widetilde{\xi})\right)^{1/p}.
		\end{align*}
		Taking the infimum over all $H$ proves all assertions in the general case (i).

		For the case (ii), we may use the same argument but now with an application of Lemma \ref{lemma: general case b zero} instead. Finally, for the case of additive noise, we may proceed as above with the only difference that we may use Lemma \ref{lemma contraction estimate} instead of Lemma \ref{lemma: general case}.
	\end{proof}

    Recall that $\pi \in \mathcal{P}(\mathcal{V})$ is called \textit{invariant measure}, if $P_t^* \pi = \pi$ holds for each $t \geq 0$. In all three cases, the Feller property implies that for each limit distribution $\pi_{\rho}$ we can associate a stationary process with the corresponding Markovian lift which is, therefore, an invariant measure, see \cite[Proposition 11.2]{DaPrato_Zabczyk_2014} and \cite[Proposition 11.5]{DaPrato_Zabczyk_2014}.

	\begin{corollary}\label{remark: stationary process}
		Suppose the same conditions as in Theorem \ref{theorem_limit_distribution} are satisfied. Then for each $\rho \in \mathcal{P}_p(\mathcal{V})$, and each $\xi \in L^p(\Omega, \F_0, \P; \mathcal{V})$ with $\xi \sim \pi_{\rho}$ the process $X(\cdot; \xi)$ is stationary. In particular, also the stochastic Volterra process $u(\cdot; G) = \Xi X(\cdot; \xi)$ with $G = \Xi S(\cdot)\xi$ is stationary.
	\end{corollary}

    It follows from Corollary \ref{remark: stationary process} that each invariant measure $\pi \in \mathcal{P}_{p}(\mathcal{V})$ is the limit distribution of the stationary process $X(\cdot; \xi)$ where $\xi \in L^p(\Omega, \F_0, \P; \mathcal{V})$ is such that $\mathcal{L}(\xi) = \pi$. Hence, the space of all invariant measures in $\mathcal{P}_p(\mathcal{V})$ coincides with the space of all limit distributions.

    Let $\pi(\xi, \cdot) \in \mathcal{P}_p(\mathcal{V})$ with $\xi \in \mathcal{V}$ be the limit distribution with initial state $\rho = \delta_{\xi}$. Define the transition operator $\Pi\colon B(\mathcal{V}) \longrightarrow B(\mathcal{V})$ by
    \[
        (\Pi f)(\xi) = \int_{\mathcal{V}} f(z)\, \pi(\xi, \mathrm{d}z), \qquad \xi \in \mathcal{V},
    \]
    where $B(\mathcal{V})$ denotes the space of bounded measurable functions $f\colon \mathcal{V} \longrightarrow \R$. Denote by $\mathrm{Lip}(\mathcal{V})$ the space of Lipschitz continuous bounded functions on $\mathcal{V}$. Since $\mathcal{V} \hookrightarrow \mathcal{V}_0$, we get $\mathrm{Lip}(\mathcal{V}_0) \hookrightarrow \mathrm{Lip}(\mathcal{V})$ and $C_b(\mathcal{V}_0) \hookrightarrow C_b(\mathcal{V})$. Below, we provide another characterisation of invariant measures in terms of the operator $\Pi$.

    \begin{corollary}\label{corollary Pi operator}
        Suppose that the same conditions as in Theorem \ref{theorem_limit_distribution} are satisfied. Then $\Pi f(\xi) = \Pi f (S_{\infty}\xi)$ holds for all $f \in B(\mathcal{V})$ and $\xi \in \mathcal{V}$. Moreover, for each $f \in C_b(\mathcal{V}_0)$
        \[
            \lim_{t \to \infty} (P_tf)(\xi) = (\Pi f)(\xi), \qquad \xi \in \mathcal{V},
        \]
        and, in particular, $\Pi P_t = \Pi = \Pi^2$ holds for each $t \geq 0$. Finally, $\rho \in \mathcal{P}(\mathcal{V})$ is an invariant measure if and only if
        \begin{align}\label{eq: Pi invariant measure}
            \int_{\mathcal{V}} \Pi f(x)\, \rho(\mathrm{d}x) = \int_{\mathcal{V}}f(x)\, \rho(\mathrm{d}x), \qquad \forall f \in C_b(\mathcal{V}).
        \end{align}
    \end{corollary}
    \begin{proof}
        Let $\xi, \widetilde{\xi} \in \mathcal{V}$ and let $H$ be the optimal coupling of $(\pi(\xi, \cdot), \pi(\widetilde{\xi}, \cdot))$ with respect to $\mathcal{W}_{1, \mathcal{V}_0}$. Then we obtain for each $f \in \mathrm{Lip}(\mathcal{V}_0)$ from \eqref{eq: continuity pi} the bound
        \begin{align*}
            |\Pi f(\xi) - \Pi f(\widetilde{\xi})| &= \left| \int_{\mathcal{V} \times \mathcal{V}} (f(x) - f(y))\, H(\mathrm{d}x, \mathrm{d}y) \right| \notag
            \\ &\leq \|f \|_{\mathrm{Lip}} \int_{\mathcal{V} \times \mathcal{V}} \|x - y\|_{\mathcal{V}_0}\, H(\mathrm{d}x, \mathrm{d}y) \notag
            \\ &= \|f\|_{\mathrm{Lip}} \mathcal{W}_{1, \mathcal{V}_0}(\pi(\xi, \cdot), \pi(\widetilde{\xi}, \cdot))
            \lesssim \|f\|_{\mathrm{Lip}} \| S_{\infty}\xi - S_{\infty} \widetilde{\xi}\|_{\mathcal{V}_0}
        \end{align*}
        which shows that $\Pi\colon \mathrm{Lip}(\mathcal{V}_0) \longrightarrow \mathrm{Lip}(\mathcal{V})$, and $\Pi f(\xi) = \Pi f(S_{\infty}\xi)$ for $\xi \in \mathcal{V}$. Since $\mathrm{Lip}(\mathcal{V}_0) \subset B(\mathcal{V}_0)$ is dense with respect to bounded pointwise convergence, by approximation, the identity $\Pi f(\xi) = \Pi f(S_{\infty}\xi)$ extends onto $f \in B(\mathcal{V}_0) \hookrightarrow B(\mathcal{V})$.

        Since convergence in the Wasserstein distance, implies weak convergence, it follows that $\lim_{t \to \infty} (P_tf)(\xi) = (\Pi f)(\xi)$ holds for $f \in C_b(\mathcal{V}_0)$ and $\xi \in \mathcal{V}$. Let $x \in \mathcal{V}$ and $f \in B(\mathcal{V})$. Since $\pi(x, \mathrm{d}y)$ is by definition a limit distribution and hence an invariant measure, we get
        \[
            \Pi P_tf(x) = \int_{\mathcal{V}} P_tf(y)\, \pi(x, \mathrm{d}y) = \int_{\mathcal{V}} f(y)\, \pi(x, \mathrm{d}y) = \Pi f(x).
        \]
        If $f \in \mathrm{Lip}(\mathcal{V}_0)$, taking the limit $t \to \infty$ in $\Pi P_t f = \Pi f$ gives $\Pi^2 f(x) = \Pi f(x)$. Since $\mathrm{Lip}(\mathcal{V}_0) \hookrightarrow B(\mathcal{V}_0)$ densely with respect to bounded pointwise convergence, by approximation, this identity extends onto all functions $f \in B(\mathcal{V}_0) \hookrightarrow B(\mathcal{V})$.

        Let $\rho \in \mathcal{P}(\mathcal{V})$ be an invariant measure. Then $\int_{\mathcal{V}}P_tf(x)\, \rho(\mathrm{d}x) = \int_{\mathcal{V}}f(x)\, \rho(\mathrm{d}x)$ holds for all $t \geq 0$ and $f \in \mathrm{Lip}(\mathcal{V}_0) \hookrightarrow C_b(\mathcal{V})$. Hence, taking the limit $t \to \infty$ proves \eqref{eq: Pi invariant measure} for $f \in \mathrm{Lip}(\mathcal{V}_0)$. By approximation, \eqref{eq: Pi invariant measure} also holds for each $f \in C_b(\mathcal{V})$. To prove the converse direction, let $\pi$ satisfy \eqref{eq: Pi invariant measure}. Let $t \geq 0$ and $f \in C_b(\mathcal{V})$. Then $P_tf \in C_b(\mathcal{V})$, and hence \eqref{eq: Pi invariant measure} gives
        \begin{align*}
            \int_{\mathcal{V}} P_tf(x)\, \rho(\mathrm{d}x)
            = \int_{\mathcal{V}} \Pi P_t f(x)\, \rho(\mathrm{d}x)
            = \int_{\mathcal{V}} \Pi f(x)\, \rho(\mathrm{d}x) = \int_{\mathcal{V}} f(x)\, \rho(\mathrm{d}x)
        \end{align*}
        which shows that $\rho$ is an invariant measure.

    \end{proof}

    A few remarks are in place. Firstly, it follows from $\Pi^2 = \Pi$, that the collection of limit distributions $\pi(x, \mathrm{d}y)$ satisfies
    \[
        \pi(x, \mathrm{d}z) = \int_{\mathcal{V}} \pi(y, \mathrm{d}z) \pi(x, \mathrm{d}y)
    \]
    while \eqref{eq: Pi invariant measure} is equivalent to $\int_{\mathcal{V}} \pi(x, \mathrm{d}y) \rho(\mathrm{d}x) = \rho(\mathrm{d}y)$. Denote by $\Pi^*$ the adjoint operator acting on probability measures. Then, according to \eqref{eq: Pi invariant measure}, invariant measures are the fixed points of $\Pi^*$, i.e.~$\Pi^* \rho = \rho$. Moreover, $\Pi^*$ maps onto invariant measures in the sense that, for any choice $\rho \in \mathcal{P}(\mathcal{V})$, $\Pi^* \rho$ is an invariant measure. For Markov transition semigroups with a unique invariant measure, $\Pi f$ is constant, whence the limit of $P_t$ as $t \to \infty$ has one-dimensional range. For Markovian lifts of stochastic Volterra processes, invariant measures only depend on the range of $S_{\infty}$, while they are uniquely determined on $\mathrm{ran}(S_{\infty})^{\perp}$. Finally, let us remark that, by standard approximation methods, $\Pi$ can be extended to a large class of continuous polynomially bounded functions as introduced in the next section.

    \section{Limit Theorems}\label{sec:limit_theorems}

	\subsection{Law of Large Numbers}

    In this section, we derive the law of large numbers, including a convergence rate. First, we formulate and prove a general result beyond the specific structure of Markovian lifts that is of independent interest. Afterwards, we derive the desired law of large numbers for the Markovian lift studied in Section~\ref{sec:limit_dist_inv_meas} as a special case.

    \begin{theorem}\label{thm: LLN abstract}
       Let $Z \hookrightarrow Z_0$ be separable Hilbert spaces, and let $(X_t)_{t\geq0}$ be a $Z$-valued Markov process with transition probability kernel $p_t(\xi,\d x)$, and $C_b$-Feller transition semigroup $(P_t)_{t \geq 0}$. Suppose that for some $p \in[1,\infty)$ there exists a constant $C_p > 0$ such that
       \begin{equation}\label{eq:LLN 2+eps}
       \int_{Z}\|y\|_Z^{2p}\, p_t(x, \mathrm{d}y) \leq C_p(1 + \|x\|_Z^{2p}), \qquad \forall t \geq 0,
       \end{equation}
       $X$ admits a unique invariant measure $\pi\in\mathcal{P}_{2p}(Z)$, and there exists $\lambda > 0$ such that
       \begin{equation*}\label{eq:appendix_rate}
            \mathcal{W}_{1,Z_0}(p_t(\xi\,\cdot),\pi)\lesssim (1 + \|\xi\|_Z)(1\lor t)^{-\lambda},\quad  t\geq0,\ \xi\in Z.
       \end{equation*}
        Fix $\gamma \in (0,1]$ and let $f\colon Z_0 \longrightarrow \R$ satisfy for some constant $C_f > 0$
        \[
            |f(x) - f(y)| \leq C_f (1 + \| x\|_{Z_0}^p + \|y\|_{Z_0}^p)^{1-\gamma}\|x-y\|_{Z_0}^{\gamma}, \qquad x,y \in Z_0.
        \]
        Then the process $X$ satisfies the Law of Large Numbers in the mean-square sense, i.e.
        \begin{multline*}
             \E\left[\left|\frac{1}{T}\int_0^T f(X_t)\,\d t -\pi(f)\right|^{2}\right] \\  \lesssim \left( 1 + \E[\|X_0\|_Z^{2p}] + \int_{Z}\|y\|_{Z}^{2p}\, \pi(\mathrm{d}y) \right) T^{-(1\land \gamma \lambda)}(\log(T))^{\1_{\{\gamma \lambda=1\}}}
        \end{multline*}
        holds, where $\pi(f) = \int_{Z}f(x)\, \pi(\mathrm{d}x)$.
    \end{theorem}
    \begin{proof}
        Let us first prove a pointwise bound on $|\E[f(X_t)] - \pi(f)|$. Fix $x_0 \in Z$, and let $G_{x_0}$ be the optimal coupling of $(\delta_{x_0}, \pi)$ with respect to $\mathcal{W}_{1, Z_0}$. Then using the \emph{Kantorovich duality} \cite[Theorem 5.10]{villani2016optimal} with $c(x,y)=\norm{x-y}_{Z_0}^{\gamma}(1+\norm{x}_{Z_0}^p+\norm{y}_{Z_0}^p)^{1-\gamma}$, the convexity of the Wasserstein distance, and finally H\"older's inequality, we find
        \begin{align}\notag
            &\ \left| \int_{Z} f(y) p_t(x_0,\mathrm{d}y) - \int_Z f(y) \pi(\mathrm{d}y) \right|
            \\\notag &\quad \leq C_f \int_{Z_0\times Z_0}\norm{x-y}_{Z_0}^{\gamma}\left(1+\norm{x}_{Z_0}^p+\norm{y}_{Z_0}^p\right)^{1-\gamma}\,G_{x_0}(\d x,\d y)
            \\ \notag &\quad \lesssim \left( \int_{Z_0 \times Z_0} (1 + \|x\|_{Z_0}^p + \|y\|_{Z_0}^p)\, G_{x_0}(\mathrm{d}x, \mathrm{dy}) \right)^{1-\gamma} \left(\int_{Z_0\times Z_0}\norm{x-y}_{Z_0}\, G_{x_0}(\d x, \d y)\right)^{\gamma}
            \\ &\quad \lesssim \left( 1 + \|x_0\|_{Z}^p + \int_{Z}\|y\|_{Z}^p\, \pi(\mathrm{d}y)\right)^{1-\gamma} \left( 1 + \|x_0\|_Z \right)^{ \gamma}  (1\lor t)^{-\lambda \gamma} \notag
            \\ &\quad \lesssim \left( 1 + \|x_0\|_{Z}^p + \int_{Z}\|y\|_{Z}^p\, \pi(\mathrm{d}y)\right) (1\lor t)^{-\lambda \gamma}\label{eq: rate of convergence cesaro}
        \end{align}
        where we have used $Z \hookrightarrow Z_0$ and extended $\pi$ onto $Z_0$ by $\pi(Z_0 \backslash Z) = 0$. For random initial conditions $X_0 \sim \rho$, we obtain the bound
        \begin{align}\label{eq: 3}
            &|\E[f(X_t)] -\pi(f)| \notag
            \\ &\leq \int_{Z} \left| \int_{Z} f(y) p_t(x_0,\mathrm{d}y) - \int_Z f(y) \pi(\mathrm{d}y) \right| \rho(\mathrm{d}x_0) \notag
            \\ &\lesssim (1\lor t)^{-\lambda \gamma} \cdot \int_{Z} \left( 1 + \|x_0\|_{Z}^p + \int_{Z}\|y\|_{Z}^p\, \pi(\mathrm{d}y)\right)\, \rho(\mathrm{d}x_0)   \notag
            \\ &= (1\lor t)^{-\lambda \gamma} \cdot \left( 1 + \E[\|X_0\|_Z^p] + \int_{Z}\|y\|_{Z}^p\, \pi(\mathrm{d}y) \right).
        \end{align}

         Next, we prove a pointwise bound on $|\E[f(X_t)f(X_s)]-\pi(f)\pi(f)|$. For $0\leq s<t$, we use the Markov property, previous bound, and $|f(x)| \lesssim 1 + \|x\|_{Z_0}^p \lesssim 1 + \|x\|_Z^p$, to find
        \begin{align*}
            &|\E[f(X_t)f(X_s)]-\pi(f)\pi(f)|\\
            &\quad\leq \left| \E\left[f(X_s)\left(P_{t-s}f(X_s)-\pi(f)\right)\right]  \right| + |\pi(f)| \left|\E[f(X_s)]-\pi(f)\right|\\
            &\quad\lesssim \left| \E\left[f(X_s)\left(P_{t-s}f(X_s)-\pi(f)\right)\right]  \right|
            \\ &\qquad + (1\lor s)^{-\lambda \gamma} \cdot \left( 1 + \E[\|X_0\|_Z^p] + \int_{Z}\|y\|_{Z}^p\, \pi(\mathrm{d}y) \right) \left( 1 + \int_Z \|y\|_Z^p\, \pi(\mathrm{d}y) \right).
        \end{align*}
        To bound the first term, we use \eqref{eq: rate of convergence cesaro} for $P_{t-s}f(X_s) - \pi(f)$ combined with a repeated use of the H\"older inequality to find
        \begin{align*}
            &\ \left| \E\left[f(X_s)\left(P_{t-s}f(X_s)-\pi(f)\right)\right]  \right|
            \\ &\quad\lesssim (1\lor (t-s))^{-\lambda \gamma} \E\left[ ( 1+ \|X_s\|_Z^p) \left( 1 + \|X_{s}\|_{Z}^p + \int_{Z}\|y\|_{Z}^p\, \pi(\mathrm{d}y)\right)\right]
            \\ &\quad\lesssim (1\lor (t-s))^{-\lambda \gamma} \left( 1 + \sup_{\tau \geq 0}\E[ \|X_{\tau}\|_{Z}^{2p}] + \left(\int_{Z}\|y\|_{Z}^p\, \pi(\mathrm{d}y)\right)^2 \right)
            \\ &\quad\lesssim (1\lor (t-s))^{-\lambda \gamma} \left( 1 + \E[ \|X_{0}\|_{Z}^{2p}] + \left(\int_{Z}\|y\|_{Z}^p\, \pi(\mathrm{d}y)\right)^2 \right)
        \end{align*}
        where the last inequality follows from $\sup_{\tau \geq 0}\E[\|X_{\tau}\|_Z^{2p} ] \lesssim 1 + \E[\|X_0\|_Z^{2p}] < \infty$ due to \eqref{eq:LLN 2+eps}. Hence, we have shown that
        \begin{align}\label{eq: rate g}
            &\ |\E[f(X_t)f(X_s)]-\pi(f)\pi(f)| \\
            &\quad \lesssim (1\lor s)^{-\lambda \gamma} \cdot \left( 1 + \E[\|X_0\|_Z^p] + \int_{Z}\|y\|_{Z}^p\, \pi(\mathrm{d}y) \right) \left( 1 + \int_Z \|y\|_Z^p\, \pi(\mathrm{d}y) \right) \notag
            \\
            &\qquad + (1\lor (t-s))^{-\lambda \gamma} \left( 1 + \E[ \|X_{0}\|_{Z}^{2p}] + \left(\int_{Z}\|y\|_{Z}^p\, \pi(\mathrm{d}y)\right)^2 \right) \notag
            \\ &\quad \lesssim \left( (1\lor s)^{-\lambda \gamma} + (1\lor (t-s))^{-\lambda \gamma} \right) \cdot \left( 1 + \E[\|X_0\|_Z^{2p}] + \int_{Z}\|y\|_{Z}^{2p}\, \pi(\mathrm{d}y) \right).  \notag
        \end{align}
        We are now prepared to prove the assertion. Observe that
        \[
        \frac{1}{T}\int_0^T f(X_t)\,\d t - \pi(f) = \frac{1}{T}\int_0^T \left( f(X_t)-\E[f(X_t)] \right)\d t + \frac{1}{T}\int_0^T\E[f(X_t)]\,\d t -\pi(f)
        \]
        and thus
        \begin{align*}
            \left(\frac{1}{T}\int_0^T f(X_t)\,\d t-\pi(f)\right)^2 &= \left(\frac{1}{T}\int_0^T \left( f(X_t)-\E[f(X_t)] \right)\d t\right)^2
            \\ &\quad+ 2\left(\frac{1}{T}\int_0^T \left(f(X_t)-\E[f(X_t)] \right)\d t\right)\left(\frac{1}{T}\int_0^T\E[f(X_t)]\,\d t -\pi(f)\right)\\
            &\quad+\left(\frac{1}{T}\int_0^T\E[f(X_t)]\,\d t -\pi(f)\right)^2.
        \end{align*}
        By taking expectations and noting that the second term vanishes, we arrive at
        \begin{align}\label{eq:LLN p2}
            \E\left[\left|\frac{1}{T}\int_0^T f(X_s)\,\d s - \pi(f)\right|^2\right] &= \frac{1}{T^2}\int_0^T\int_0^T \operatorname{Cov}(f(X_t),f(X_s))\,\d t\,\d s
            \\ &\qquad + \left(\frac{1}{T}\int_0^T \E[f(X_t)]\,\d t - \pi(f)\right)^2. \notag
        \end{align}
        Let us now bound the first integral in \eqref{eq:LLN p2}. Firstly, we write
        \begin{align}\label{eq: covariance decomposition}
            \frac{1}{T^2}\int_0^T\int_0^T \operatorname{Cov}(f(X_t),f(X_s))\,\d t\,\d s
            &= \frac{1}{T^2}\int_0^T\int_0^T \left(\E[f(X_s)f(X_t)]-\pi(f)^2\right)\,\d t\,\d s
            \\ &\qquad + \frac{1}{T^2}\int_0^T\int_0^T \left(\pi(f)^2 - \E[f(X_t)]\E[f(X_s)]\right)\,\d t\,\d s. \notag
        \end{align}
        Then, for the second term, using $|f(x)| \lesssim 1 + \|x\|_{Z_0}^p \lesssim 1 + \|x\|_Z^p$ and then \eqref{eq: 3}, we obtain
        \begin{align} \notag
            & \frac{1}{T^2}\int_0^T\int_0^T |\pi(f)\pi(f)-\E[f(X_s)]\E[f(X_t)]|\,\d s \,\d t
            \\ &\quad\leq \frac{1}{T^2}\int_0^T\int_0^T |\pi(f)||\pi(f)-\E[f(X_s)]|\,\d s\,\d t + \frac{1}{T^2}\int_0^T\int_0^T \E[|f(X_s)|]|\pi(f)-\E[f(X_t)]|\,\d s\,\d t \notag
            \\ &\quad\lesssim \left( 1 + \sup_{\tau \geq 0}\E[ \|X_{\tau}\|_{Z}^p] + \int_{Z}\|y\|_{Z}^p\, \pi(\mathrm{d}y)  \right)\frac{1}{T}\int_0^T |\pi(f)-\E[f(X_s)]|\,\d s  \notag
            \\ &\quad \lesssim \left( 1 + \E[\|X_{0}\|_{Z}^p] + \int_{Z}\|y\|_{Z}^p\, \pi(\mathrm{d}y)\right)^{2} T^{-(1\land\lambda \gamma)} \left(\log(T)\right)^{{\1_{\{\lambda \gamma=1\}}}}.\nonumber \label{eq: second covariance}
        \end{align}
        For the first term, we find using \eqref{eq: rate g},
        \begin{align*}
            &\ \frac{1}{T^2}\int_0^T\int_0^T \left(\E[f(X_s)f(X_t)]-\pi(f)^2\right)\d s\,\d t
            \\ &= \frac{2}{T^2} \int_0^T \int_0^t \left(\E[f(X_s)f(X_t)]-\pi(f)^2\right) \mathrm{d}s\, \mathrm{d}t
            \\ &\lesssim \left( 1 + \E[\|X_0\|_Z^{2p}] + \int_{Z}\|y\|_{Z}^{2p}\, \pi(\mathrm{d}y) \right) \cdot
            \frac{2}{T^2} \int_0^T \int_0^t \left( (1\lor s)^{-\lambda \gamma} + (1\lor (t-s))^{-\lambda \gamma} \right) \, \mathrm{d}s\, \mathrm{d}t
            \\ &\lesssim \left( 1 + \E[\|X_0\|_Z^{2p}] + \int_{Z}\|y\|_{Z}^{2p}\, \pi(\mathrm{d}y) \right) T^{-(1\land \gamma \lambda)}(\log(T))^{\1_{\{\gamma \lambda=1\}}}.
        \end{align*}
        Collecting all estimates, see \eqref{eq: rate of convergence cesaro} and the bound on \eqref{eq: covariance decomposition}, yields
        \begin{align*}
            &\ \E\left[\left|\frac{1}{T}\int_0^T f(X_s)\,\d s  - \pi(f)\right|^2\right]
            \\ &\quad \lesssim \left( 1 + \E[\|X_0\|_Z^p] + \int_{Z}\|y\|_{Z}^p\, \pi(\mathrm{d}y) \right)^{2} \left(\frac{1}{T}\int_0^T (1\lor t)^{-\lambda \gamma}\, \mathrm{d}t \right)^2
            \\ &\qquad + \left( 1 + \E[\|X_{0}\|_{Z}^p] + \int_{Z}\|y\|_{Z}^p\, \pi(\mathrm{d}y)\right)^{2} T^{-(1\land\lambda \gamma)} \left(\log(T)\right)^{{\1_{\{\lambda \gamma=1\}}}}
            \\ &\qquad + \left( 1 + \E[\|X_0\|_Z^{2p}] + \int_{Z}\|y\|_{Z}^{2p}\, \pi(\mathrm{d}y) \right) T^{-(1\land \gamma \lambda)}(\log(T))^{\1_{\{\gamma \lambda=1\}}}
            \\ &\quad \lesssim \left( 1 + \E[\|X_0\|_Z^{2p}] + \int_{Z}\|y\|_{Z}^{2p}\, \pi(\mathrm{d}y) \right) T^{-(1\land \gamma \lambda)}(\log(T))^{\1_{\{\gamma \lambda=1\}}}
        \end{align*}
        and hence proves the assertion.
    \end{proof}

    Below, we apply this result to our setting of Markovian lifts with possible multiple invariant measures as studied in Section~\ref{sec:limit_dist_inv_meas}. As a first step, define the class of admissible test functions as the weighted H\"older space $C^\gamma_q(\mathcal{V}_0)$ for $\gamma\in(0,1)$ and $q>0$ that consists of all functions $f\colon\mathcal{V}_0\longrightarrow\R$ such that
    \[
    \sup_{\xi,\eta\in\mathcal{V}_0,\xi\neq\eta} \frac{|f(\xi)-f(\eta)|}{\norm{\xi-\eta}_{\mathcal{V}_0}^{\gamma}(V_q(\xi) + V_q(\eta))^{1-\gamma}}<\infty.
    \]
    Here $V_q\colon \mathcal{V}_0 \longrightarrow [1,\infty ]$ denotes the weight function $V_q(\eta) = 1+\norm{\eta}_{\mathcal{V}_0}^q$.
    The following example illustrates how the limit theorems obtained in this section can be used to derive limit theorems for the corresponding stochastic Volterra process.
    \begin{example}
        Let $F\colon V \longrightarrow \R$ be such that
        \[
            |F(u) - F(v)| \leq C_F (1 + \|u\|_V^q + \|v\|_V^q)^{1-\gamma} \| u - v\|_V^{\gamma}, \qquad u,v\in V.
        \]
        Then $f\colon \mathcal{V}_0 \longrightarrow \R$ defined by $f(y) = F(\Xi y)$ satisfies $f \in C_{q}^{\gamma}(\mathcal{V}_0)$.
    \end{example}

    Recall that, for each $\xi \in \mathcal{V}$, there exists a limit distribution $\pi(\xi, \cdot) \in \mathcal{P}_p(\mathcal{V})$ given by Theorem \ref{theorem_limit_distribution}. Given $\xi \in L^p(\Omega, \F_0, \P; \mathcal{V})$, we show that the time averages $\frac{1}{T} \int_0^T f(X(t;\xi))\, \mathrm{d}t$ converge to the random variable
    \[
        (\Pi f)(\xi) = \int_{\mathcal{V}} f(y) \pi(\xi, \mathrm{d}y)
    \]
    obtained by pointwise evaluation of $\Pi f$ at the random variable $\xi$. Using Lemma \ref{lemma: convergence rate bound}, we have for each $\mathcal{R}\in\{(\mathcal{R}^{p}_{\text{gen}})^{1/p},(\mathcal{R}^{p}_{\text{b=0}})^{1/p},\mathcal{R}_{\text{add}}\}$ the bound
    \[
        \mathcal{R}(t) \lesssim (1 \vee t)^{- \chi}, \qquad t > 0
    \]
    where the convergence rate $\chi > 0$ satsifies
    \[
        \chi < \begin{cases}
            \frac{1}{p}\min\{\log(1/\norm{\rho_{\text{gen}}^{(p)}}_{L^1(\R_+)}),\lambda\}, & \text{if } \mathcal{R} = (\mathcal{R}^{p}_{\text{gen}})^{1/p}
            \\ \frac{1}{p}\min\{\log(1/\norm{\rho_{\text{b=0}}^{(p)}}_{L^1(\R_+)}),2\lambda\}, & \text{if } \mathcal{R} = (\mathcal{R}^{p}_{\text{b=0}})^{1/p}
            \\ \min\{\log(1/\norm{\rho_{\text{add}}}_{L^1(\R_+)}),\lambda\}, & \text{if } \mathcal{R} = \mathcal{R}_{\text{add}}.
        \end{cases}
    \]

    Then we obtain the following explicit convergence rates for the Law of Large Numbers.

    \begin{corollary}\label{theorem LLN lift}
        Suppose that Assumptions \ref{assumption SEE} and \ref{assumption long-time} are satisfied. Let $p \in (2,\infty)$ satisfy condition \eqref{eq: continuity}, and suppose that one of the following conditions holds:
		\begin{enumerate}
			\item (general case) Condition \eqref{eq:bounded_moment_condi_multi} is satisfied for $2p$ and \eqref{eq:contraction_estimate_condi_multi} for $p$.

            \item (no drift) If $b = 0$ and $\xi_b = 0$, then \eqref{eq:bounded_moment_condi_multi b zero} holds for $2p$ and \eqref{eq:contraction_estimate_condi_multi b zero} for $p$.

			\item (additive noise) If $\sigma \equiv \sigma_0$ does not depend on $u \in H$, then \eqref{eq:small_nonlinearity} holds.
		\end{enumerate}
		Let $\xi \in L^{2p}(\Omega, \F_0, \P; \mathcal{V})$, $\gamma \in (0,1]$ and $f \in C_{p}^{\gamma}(\mathcal{V}_0)$. Then
        \begin{align}\label{cor: LLN}
            \E\left[ \left| \frac{1}{T}\int_0^T f(X(t;\xi))\, \mathrm{d}t - (\Pi f)(\xi)\right|^{2}\right] \lesssim \left( 1 + \E[ \|\xi\|_{\mathcal{V}}^{2p}] \right)T^{- \min\{1, \chi \gamma\}}\left(\log(T)\right)^{{\1_{\{\chi
            \gamma=1 \}}}}.
        \end{align}
    \end{corollary}
    \begin{proof}
        Let us first consider the case of deterministic $\xi \in \mathcal{V}$. Since $S_{\infty}X(t;\xi) = S_{\infty}\xi$ by Lemma \ref{lemma: projection}, it follows that $\pi(\xi, \mathcal{V}^a) = 1$ where $a = S_{\infty}\xi$ and $\mathcal{V}^a = \{ \eta \in \mathcal{V} \ : \ S_{\infty}\eta = a \}$. In particular, by \eqref{eq: continuity pi} this limit distribution is unique on $\mathcal{V}^a$, and satisfies for all $\eta\in \mathcal{V}^a$
		\begin{equation*}\label{eq:convergence_rate coupling dist}
			\mathcal{W}_{p,\mathcal{V}_0}( p_t(\eta, \cdot), \pi(\xi, \cdot))
            \lesssim \left(1+\norm{\eta}_{\mathcal{V}}\right)\mathcal{R}(t)\lesssim \left(1+\norm{\eta}_{\mathcal{V}}\right)(1\lor t)^{-\chi},\qquad t\geq0,
		\end{equation*}
        where $\mathcal{R}\in\{(\mathcal{R}^{p}_{\text{gen}})^{1/p},(\mathcal{R}^{p}_{\text{b=0}})^{1/p},\mathcal{R}_{\text{add}}\}$. Since $\mathcal{W}_{1,\mathcal{V}_0}( p_t(\eta, \cdot), \pi(\xi, \cdot)) \leq \mathcal{W}_{p,\mathcal{V}_0}( p_t(\eta, \cdot), \pi(\xi, \cdot))$, the assumptions of Theorem \ref{thm: LLN abstract} are satisfied, which gives the desired result. Let us now consider the general case of random initial conditions $\xi \in L^{2p}(\Omega, \F_0, \P; \mathcal{V})$.  By disintegration, we arrive for $\rho\sim\xi$ at
        \begin{align*}
            &\E\left[\left|\frac{1}{T}\int_0^T f(X(t;\xi))\,\d t - \Pi f(\xi)\right|^{2}\right]\\
            &\quad= \int_{\mathcal{V}} \E\left[\left|\frac{1}{T}\int_0^T f(X(t; x))\,\d t - \Pi f(x) \right|^{2}\right] \rho(\mathrm{d}x)\\
            &\quad\lesssim T^{-\min\{1,\chi \gamma\}}\left(\log(T)\right)^{{\1_{\{\chi \gamma=1\}}}}\int_{\mathcal{V}}  1+\sup_{t\geq0}\E[\norm{X(t;x)}_{\mathcal{V}_0}^{2p} ] + \Pi\norm{\cdot}_{\mathcal{V}_0}^{2p}(x) \,\rho(\d x)
            \\ &\quad\lesssim T^{-\min\{1,\chi \gamma\}}\left(\log(T)\right)^{{\1_{\{\chi \gamma=1\}}}} \left( 1 + \int_{\mathcal{V}} \norm{x}_{\mathcal{V}}^{2p}\,\rho(\d x) \right)
        \end{align*}
        where we have used Theorem \ref{thm: LLN abstract}, Lemma \ref{lemma: general case} and Corollary \ref{corollary Pi operator} so that $\E[\norm{X(t;x)}_{\mathcal{V}_0}^{2p}] \lesssim \E[\norm{X(t;x)}_{\mathcal{V}}^{2p}] \lesssim 1 + \|x\|_{\mathcal{V}}^{2p}$ and by Fatou's Lemma
        \[
            |\Pi \norm{\cdot}_{\mathcal{V}_0}^{2p}(x)| = \int_{\mathcal{V}} \|y\|_{\mathcal{V}_0}^{2p}\, \pi(x, \mathrm{d}y)
            \lesssim \sup_{t \geq 0} \int_{\mathcal{V}} \|y\|_{\mathcal{V}_0}^{2p}\, p_t(x, \mathrm{d}y) \lesssim 1 + \|x\|_{\mathcal{V}}^{2p}.
        \]
    \end{proof}

\subsection{Central limit theorem}

    In this section, we prove the central limit theorem for the process
    \[
        A_T(\xi) = \sqrt{T}\left(\frac{1}{T}\int_0^T f(X(t;\xi))\,\d t - (\Pi f)(\xi) \right)
    \]
    where $f \in C_{\sqrt{p}}^{\gamma}(\mathcal{V}_0)$ and $\xi \in L^p(\Omega, \F_0, \P; \mathcal{V})$. For this purpose, let us first show an auxiliary result that the space $C_{q}^{\gamma}(\mathcal{V}_0)$ is sufficiently large to approximate all other functions on $L^2$ as stated below.
	\begin{lemma}
		For any $q\geq1$ and $\gamma\in(0,1]$, the space $C^\gamma_q(\mathcal{V}_0)$ is dense in $L^2(\mathcal{V}_0,\nu)$, where $\nu$ is an arbitrary finite Borel-measure on $\mathcal{V}_0$.
	\end{lemma}
	\begin{proof}
		Recall that every finite Borel measure on a metric space is regular, c.f. \cite[Theorem 1.1] {billingsley}. Thus, $\nu$ is completely determined by the values it attains on closed sets. Consequently, it suffices to show that indicator functions on closed sets can be approximated by bounded $C^\gamma_q(\mathcal{V}_0)$-functions. Let $C\subseteq \mathcal{V}_0$ be a closed set and $f_n\colon \mathcal{V}_0 \longrightarrow \R$ be given by $f_n(y) = ((1-n\operatorname{dist}(y,C))_+)^{\gamma}$ with $y\in \mathcal{V}_{0}$ and $n\in\N$, where $\operatorname{dist}(y,C)=\inf_{c\in C}\norm{y-c}_{\mathcal{V}_0}$ and $f_+ = f\lor0$. Note that $y\longmapsto\operatorname{dist}(y,C)$ is Lipschitz continuous, so $f_n\in C^\gamma_q(\mathcal{V}_0)$. Moreover, $f_n$ satisfies $0\leq f_n\leq1$ and for every $y\in \mathcal{V}_0$ $\lim_{n\to\infty}f_n(y)=\1_C(y)$. Clearly, $1\in L^2(\mathcal{V}_0,\nu)$ and so by dominated convergence $\lim_{n\to\infty}\norm{f_n-\1_C}_{L^2(\mathcal{V}_0,\nu)}=0$.
	\end{proof}

    Recall that $\rho_{\mathrm{gen}} = \rho^{(p)}_{\mathrm{gen}}, \rho_{\mathrm{b=0}} = \rho_{\mathrm{b=0}}^{(p)}$ also depend on the choice of $p$.

	\begin{theorem}[Central Limit Theorem]\label{theorem CLT general}
		Suppose that Assumptions \ref{assumption SEE} and \ref{assumption long-time} are satisfied. Let $p\in (4,\infty)$ be such that
        \begin{equation*}\label{eq:clt p condition}
            \frac{1}{\sqrt{p}} + \rho < 1.
        \end{equation*}
        Let $\gamma \in (0,1]$ and assume one of the following cases:
        \begin{enumerate}
            \item (general noise) Inequality \eqref{eq:bounded_moment_condi_multi} holds for $p$ and $\sqrt{p}$, \eqref{eq:contraction_estimate_condi_multi} holds for $\sqrt{p}$, and
            \[
                \frac{\lambda \gamma}{\sqrt{p}} > 1 \ \text{ and } \
                \| \rho^{(\sqrt{p})}_{\mathrm{gen}}\|_{L^1(\R_+)} < \e{- \frac{\sqrt{p}}{\gamma} }.
            \]

            \item (no drift) Inequality \eqref{eq:bounded_moment_condi_multi b zero} holds for $p$ and $\sqrt{p}$, \eqref{eq:contraction_estimate_condi_multi b zero} holds for $\sqrt{p}$,
            \[
                \frac{2 \lambda \gamma}{\sqrt{p}} > 1 \ \text{ and } \
                \| \rho_{\mathrm{b=0}}^{(\sqrt{p})}\|_{L^1(\R_+)} < \e{- \frac{\sqrt{p}}{\gamma}}.
            \]

			\item (additive noise) The inequality \eqref{eq:small_nonlinearity} is satisfied and
            \begin{equation*}\label{eq: CLT add condition}
                \lambda \gamma > 1 \ \text{ and } \
                \| \rho_{\mathrm{add}}\|_{L^1(\R_+)} < \e{- \frac{1}{\gamma} }.
            \end{equation*}
        \end{enumerate}
        Let $\xi \in L^p(\Omega, \F_0, \P; \mathcal{V})$, $f\in C^{\gamma}_{\sqrt{p}}(\mathcal{V}_0)$, and $\widetilde{\xi} \in L^p(\Omega, \F_0, \P; \mathcal{V})$ be such that $\widetilde{\xi} \sim \pi_{\mathcal{L}(\xi)}$, then
        \begin{equation*}\label{eq: sigma definition}
		      \sigma(x)^2 \coloneqq 2 \int_0^\infty \E \left[ \left(f(X(t; \widetilde{\xi})) - (\Pi f)(x)\right)\left(f(\widetilde{\xi}) - (\Pi f)(x) \right) \right]\,\d t
	   \end{equation*}
       is for $\mathcal{L}(\xi)$-a.a. $x \in \mathcal{V}$ well-defined, nonnegative, and it holds that
        \begin{equation}\label{eq:CLT}
			\sqrt{T}\left(\frac{1}{T}\int_0^T f(X(t;\xi))\,\d t - (\Pi f)(\xi) \right) \Longrightarrow \sigma(\xi)Z,\quad T\longrightarrow\infty
		\end{equation}
        where $Z \sim \mathcal{N}(0, 1)$ is independent of $\xi$.
	\end{theorem}
	\begin{proof}
        (i) Let us again first consider the case of deterministic $\xi \in \mathcal{V}$. By Theorem \ref{theorem_limit_distribution} applied for $\sqrt{p}$ and \eqref{eq: moment bound limit distribution} applied for $p$, there exists the limit distribution $\pi(\xi,\cdot) \in \mathcal{P}_p(\mathcal{V})$ satisfying $\pi(\xi, \mathcal{V}^a) = 1$ with $a = S_\infty\xi$, and for all $\eta\in\mathcal{V}^a$
        \begin{align*}
		      \mathcal{W}_{\sqrt{p},\mathcal{V}_0}( P_t^* \delta_{\eta}, \pi(\xi,\cdot))^{\sqrt{p}}\lesssim \left(1 + \norm{\eta}_{\mathcal{V}}^{\sqrt{p}}\right)\mathcal{R}_{\text{gen}}^{\sqrt{p}}(t).
		\end{align*}
        Note that $V_{\sqrt{p}} \in L^{\sqrt{p}}(\mathcal{V},\pi(\xi, \cdot))$. Moreover, using Lemma \ref{lemma: convergence rate bound} we obtain for each $\kappa\in (0,1)$
		\begin{align*}
			\int_1^\infty \left(\mathcal{R}_{\text{gen}}^{\sqrt{p}}(t)\right)^{\gamma/\sqrt{p}}\,\d t \lesssim \int_1^\infty  \left( t^{-\gamma\lambda}
            + t^{- \frac{\gamma}{\sqrt{p}} \log(1/\|\rho^{(\sqrt{p})}_{\mathrm{gen}}\|_{L^1(\R_+)})} + t^{- \lambda \gamma ( 1- \kappa)/ \sqrt{p}} \right)\,\d t.
		\end{align*}
        By assumption, this integral is finite provided that $\kappa$ is small enough.

        Let $\widetilde{\xi}\in L^p(\Omega,\P;\mathcal{V}^a)$ be such that $\widetilde{\xi}\sim \pi(\xi, \cdot)$. Then by Corollary \ref{remark: stationary process}, $X(\cdot;\widetilde{\xi})$ is a stationary process with $X(t;\widetilde{\xi})\sim \pi(\xi, \cdot)$ for each $t \geq 0$. For any $f\in C^{\gamma}_{\sqrt{p}}(\mathcal{V}_0)$ we obtain
        \begin{align*}
            |f(\eta) - f(\xi)| &\lesssim (1 + \| \eta\|_{\mathcal{V}_0}^{\sqrt{p}} + \| \xi\|_{\mathcal{V}_0}^{\sqrt{p}})^{1 - \gamma_p \sqrt{p}} \| \eta - \xi\|_{\mathcal{V}_0}^{\gamma_p \sqrt{p}}
            \\ &\lesssim (1 + \| \eta\|_{\mathcal{V}_0}^{\sqrt{p}} + \| \xi\|_{\mathcal{V}_0}^{\sqrt{p}})^{1 - \gamma_p} \| \eta - \xi\|_{\mathcal{V}_0}^{\gamma_p \sqrt{p}}
        \end{align*}
        where $\gamma_p = \gamma / \sqrt{p}$. Hence, it follows from \cite[Theorem 5.3.4]{Kulik}, that the Central Limit Theorem holds for $X(\cdot;\widetilde{\xi})$ with $\gamma$ replaced by $\gamma_p$. In particular, we obtain, $\sigma(\xi)^2 < \infty$ and
		\begin{equation*}\label{eq:clt addi stat}
			\sqrt{T}\left(\frac{1}{T}\int_0^T f(X(t;\widetilde{\xi}))\,\d t - (\Pi f)(\xi) \right) \Longrightarrow \mathcal{N}(0,\sigma(\xi)^2) \sim \sigma(\xi)Z,\quad T\longrightarrow\infty.
		\end{equation*}
		For the general case, we note that
		\begin{multline*}
			\sqrt{T}\left(\frac{1}{T}\int_0^T f(X(t;\xi))\,\d t - (\Pi f)(\xi) \right)
            \\ = \sqrt{T}\left(\frac{1}{T}\int_0^T f(X(t;\xi)) - f(X(t;\widetilde{\xi}))\,\d t\right) + \sqrt{T}\left(\frac{1}{T}\int_0^T f(X(t;\widetilde{\xi}))\,\d t - (\Pi f)(\xi)\right).
		\end{multline*}
		As $T\longrightarrow\infty$, the second summand converges weakly to $\mathcal{N}(0,\sigma(\xi)^2)$. For the first one, let us note that the pair of processes $(X(\cdot;\xi),X(\cdot;\widetilde{\xi}))$ satisfies by Lemma \ref{lemma: general case}
		\begin{align*}
			\norm{X(t;\xi)-X(t;\widetilde{\xi})}_{L^{\sqrt{p}}(\Omega;\mathcal{V}_0)}^{\sqrt{p}} \lesssim \norm{\xi-\widetilde{\xi}}_{L^{\sqrt{p}}(\Omega;\mathcal{V})}^{\sqrt{p}} \mathcal{R}_{\mathrm{gen}}^{\sqrt{p}}(t)
		\end{align*}
        since $a = S_\infty\xi = S_\infty\widetilde{\xi}$. Hence, using the H\"older continuity of $f$, then H\"olders inequality, Lemma \ref{lemma: general case}, and finally Jensen's inequality we arrive at
		\begin{align}
		&\frac{1}{\sqrt{T}}\E\left[\int_0^T |f(X(t;\xi)-f(X(t;\widetilde{\xi}))|\,\d t\right]\nonumber
        \\ &\quad\lesssim \frac{1}{\sqrt{T}}\int_0^T \E\left[\norm{X(t;\xi)-X(t;\widetilde{\xi})}_{\mathcal{V}_0}^{\gamma_p \sqrt{p} } (1+\norm{\xi}_\mathcal{V}^{\sqrt{p}}+\norm{\widetilde{\xi}}^{\sqrt{p}}_\mathcal{V})^{1-\gamma_p}\right]\d t \nonumber
        \\ &\quad\leq \frac{1}{\sqrt{T}}\int_0^T \left(\E\left[\norm{X(t;\xi)-X(t;\widetilde{\xi})}_{\mathcal{V}_0}^{ \sqrt{p}}\right]\right)^{\gamma_p} \left(\E\left[1+\norm{\xi}_\mathcal{V}^{\sqrt{p}}+\norm{\widetilde{\xi}}^{\sqrt{p}}_\mathcal{V}\right]\right)^{1-\gamma_p}\,\d t \nonumber
        \\ &\quad\lesssim \frac{1}{\sqrt{T}}\int_0^T \left(\mathcal{R}_{\mathrm{gen}}^{\sqrt{p}}(t)\right)^{\gamma_p}\,\d t \nonumber
		\end{align}
		which tends to zero as $T\longrightarrow\infty$. Then, by Slutsky's theorem, the central limit theorem \eqref{eq:CLT} follows for deterministic $\xi \in \mathcal{V}$. Finally, let $\xi \in L^p(\Omega, \F_0, \P; \mathcal{V})$ with $\mathcal{L}(\xi)=\rho$. Let $F \in C_b(\R)$. Then by conditioning and the corresponding result for deterministic $\xi$, we obtain
        \begin{align*}
            \E[ F(A_T(\xi)) ]
            = \int_{\mathcal{V}} \E[F(A_T(x))]\, \rho(\mathrm{d}x)
            \longrightarrow \int_{\mathcal{V}} \E[F(\mathcal{N}(0,\sigma(x)^2))]\, \rho(\mathrm{d}x).
        \end{align*}

        In case (ii), the proof is identical to case (i) with the only difference that we need to replace $\mathcal{R}_{\mathrm{gen}}^{\varepsilon}$ by $\mathcal{R}_{\mathrm{b=0}}^{\varepsilon}$. Similarly, case (iii) is analogous to case (I), with the only difference that the rate of convergence provided by Theorem \ref{theorem_limit_distribution} can be improved to
        \[
            \mathcal{W}_{\sqrt{p},\mathcal{V}_0}( P_t^* \delta_{\eta}, \pi(\xi, f))^{\sqrt{p}} \lesssim \left(1+\norm{\eta}_{\mathcal{V}_0}^{\sqrt{p}}\right) \left(\mathcal{R}^\lambda_{\text{add}}(t)\right)^{\sqrt{p}}.
        \]
        Since $\int_0^\infty \left(\mathcal{R}^\lambda_{\text{add}}(t)\right)^{\gamma}<\infty$ due to $\gamma \lambda > 1$ and Lemma \ref{lemma: convergence rate bound}, the assertion follows by the same arguments as in the general case.
\end{proof}

Note that in the case of additive noise, the conditions are independent of $p$. Moreover, letting $p \nearrow 4$, we obtain for $\lambda$ in case (i) the asymptotic condition $\lambda \gamma > 2$, and for cases (ii) and (iii) the condition $\lambda \gamma > 1$. The latter essentially states that the central limit theorem is only valid if the normalised $L^2$-convergence rate in the Law of Large Numbers, see \eqref{cor: LLN}, satisfies $\min\{1, \chi \gamma\}/2 = 1/2$.


\section{Markovian lift on space of Laplace transforms}\label{section markovia lift cm}

\subsection{General framework}

In this section, we provide a Markovian lift based on the representation of the Volterra kernels in terms of their Bernstein measures. Let $V \hookrightarrow H$ be separable Hilbert spaces, and fix a reference Borel measure $\mu$ on $\R_+$. Here and below, we shall always assume that there exist $\delta_*, \eta_* \in \R$ such that
\begin{align}\label{eq: projection condition 1}
      \int_{(0,1]}x^{\delta_*}\,\mu(\d x)<\infty
            \ \text{ and } \
         \int_{(1,\infty)} x^{-\eta_*}\,\mu(\d x)<\infty.
\end{align}
Note that in condition \eqref{eq: projection condition 1}, we may always replace $\delta_\ast, \eta_\ast$ by a larger value. For particular applications, it is feasible to choose $\delta_*$ and $\eta_*$ as small as possible.

Define for $\delta, \eta \in \R$ a two-parameter scale of Hilbert spaces $\mathcal{H}_{\delta, \eta}$ consisting of equivalence classes of measurable functions $y\colon \R_+ \longrightarrow V$ with finite norm
\[
	\vertiii{ y}_{\delta, \eta}^2 = \int_{\R_+} \|y(x)\|_V^2 w_{\delta, \eta}(x)\, \mu(\mathrm{d}x) < \infty
\]
where $w_{\delta, \eta}\colon \R_+ \longrightarrow (0,\infty)$ denotes the weight function
\begin{align}\label{eq: weight function}
		w_{\delta, \eta}(x) = \1_{\{0\}}(x) + \1_{(0,1]}(x)x^{-\delta} + \1_{(1,\infty)}(x)x^{\eta}.
\end{align}
Hence, for $y \in \mathcal{H}_{\delta, \eta}$, the parameter $\delta$ controls the integrability of $y$ at the origin while $\eta$ captures its integrability at infinity. For stochastic Volterra processes, this translates to small time regularity $t \to 0$ captured by $\eta$, and large time decay $t \to \infty$ modelled by $\delta$. Finally, the artificially added term $y(0)$ represented by $\1_{\{0\}}(x)$ in \eqref{eq: weight function} allows us to include, e.g., constant functions $G(t) = G_0$ in \eqref{eq: mild formulation} provided that $\mu(\{0\}) > 0$.

By construction, $(\mathcal{H}_{\delta,\eta})_{\delta,\eta\in\R}$ is a two-parameter scale of Hilbert spaces such that $\mathcal{H}_{\delta,\eta'}\subset\mathcal{H}_{\delta,\eta}$ for $\eta<\eta'$ and $\mathcal{H}_{\delta',\eta}\subset\mathcal{H}_{\delta,\eta}$ for $\delta<\delta'$ densely. On each of the spaces $\mathcal{H}_{\delta,\eta}$, we define the strongly continuous semigroup $(S(t))_{t\geq0}$ of multiplication operators by
\begin{equation*}\label{eq: semigroup_definition}
	S(t)y(x) = \e{-tx}y(x),\qquad y\in\mathcal{H}_{\delta,\eta},\,x\in\R_+.
\end{equation*}
The generator of $(S(t))_{t \geq 0}$ on $\mathcal{H}_{\delta, \eta}$ is given by $\mathcal{A}y(x) = -xy(x)$ with maximal domain $D(\mathcal{A}) = \mathcal{H}_{\delta, \eta + 2} \subset \mathcal{H}_{\delta, \eta}$. The next lemma provides the basic properties of the Markovian lift with focus on Assumption \ref{assumption SEE}.

\begin{lemma}\label{lemma: CM SSE}
    Let $\mu$ be a Borel measure on $\R_+$ satisfying \eqref{eq: projection condition 1}. The following assertions hold:
    \begin{enumerate}
        \item[(a)] Let $\eta' < \eta$. Then $S(t) \in L(\mathcal{H}_{\delta, \eta'}, \mathcal{H}_{\delta, \eta})$ such that
        \[
            \|S(t)\|_{L(\mathcal{H}_{\delta, \eta'}, \mathcal{H}_{\delta, \eta})} \leq \max\{ \mu(\{0\}), 1, C(\eta - \eta')\}^{1/2}\left( 1 + t^{-(\eta - \eta')/2}\right)
        \]
        where the constant is given by $C(\varrho) = 2^{-\varrho} \varrho^{\varrho} \e{-\varrho}$.

        \item[(b)] For each $\delta \geq \delta_*$ and $\eta \geq \eta_*$, $\Xi\colon \mathcal{H}_{\delta, \eta} \longrightarrow V$ is a bounded linear operator, where
        \[
            \Xi y = \int_{\R_+} y(x)\, \mu(\mathrm{d}x).
        \]

        \item[(c)] Let $y \in \mathcal{H}_{\delta, \eta}$ for some $\delta, \eta \in \R$ with $\delta\geq \delta_\ast$. Then
        $G(t) = \int_{\R_+} \e{-xt}y(x)\, \mu(\mathrm{d}x)$ is for $t > 0$ well-defined, and satisfies
        \begin{equation}\label{eq: bound G}
            \|G(t)\|_{V} \lesssim \left( 1 + t^{-(\eta_* - \eta)_+/2} \right) \vertiii{y}_{\delta,\eta}
        \end{equation}
        where $x_+ = \max\{x, 0\}$. Moreover, let $0 < \theta \leq 1$, then for $s,t > 0$
        \[
            \|G(t) - G(s)\|_V \leq (t \wedge s)^{-(\eta_* + \theta - \eta)_+/2}\vertiii{y}_{\delta,\eta}|t-s|^{\theta}.
        \]
    \end{enumerate}
    In particular, Assumption \ref{assumption SEE} is satisfied for
    \begin{align}\label{assumption A completely monontone}
        \delta \geq \delta_*, \ \ \max \{\eta, \eta_*\} \leq \eta' < 1 + \eta, \ \ \mathcal{H} = \mathcal{H}_{\delta, \eta}, \ \ \mathcal{V} = \mathcal{H}_{\delta, \eta'}, \ \ \rho = \frac{(\eta' - \eta)_+}{2}.
    \end{align}
\end{lemma}
\begin{proof}
    Assertion (a) follows from the elementary inequality
    \begin{align}\label{eq: elementary inequality}
        x^\varrho\e{-2xt}\leq 2^{-\varrho}\varrho^\varrho\e{-\varrho}t^{-\varrho}\eqqcolon C(\varrho)t^{-\varrho}
    \end{align}
    where $\varrho>0$, $t>0$, $x\geq0$, and the bound
    \begin{align}
        \vertiii{S(t)y}_{\delta, \eta}^2
        &= \mu(\{0\})\|y(0)\|_{V}^2 + \int_{(0,1]} \e{-xt} \|y(x)\|_V^2 x^{-\delta}\, \mu(\mathrm{d}x)  \notag
        \\ &\qquad + \int_{(1,\infty)} x^{\eta - \eta'}\e{-2xt} \|y(x)\|_V^2 x^{\eta'}\, \mu(\mathrm{d}x) \notag
        \\ &\leq \max\{\mu(\{0\}),\ 1, \ C(\eta - \eta') \}\left( 1 + t^{-(\eta - \eta')}\right) \vertiii{y}_{\delta, \eta'}^2. \label{eq: 2}
    \end{align}
    The second assertion follows from the Cauchy-Schwarz inequality
    \begin{align*}
        \|\Xi y\|_V^2 = \int_{[0,\infty)} \|y(x)\|_V^2 w_{\delta, \eta}(x)^{1/2} \, \frac{\mu(\mathrm{d}x)}{w_{\delta,\eta}(x)^{1/2}}
        \leq \vertiii{y}_{\delta, \eta}^2 \int_{\R_+} \frac{\mu(\mathrm{d}x)}{w_{\delta, \eta}(x)} < \infty
    \end{align*}
    and assumption \eqref{eq: projection condition 1}. In particular, assertions (a) and (b) combined imply that Assumption \ref{assumption SEE} is satisfied for \eqref{assumption A completely monontone}.

    It remains to prove (c). Let $y\in\mathcal{H}_{\delta,\eta}$. From part (a) it is easily seen that $S(t)y\in\mathcal{H}_{\delta_\ast,\eta_\ast}$ for $t > 0$. In particular, $G$ is well-defined and \eqref{eq: bound G} follows from a combination of assertions (a) and (b). For the last inequality, we use
    \begin{align*}
        |\e{-tx}-\e{-sx}|^2
        = \e{- 2x (t\wedge s)} 1 \wedge |1-\e{-(t\vee s - t \wedge s)x}|^2
        \leq x^{2\theta} \e{- 2x (t\wedge s)}|t-s|^{2\theta}
    \end{align*}
    and proceed similarly to the proof of (a), to find
    \begin{align*}
        \norm{G(t)-G(s)}_V^2 &\leq \int_{\R_+}|\e{-tx}-\e{-sx}|^2\norm{y(x)}_V^2w_{\delta_*,\eta_\ast}(x)\,\mu(\d x)\left(\int_{\R_+}\frac{1}{w_{\delta_*,\eta_\ast}}\,\mu(\d x)\right)
        \\ &\lesssim |t-s|^{2\theta} \int_{\R_+} \e{-2(t \wedge s)x}x^{2\theta}\norm{y(x)}_V^2 w_{\delta_*,\eta_\ast}(x)\,\mu(\d x)
        \\ &\lesssim (t \wedge s)^{-(\eta_* + \theta - \eta)_+}\vertiii{y}_{\delta,\eta}^2|t-s|^{2\theta}.
    \end{align*}
    This proves all assertions.
\end{proof}

Below we proceed to verify Assumptions \ref{assumption long-time}.(b) and (c) under the structural condition~\eqref{eq: timehomogeneous}. Let $H_b, H_{\sigma}$ be separable Hilbert spaces such that $H \hookrightarrow H_b, H_{\sigma}$, and suppose that the Volterra kernels $E_b, E_{\sigma}$ have the representation
\begin{align}\label{eq: kernel representation completely monotone}
		E_b(t) = \int_{\R_+}\e{-xt}\xi_b(x)\,\mu(\d x) \ \text{ and }\
		E_{\sigma}(t) = \int_{\R_+}\e{-xt}\xi_\sigma(x)\,\mu(\d x)
\end{align}
where $\xi_b\colon \R_+ \longrightarrow L(H_b, V)$ and $\xi_{\sigma}\colon \R_+ \longrightarrow L_q(H_\sigma, V)$ with some $q\in[1,\infty]$ are strongly measurable such that both integrals are absolutely convergent for each $t > 0$. Remark that, if $V = H = H_b = H_{\sigma}$ and $\xi_b, \xi_{\sigma}$ are $\mu$-a.e.~symmetric and positive-semidefinite, then $E_b, E_{\sigma}$ are completely monotone in the sense of \cite{MR2798103}. However, we do not suppose this condition, i.e., also not necessarily completely monotone kernels are allowed.

For given $\xi_b, \xi_{\sigma}$, suppose there exist $\delta^*, \eta^* \in \R$ such that
\begin{align} \label{eq: kernels condition 1}
    \begin{cases}
    \quad \int_{(0,1]} \left(\norm{\xi_b(x)}_{L(H_b,V)}^2 + \norm{\xi_\sigma(x)}_{L_q(H_{\sigma}, V)}^2\right)x^{-\delta^*}\,\mu(\d x)<\infty,
    \\ \quad \int_{(1,\infty)}\left(\norm{\xi_b(x)}_{L(H_{b}, V)}^2+\norm{\xi_\sigma(x)}_{L_q(H_{\sigma},V)}^2\right)x^{\eta^*}\,\mu(\d x)<\infty.
    \end{cases}
\end{align}
As for $\delta_*, \eta_*$, also here we are typically interested in the largest possible choice for $\delta^*, \eta^*$. Let us define for $a\in\{b,\sigma\}$ the action of $\xi_a$ on $h\in H_a$ via $(\xi h)(x)=\xi(x)h$ where $x\geq0$. Then we obtain $\vertiii{\xi_b h}_{\delta^*,\eta^*}^2\leq \norm{h}_{H_b}^2\int_{\R_+}\norm{\xi_b(x)}_{L(H_b,V)}^2 w_{\delta^*,\eta^*}(x)\,\mu(\d x)$ and similarly $\vertiii{\xi_{\sigma} h}_{\delta^*,\eta^*}^2\leq \norm{h}_{H_{\sigma}}^2\int_{\R_+}\norm{\xi_{\sigma}(x)}_{L_q(H_{\sigma},V)}^2 w_{\delta^*,\eta^\ast}(x)\,\mu(\d x)$. This shows that under condition \eqref{eq: kernels condition 1} one has
\[
    \xi_b\in L(H_b,\mathcal{H}_{\delta^*,\eta^*}) \ \text{ and } \ \xi_{\sigma} \in L_q(H_{\sigma}, \mathcal{H}_{\delta^*, \eta^*}).
\]
Below, we summarise the properties of this lift with particular focus on Assumption \ref{assumption long-time}.

\begin{theorem}\label{theorem cm lift}
    Let $\mu$ be a Borel measure on $\R_+$ satisfying \eqref{eq: projection condition 1} and let $\xi_b, \xi_{\sigma}$ satisfy \eqref{eq: kernels condition 1}.
    \begin{enumerate}
        \item[(a)] For any $\delta, \eta \in \R$, it holds that
        $S(t) \longrightarrow S_\infty$ strongly on $\mathcal{H}_{\delta, \eta}$, where $S_\infty$ is given by
        \begin{equation*}\label{eq: sg limit}
        S_\infty\colon\mathcal{H}_{\delta,\eta}\longrightarrow\mathcal{H}_{\delta,\eta},\quad S_\infty y = \begin{cases}
        y(0)\1_{\{0\}}(\cdot), &\text{if }\mu(\{0\})>0\\
        0, &\text{if }\mu(\{0\})=0.
        \end{cases}
        \end{equation*}
        Moreover, if $\delta' < \delta$, then also
        \[
            \| S(t) - S_\infty \|_{L(\mathcal{H}_{\delta, \eta}, \mathcal{H}_{\delta', \eta})} \leq \max\{1,\ C(\delta - \delta')\}^{1/2} (1 \vee t)^{-\frac{\delta - \delta'}{2}}.
        \]

        \item[(b)] If $\mu(\{0\})>0$, then $\xi_b(0)=0$, $\xi_\sigma(0)=0$. Let $\delta^* - \delta_* > 2$, then Assumptions \ref{assumption long-time} (b) and (c) are satisfied for
        \begin{align}\label{eq: 1}
            \mathcal{H} = \mathcal{H}_{\delta, \eta^*}, \ \ \mathcal{V} = \mathcal{H}_{\delta, \eta}, \ \ \mathcal{V}_0 = \mathcal{H}_{\delta_*,\eta}, \ \, \lambda=\frac{\delta-\delta_*}{2}, \ \ \rho =  \frac{(\eta - \eta^*)_+}{2}
        \end{align}
        and any $\max\{\eta_*, \eta^*\} \leq \eta < 1 + \eta^*$, $\delta\in (\delta_*,\delta^*)$ such that $\delta^* - \delta > 2$.

        \item[(c)] Suppose that $b \equiv 0$ and $\xi_b \equiv 0$. If $\mu(\{0\})>0$, assume that $\xi_\sigma(0)=0$. If $\delta^* - \delta_* > 1$, then Assumptions \ref{assumption long-time} (b) and (c) hold with \eqref{eq: 1} where $\max\{\eta_*, \eta^*\} \leq \eta < 1 + \eta^*$ and $\delta \in (\delta_*, \delta^*)$ is such that $\delta^* - \delta > 1$.
    \end{enumerate}
\end{theorem}
\begin{proof}
    (a) The strong convergence $S(t) \longrightarrow S_\infty$ on $\mathcal{H}_{\delta, \eta}$ follows from
    \[
        \vertiii{S(t)y - S_\infty y}_{\delta, \eta}^2 = \int_{(0,\infty)}\e{-2xt} \|y(x)\|_V^2 w_{\delta, \eta}(x)\, \mu(\mathrm{d}x)
    \]
    and an application of dominated convergence. For the second assertion, let $\delta' < \delta$. Then using \eqref{eq: elementary inequality}, we obtain for $y\in\mathcal{H}_{\delta,\eta}$
    \begin{align*}
        \vertiii{S(t)y - S_\infty y}_{\delta',\eta}^2 &= \int_{(0,\infty)} \e{-2xt}\|y(x)\|_V^2 w_{\delta', \eta}(x) \,  \mu(\mathrm{d}x)
        \\ &\leq C(\delta - \delta') t^{-(\delta - \delta')}\int_{(0,1]}\norm{y(x)}_V^2 x^{-\delta}\,\mu(\d x)
         + \e{-2t}\int_{(1,\infty)}\norm{y(x)}_V^2x^{\eta}\,\mu(\d x)
        \\ &\leq C(\delta - \delta')t^{-(\delta - \delta')} \vvvert y \vvvert_{\delta, \eta}^2
    \end{align*}
    where we have used $\e{-2t} \leq C(\delta - \delta') t^{-(\delta - \delta')}$ with $C$ defined in \eqref{eq: elementary inequality}. For $t \in [0,1]$ we may also use the trivial bound $\vertiii{S(t)y - S_\infty y}_{\delta',\eta} \leq \vertiii{y}_{\delta', \eta} \leq \vertiii{y}_{\delta, \eta}$. Combining both bounds proves the assertion.

    (b) Since $(S(t))_{t\geq0}$ is a contraction semigroup, its operator norm is uniformly bounded. By assumption $\xi_b(0) = 0$, $\xi_{\sigma}(0) = 0$ if $\mu(\{0\}) > 0$, we find $S_\infty S(t)\xi_b=0$ and $S_\infty S(t)\xi_\sigma=0$ for $t>0$. In particular, an application of (a) gives
    \begin{align}\label{eq: estimate S(t) xi_b}
    \norm{S(t)\xi_b}_{L(H_b,\mathcal{H}_{\delta,\eta})}
    &\leq \norm{S(t/2) - S_\infty}_{L(\mathcal{H}_{\delta^*,\eta},\mathcal{H}_{\delta,\eta})}\norm{S(t/2)\xi_b}_{L(H_b,\mathcal{H}_{\delta^*,\eta})}
    \\ &\leq 2\max\{\mu(\{0\}),\ 1, \ C(\eta - \eta^*) \} C(\delta^* - \delta)(1\lor t)^{-\frac{\delta^*-\delta}{2}} \norm{\xi_b}_{L(H_b,\mathcal{H}_{\delta^*,\eta^*})} \notag
    \end{align}
    where we have used
    \[
        \norm{S(t/2)\xi_b}_{L(H_b,\mathcal{H}_{\delta^*,\eta})} \leq \max\{\mu(\{0\}),\ 1, \ C(\eta - \eta^*) \} ( 1 + (1 \wedge t)^{- \frac{\eta - \eta^*}{2}})\norm{\xi_b}_{L(H_b,\mathcal{H}_{\delta^*,\eta^*})},
    \]
    which follows similarly to \eqref{eq: 2} when taking into account $S_{\infty}S(t/2)\xi_b = 0$ so that the first term vanishes. Similarly we prove
    \begin{multline*}\label{eq: estimate S(t) xi_sigma}
        \norm{S(t)\xi_\sigma}_{L_q(H_\sigma,\mathcal{H}_{\delta,\eta})}
        \\   \leq 2\max\{\mu(\{0\}),\ 1, \ C(\eta - \eta^*) \} C(\delta^* - \delta)(1\lor t)^{-\frac{\delta^*-\delta}{2}} \norm{\xi_{\sigma}}_{L_q(H_{\sigma},\mathcal{H}_{\delta^*,\eta^*})}.
    \end{multline*}
    This proves \eqref{eq:assumption long-timme - integrability} for the general case.

    (c) When $b = 0$ and $\xi_b = 0$, then we only need that $\norm{S(\cdot)\xi_\sigma}_{L_q(H_\sigma,\mathcal{H}_{\delta,\eta})}^2$ is integrable, whence $\delta^* - \delta > 1$ is sufficient, which is possible whenever $\delta^* - \delta_* > 1$.
\end{proof}

This theorem allows us to study the Markovian lift \eqref{eq: abstract mild formulation Markovian lift} for various stochastic Volterra equations \eqref{eq: mild formulation}, provided that the kernels admit the representation \eqref{eq: kernel representation completely monotone}, i.e.~$E_b(t) = \Xi S(t)\xi_b$ and $E_{\sigma}(t) = \Xi S(t)\xi_{\sigma}$, such that \eqref{eq: kernels condition 1} holds. In such a case the class of admissible $G$ is given by all functions of the form $G(t) = \Xi S(t)\xi$, i.e.
\[
    G(t) = \int_{\R_+} \e{-xt}\xi(x)\, \mu(\mathrm{d}x), \qquad t > 0,
\]
where $\xi \in L^p(\Omega, \F_0, \P; \mathcal{V})$ with $\mathcal{V} = \mathcal{H}_{\delta^*, \eta}$ and $\max\{\eta_*, \eta^*\} \leq \eta < 1 + \eta^*$. Remark that, if $\mu(\{0\}) = 0$, then $S_\infty = 0$, and hence limit distributions are necessarily unique. We will see that for the mild formulation \eqref{VSPDE}, this is typically the case whenever the underlying Volterra kernels $k,h$ are not integrable at $t = \infty$.

Concerning condition \eqref{eq: kernels condition 1}, the following remark illustrates how we may obtain new Volterra kernels via regularisation in short or long time.
\begin{remark}\label{remark: regularization}
    Let $E(t) = \int_{\R_+} \e{-tx}\xi(x)\, \mu(\mathrm{d}x)$ satisfy \eqref{eq: kernels condition 1}, and define for $\varepsilon, \lambda > 0$
    \[
        E^{\varepsilon}(t) \coloneqq E(t+\varepsilon) \ \text{ and } \ E^{\lambda}(t) \coloneqq \e{-\lambda t}E(t).
    \]
    Then $E^{\varepsilon}(t) = \Xi S(t)\xi^{\varepsilon}$ with $\xi^{\varepsilon}(x) = \e{-\varepsilon x}\xi(x)$ satisfies \eqref{eq: kernels condition 1} for any choice of $\eta^* \in \R$, and $E^{\lambda}(t) = \Xi S(t)\xi^{\lambda}$ with $\xi^{\lambda}(x) = \1_{(\lambda,\infty)}(x)\xi(x - \lambda)$ satisfies \eqref{eq: kernels condition 1} for any choice of $\delta^* \in \R$.
\end{remark}

Next, let us observe that for completely monotone Volterra kernels, we may always obtain the representation \eqref{eq: kernel representation completely monotone}.

\begin{remark}
    Suppose that $E_b(t) = \int_{\R_+} \e{-xt}\,\nu_b(\mathrm{d}x)$ and $E_{\sigma}(t) = \int_{\R_+}\e{-xt}\,\nu_{\sigma}(\mathrm{d}x)$ are scalar-valued and completely monotone kernels with Bernstein measures $\nu_b, \nu_{\sigma}$. Let $\widetilde{E}_{b} \in L(H_b, V)$, $\widetilde{E}_{\sigma} \in L_q(H_{\sigma}, V)$, and define $\mu(\mathrm{d}x) = \nu_b(\mathrm{d}x) + \nu_{\sigma}(\mathrm{d}x)$. Then $E_b, E_{\sigma}$ have representation \eqref{eq: kernel representation completely monotone} with
    \[
        \xi_b(x) = \widetilde{E}_b \frac{\nu_b(\mathrm{d}x)}{\mu(\mathrm{d}x)} \ \text{ and } \
        \xi_{\sigma}(x) = \widetilde{E}_{\sigma}\frac{\nu_{\sigma}(\mathrm{d}x)}{\mu(\mathrm{d}x)}.
    \]
\end{remark}

While the above remark guarantees that we may always find a reference measure $\mu$, in many cases, one may take the Lebesgue measure $\mu(\mathrm{d}x) = \mathrm{d}x$. In such a case, \eqref{eq: projection condition 1} is satisfied for any $\eta_* > 1$ and $\delta_* > -1$. Moreover, it is clear that in the above remark $\widetilde{E}_{b} \in L(H_b, V)$, $\widetilde{E}_{\sigma} \in L_q(H_{\sigma}, V)$ may also depend on $x$. Let us illustrate this with two particular examples of kernels where Assumption \ref{assumption SEE} is satisfied.

\begin{example}\label{example: fractional}
    Let $k(t) = \int_{\R_+} \e{-xt}\xi(x)\, \mathrm{d}x$ where $\xi$ is specified below. Then $\mu(\mathrm{d}x) = \mathrm{d}x$ and \eqref{eq: projection condition 1} is satisfied for any $\eta_* > 1$ and $\delta_* > -1$.
    \begin{enumerate}
        \item[(a)] Take the fractional Riemann-Liouville kernel $k_{\alpha}(t) = \frac{t^{\alpha - 1}}{\Gamma(\alpha)}$ with $\alpha \in (0,1)$. Then
        \[
            \xi(x) = \frac{x^{-\alpha}}{\Gamma(\alpha)\Gamma(1-\alpha)}
        \]
        and we may choose any $\eta^* < -1 + 2\alpha$ and $\delta^* < 1 - 2\alpha$.

        \item[(b)] Take the $\log$-kernel $k(t) = \log(1 + 1/t)$. Then
        \[
            \xi(x) = \frac{1 - \e{-x}}{x},
        \]
        and we may choose $\eta^* < 1$ and $\delta^* < 1$.
    \end{enumerate}
    In both cases, choosing $k_b(t) = k(t) \widetilde{E}_b$ and $k_{\sigma}(t) = h(t)\widetilde{E}_{\sigma}$, with $k,h$ given as in (a) or (b), it is clear that Assumption \ref{assumption SEE} is satisfied. However, Assumption \ref{assumption long-time} is not satisfied in case (a) since $\delta^* - \delta_* < 1$, while in case (b) we may choose $\delta^*, \delta_*$ such that $\delta^* - \delta_* \in (1,2)$, whence Assumptions \ref{assumption long-time} (b) and (c) are satisfied when $b \equiv 0$.
\end{example}

For Assumption \ref{assumption long-time}, integrability of the Volterra kernels is essential and could, e.g., be achieved by the regularisation procedure given in Remark \ref{remark: regularization}.

\begin{example}
    Suppose that the reference measure is given by
    \[
        \mu(\mathrm{d}x) = \sum_{n=0}^{\infty} c_n \delta_{\lambda_n}(\mathrm{d}x)
    \]
    where $c_n, \lambda_n \geq 0$, $(\lambda_n)_{n \geq 1}$ is increasing such that $\mu$ is locally finite. Then each $E(t) = \Xi S(t)\xi$ is of the form
    \[
        E(t) = \sum_{n=0}^{\infty}c_n \e{-\lambda_n t}\xi(\lambda_n).
    \]
    Condition \eqref{eq: projection condition 1} is satisfied whenever
    \[
        \sum_{n=0}^{\infty} \left(\1_{\{\lambda_n \leq 1\}}c_n \lambda_n^{\delta_*} + \1_{\{\lambda_n > 1\}}c_n \lambda_n^{-\eta_*} \right) < \infty
    \]
    while condition \eqref{eq: kernels condition 1} holds, provided that
    \[
        \sum_{n=0}^{\infty}\left( \1_{\{\lambda_n \leq 1\}} c_n^2 \| \xi(\lambda_n)\|^2 \lambda_n^{-\delta^*} + \1_{\{\lambda_n > 1\}}c_n^2 \|\xi(\lambda_n)\|^2 \lambda_n^{\eta^*} \right) < \infty
    \]
    where the norm $\|\xi(\lambda_n)\|$ depends on the spaces $E$ is acting on.
\end{example}

For such examples, the Markovian lift can be written as an infinite system of stochastic equations. The latter arises in the study of (finite-dimensional) Markovian approximations.

\subsection{Fractional Volterra kernels in the mild formulation}\label{sec:frac_kernels_mild}

In this section, we discuss the particular case where the stochastic Volterra equation \eqref{VSPDE}, and the corresponding Markovian lift is carried out for \eqref{eq: mild formulation} with resolvent operators given by \eqref{eq: resolvent equation}. Below, we focus on the case of fractional Volterra kernels under the assumption that $(A,D(A))$ admits an orthonormal basis $(e_n^H)_{n \in \N}$ of eigenvectors such that
\begin{align}\label{eq: A diagonal}
        Ae_n^H = - \theta_n e_n^H, \qquad n \geq 1,
\end{align}
where the sequence of eigenvalues $(\theta_n)_{n \geq 1} \subset (0,\infty)$ increases to infinity. Let $W$ be a Gaussian process with covariance operator $Q = \sum_{n=1}^{\infty} \lambda_n (e_n^H \otimes e_n^H)$, where $(\lambda_n)_{n \geq 1}$ denotes the corresponding sequence of eigenvalues, and $W$ has for $x \in H$ representation
\begin{align}\label{eq: W}
    W_t(x) = \sum_{n = 1}^{\infty}\sqrt{\lambda_n} \beta_n(t)\langle x, e_n^H\rangle_{H}\, e_n^H
\end{align}
where $(\beta_n)_{n \geq 1}$ is a sequence of independent one-dimensional Brownian motions. Remark that, if  $(\lambda_n)_{n \geq 1}$ is summable, then $Q$ is trace-class and hence $W$ is a genuine $Q$-Wiener process. However, if $\lambda = 1$ then $W$ is a cylindrical Wiener process. Let us consider the stochastic Volterra equation, for simplicity, with additive noise, given by
\begin{equation*}\label{eq: cm general eq add}
    u(t) = g(t) + \int_0^t \frac{(t-s)^{\alpha - 1}}{\Gamma(\alpha)}\left( Au(s) + b(u(s))\right)\, \d s + \int_0^t \frac{(t-s)^{\beta - 1}}{\Gamma(\beta)}\, \d W_s
\end{equation*}
where $b\colon H \longrightarrow H_b$ is Lipschitz continuous, $\alpha \in (0,1)$ and $1/2 < \beta < 1 + \alpha$, and we implicitly assume that all integrals are well-defined. To study this equation in its mild formulation, let us define the family of operators $E^{\alpha, \beta}$ determined as the unique solution of \eqref{eq: resolvent equation} with $k(t) = t^{\alpha-1}/\Gamma(\alpha)$ and $h(t) = t^{\beta - 1}/\Gamma(\beta)$, i.e.
\begin{equation*}\label{fractional resolvent equation}
            E^{\alpha, \beta}(t)x = \frac{t^{\beta-1}}{\Gamma(\beta)}x + A \int_0^t \frac{(t-s)^{\alpha-1}}{\Gamma(\alpha)}E^{\alpha, \beta}(s)x\, \mathrm{d}s, \qquad x \in D(A).
\end{equation*}
The next remark provides an explicit formula for $E^{\alpha, \beta}$ and shows that  \eqref{eq: kernel representation completely monotone} is satisfied.

\begin{remark}\label{remark resolvents frac}
For $n \geq 1$, let $e_n(\cdot; \alpha, \beta)$ be the unique solution of the one-dimensional Volterra equation
\[
    e_n(t;\alpha, \beta) = \frac{t^{\beta - 1}}{\Gamma(\beta)} - \theta_n \int_0^t \frac{(t-s)^{\alpha - 1}}{\Gamma(\alpha)}e_n(s; \alpha, \beta)\, \mathrm{d}s.
\]
Taking Laplace transforms, one can verify that the unique solution is given by
\begin{equation*}\label{eq: 1dim fractional resolvent}
    e_n(t;\alpha, \beta) = t^{\beta - 1}E_{\alpha,\beta}(-\theta_n t^{\alpha})
\end{equation*}
where $E_{\alpha,\beta}$ denotes the two parameter Mittag-Leffler function. Furthermore, by an application of \cite[Lemma 2.1]{MR1920979}, we find $e_n(t;\alpha, \beta) = \int_{\R_+}\e{-xt}\xi_{\alpha,\beta}(x; \theta_n)\, \mathrm{d}x$ whenever $\alpha \in (0,1)$ and $\beta \in (0,\alpha + 1)$. Since $(A, D(A))$ satisfies \eqref{eq: A diagonal}, it follows that
\begin{align*}
           E^{\alpha, \beta}(t) = \int_{\R_+} \e{-xt} \xi^{\alpha,\beta}(x)\, \mathrm{d}x = \sum_{n=1}^{\infty}e_n(t; \alpha, \beta) (e_n^H \otimes e_n^H),
\end{align*}
where $\xi^{\alpha,\beta}(x) = \sum_{n=1}^{\infty} \xi_{\alpha,\beta}(x;\theta_n) (e_n^H \otimes e_n^H)$ and
\begin{align}\label{eq: xi mun}
            \xi_{\alpha,\beta}(x; \theta_n) = \frac{1}{\pi} \frac{x^{2\alpha-\beta}\sin(\beta\pi)- \theta_n x^{\alpha-\beta}\sin((\alpha-\beta)\pi)}{\theta_n^2 + 2\theta_n \cos(\pi\alpha)x^\alpha+x^{2\alpha}}.
\end{align}
\end{remark}

Hence, setting $E_b = E^{\alpha, \alpha}$, $E_{\sigma} = E^{\alpha, \beta}$, and writing $W = Q^{1/2}\widetilde{W}$ with a cylindrical Wiener process $\widetilde{W}$ on $H$, we obtain the desired mild formulation \eqref{eq: mild formulation}, i.e.
\begin{align}\label{eq: mild formulation additive noise}
        u(t;G) &=  G(t) + \int_0^t E^{\alpha,\alpha}(t-s)b(u(s;G))\,\d s + \int_0^t E^{\alpha,\beta}(t-s) Q^{\frac{1}{2}}\,\d \widetilde{W}_s,
        \\ \notag G(t) &= g(t) + A \int_0^t E^{\alpha, \alpha}(t-s)g(s)\, \d s.
\end{align}
Next, we formulate the corresponding Markovian lift and verify our main assumptions \ref{assumption SEE} and \ref{assumption long-time} in the scale of weighted Hilbert spaces
\begin{equation}\label{eq:H-kappa space}
        H^{\varkappa} = \left\{ h \in H \ : \ \sum_{n=1}^{\infty}\theta_n^{2\varkappa}\ |\langle h, e_n^H \rangle_H|^2 < \infty \right\}, \qquad \varkappa \in \R,
\end{equation}
with inner product $\langle h,h'\rangle_\varkappa =\sum_{n=1}^\infty \theta^{2\varkappa} \langle h,e_n^H\rangle_H\langle h',e_n^H\rangle_H$ and induced norm $\|\cdot\|_\varkappa$. Note that $H^\varkappa\subset H^{\varkappa'}$ for $\varkappa'<\varkappa$. Recall that \eqref{eq: kernels condition 1} depends on the choice of $q$. Below we focus on the cases $q \in \{2,\infty\}$ for which $\xi^{\alpha,\beta} \colon\R_+ \longrightarrow L_q(H^{\varkappa}, H^{\gamma})$.

\begin{lemma}\label{theorem cm frac spaces}
    Suppose that \eqref{eq: A diagonal} holds. Let $\alpha \in (\frac12, 1)$, $\frac{1}{2} < \beta < \frac{1}{2} + \alpha$, and take $\delta^*, \eta^* \in \R$ such that $\delta^* < 1+2\alpha-2\beta\1_{\{\alpha\neq\beta\}}$ and $\eta^* < 2\beta - 1$. Then we obtain
    \begin{align*}
            \int_{\R_+}  \| \xi^{\alpha, \beta}(x)\|_{L(H)}^2 w_{\delta^*, \eta^*}(x)\, \mathrm{d}x
            &\leq \frac{\left( \frac{|\sin(\beta \pi)|}{\pi \sin(\alpha \pi)^2} \theta_1^{-2} + \frac{|\sin((\alpha - \beta)\pi)|}{\pi \sin(\alpha \pi)^2 } \theta_1^{-1}\right)^2}{1 + 2\alpha - 2\beta \1_{\{\alpha \neq \beta\}} - \delta^*}
        \\ &\qquad + \frac{\left( \frac{|\sin(\beta \pi)|}{ \pi \sin(\alpha \pi)^2} +  \frac{|\sin((\alpha-\beta)\pi)|}{2\pi (1+ \cos(\alpha \pi))}  \right)^2}{2\beta - 1 - \eta^*}
    \end{align*}
    and if $\gamma, \varkappa \in \R$, then also
        \begin{align*}
            \int_{\R_+} \| &\xi^{\alpha, \beta}(x)\|_{L_2(H^{\varkappa}, H^{\gamma})}^2 w_{\delta^*, \eta^*}(x)\, \mathrm{d}x
            \\ &\leq \left( \frac{|\sin(\beta \pi)|}{\pi \sin(\alpha \pi)^2} \theta_1^{-1} + \frac{|\sin((\alpha - \beta)\pi)|}{\pi \sin(\alpha \pi)^2 } \right)^2  \sum_{n=1}^{\infty}\frac{\theta_n^{2(\gamma - \varkappa)-2}}{1 + 2\alpha - 2\beta \1_{\{\alpha \neq \beta\}} - \delta^*}
            \\ &\qquad + \left( \frac{|\sin(\beta \pi)|}{\pi \sin(\alpha \pi)^2} \theta_1^{-2} + \frac{|\sin((\alpha-\beta) \pi)|}{\pi \sin(\alpha \pi)^2} \right)^2 \sum_{n=1}^{\infty} \frac{\theta_n^{2(\gamma-\varkappa)+\frac{-2\beta+\eta^\ast+1}{\alpha}}+\theta_n^{2(\gamma-\varkappa)-2}}{1 + 2\alpha - 2\beta\1_{\{\alpha \neq \beta\}} + \eta^*}
            \\ &\qquad +  \frac{\left(\frac{|\sin(\beta \pi)|}{\pi \sin(\alpha \pi)^2} + \frac{|\sin((\alpha-\beta)\pi)|}{2\pi (1 + \cos(\alpha \pi))}\right)^2}{2\beta - 1 - \eta^*} \sum_{n=1}^{\infty}\theta_n^{2(\gamma - \varkappa) + \frac{- 2\beta+\eta^\ast+1}{\alpha}}.
        \end{align*}
\end{lemma}
\begin{proof}
    Define $f(\theta) = \theta^2 + 2\theta \cos(\alpha \pi) + x^{2\alpha}$ with $x$ fixed. Firstly, when $x \in (0,1]$, we obtain from $f(\theta_n) = (\cos(\alpha \pi)\theta_n + x^{\alpha})^2 + \sin(\alpha \pi)^2 \theta_n^2 \geq \sin(\alpha \pi)^2 \theta_n^2$ the bound
    \begin{align}\label{eq: 4}
        |\xi_{\alpha, \beta}(x; \theta_n)| &\leq  \frac{|\sin(\beta \pi)| \theta_n^{-2} x^{2\alpha - \beta} + |\sin((\alpha - \beta)\pi)| \theta_n^{-1} x^{\alpha - \beta}}{\pi \sin(\alpha \pi)^2}.
    \end{align}
    When $x > 1$, we obtain
    \begin{align}\notag
        |\xi_{\alpha, \beta}(x; \theta_n)| &\leq \frac{1}{\pi} \frac{x^{2\alpha-\beta}|\sin(\beta\pi)|}{\theta_n^2 + 2\theta_n \cos(\pi\alpha)x^\alpha+x^{2\alpha}} + \frac{1}{\pi} \frac{\theta_n x^{\alpha-\beta}|\sin((\alpha-\beta)\pi)|}{\theta_n^2 + 2\theta_n \cos(\pi\alpha)x^\alpha+x^{2\alpha}}
        \\ &\leq \frac{|\sin(\beta \pi)|}{\pi \sin(\alpha \pi)^2} x^{-\beta} + \frac{|\sin((\alpha - \beta)\pi)|}{2\pi (1 + \cos(\alpha \pi))}x^{-\beta} \label{eq: 5}
    \end{align}
    where the first term follows from $f(\theta) \geq x^{2\alpha} \sin(\alpha \pi)^2$, while the second term can be obtained by maximising $g(\theta) = \frac{\theta}{f(\theta)}$ at $\theta = x^{\alpha}$.

    To prove the desired inequality with respect to the operator norm, we use \eqref{eq: 4} and \eqref{eq: 5} to find
    \begin{align*}
        &\ \int_0^1 \|\xi^{\alpha,\beta}(x)\|_{L(H)}^2 x^{-\delta^*}\, \mathrm{d}x + \int_1^\infty \|\xi_{\sigma}(x)\|_{L(H)}^2 x^{\eta^*}\, \mathrm{d}x
        \\ &= \int_0^{1} \sup_{n \geq 1}|\xi_{\alpha,\beta}(x; \theta_n)|^2 x^{-\delta^*}\, \mathrm{d}x + \int_1^{\infty} \sup_{n \geq 1} |\xi_{\alpha, \beta}(x; \theta_n)|^2 x^{\eta^*}\, \mathrm{d}x
        \\ &\leq \int_0^{1} \left(\frac{|\sin(\beta \pi)|}{\pi \sin(\alpha \pi)^2} \theta_1^{-2} x^{2\alpha - \beta} + \frac{|\sin((\alpha - \beta)\pi)|}{\pi \sin(\alpha \pi)^2 } \theta_1^{-1} x^{\alpha - \beta} \right)^2 x^{-\delta^*}\, \mathrm{d}x
        \\ &\qquad + \int_1^{\infty} \left( \frac{|\sin(\beta \pi)|}{ \pi \sin(\alpha \pi)^2} +  \frac{|\sin((\alpha-\beta)\pi)|}{2\pi (1+ \cos(\alpha \pi))}  \right)^2x^{-2\beta + \eta^*}\, \mathrm{d}x
        \\ &= \left(\frac{\sin(\beta \pi)}{\pi \sin(\alpha \pi)^2}\right)^2 \frac{\theta_1^{-4}}{1 + 4\alpha - 2\beta - \delta^*} + \frac{|\sin(\beta \pi)| |\sin((\alpha-\beta)\pi)|}{\pi^2 \sin(\alpha \pi)^4}  \frac{2\theta_1^{-3}}{1 + 3\alpha - 2\beta - \delta^*}
        \\ &\qquad + \left(\frac{\sin((\alpha - \beta)\pi)}{\pi \sin(\alpha \pi)^2 }\right)^2 \frac{\theta_1^{-2}}{1 + 2\alpha - 2\beta \1_{\{\alpha \neq \beta\}} - \delta^*}
        \\ &\qquad + \left( \frac{|\sin(\beta \pi)|}{ \pi \sin(\alpha \pi)^2} +  \frac{|\sin((\alpha-\beta)\pi)|}{2\pi (1+ \cos(\alpha \pi))}  \right)^2 \frac{1}{2\beta - 1 - \eta^*}
        \\ &\leq \left( \frac{|\sin(\beta \pi)|}{\pi \sin(\alpha \pi)^2} \theta_1^{-2} + \frac{|\sin((\alpha - \beta)\pi)|}{\pi \sin(\alpha \pi)^2 } \theta_1^{-1}\right)^2 \frac{1}{1 + 2\alpha - 2\beta\1_{\{\alpha \neq \beta\}} - \delta^*}
        \\ &\qquad + \left( \frac{|\sin(\beta \pi)|}{ \pi \sin(\alpha \pi)^2} +  \frac{|\sin((\alpha-\beta)\pi)|}{2\pi (1+ \cos(\alpha \pi))}  \right)^2 \frac{1}{2\beta - 1 - \eta^*}.
    \end{align*}
    Remark that this inequality also remains valid when $\alpha = \beta$ and $\delta^* = 1 + \alpha$ since then $\sin((\alpha - \beta)\pi) = 0$ so that the additional cross-terms actually vanish. This proves the first assertion.

    Next, we prove the inequality for the Hilbert-Schmidt norm. Since $(e_n^{H^{\varkappa}})_{n\geq1}=(\theta_n^{-\varkappa}e_n^H)_{n\geq1}$ is an orthonormal basis of $H^{\varkappa}$, we obtain
    \begin{multline*}
        \int_0^{\infty}\norm{\xi^{\alpha,\beta}(x)}_{L_2(H^{\varkappa},H^{\gamma})}^2 w_{\delta^\ast, \eta^\ast}(x)\, \mathrm{d}x\\
        = \sum_{n=1}^\infty \theta_n^{2(\gamma-\varkappa)} \int_{0}^{1} |\xi_{\alpha,\beta}(x;\theta_n)|^2 x^{-\delta^\ast}\, \mathrm{d}x + \sum_{n=1}^\infty \theta_n^{2(\gamma-\varkappa)} \int_{1}^{\infty} |\xi_{\alpha,\beta}(x;\theta_n)|^2 x^{\eta^\ast}\, \mathrm{d}x. \label{eq:xi L2 estimate}
    \end{multline*}
    It remains to bound the last two integrals. For small $x$ we obtain from \eqref{eq: 4} with a similar computation to above, for each $\delta^\ast< [1+2\alpha]\1_{\{\alpha=\beta\}} + [1 + 2(\alpha - \beta)]\1_{\{\alpha\neq\beta\}}$,
    \begin{align*}
         \int_0^1 |\xi_{\alpha,\beta}(x;\theta_n)|^2 x^{-\delta^\ast}\,\d x
        &\leq \int_0^1 \left(\frac{|\sin(\beta \pi)|}{\pi \sin(\alpha \pi)^2} \theta_n^{-2} x^{2\alpha - \beta} + \frac{|\sin((\alpha - \beta)\pi)|}{\pi \sin(\alpha \pi)^2 } \theta_n^{-1} x^{\alpha - \beta} \right)^2\, x^{- \delta^*} \d x
        \\ &\leq \left( \frac{|\sin(\beta \pi)|}{\pi \sin(\alpha \pi)^2} \theta_n^{-2} + \frac{|\sin((\alpha - \beta)\pi)|}{\pi \sin(\alpha \pi)^2 } \theta_n^{-1}\right)^2 \frac{1}{1 + 2\alpha - 2\beta \1_{\{\alpha \neq \beta\}} - \delta^*}
        \\ &\leq \left( \frac{|\sin(\beta \pi)|}{\pi \sin(\alpha \pi)^2} \theta_1^{-1} + \frac{|\sin((\alpha - \beta)\pi)|}{\pi \sin(\alpha \pi)^2 } \right)^2 \frac{\theta_n^{-2}}{1 + 2\alpha - 2\beta \1_{\{\alpha \neq \beta\}} - \delta^*}.
    \end{align*}
    For the second integral, the second part of inequality \eqref{eq: 5} does not give the correct asymptotics with respect to $\theta_n$. Thus, let us first note that $\xi_{\alpha,\beta}(x; \theta_n)$ satisfies the scaling property $\xi_{\alpha,\beta}(x; \theta_n) = \theta_n^{-\beta / \alpha }\xi_{\alpha,\beta}\left(x\theta_n^{-1/\alpha};1\right)$. Then we obtain
    \begin{align*}
        \int_1^\infty |\xi_{\alpha,\beta}(x;\theta_n)|^2x^{\eta^\ast}\,\d x
        &= \int_1^\infty \theta_n^{- 2\beta/\alpha}|\xi_{\alpha,\beta}\left(x\theta_n^{-1/\alpha};1\right)|^2 x^{\eta^\ast}\,\d x
        \\ &= \theta_n^{\frac{-2\beta+\eta^\ast+1}{\alpha}}\left( \int_{\theta_n^{-1/\alpha}}^1 |\xi_{\alpha,\beta}(y;1)|^2y^{\eta^\ast}\,\d y + \int_1^\infty |\xi_{\alpha,\beta}(y;1)|^2 y^{\eta^\ast}\,\d y   \right).
    \end{align*}
    For the first term, we obtain
    \begin{align*}
       & \theta_n^{\frac{-2\beta+\eta^\ast+1}{\alpha}} \int_{\theta_n^{-1/\alpha}}^1 |\xi_{\alpha,\beta}(y;1)|^2y^{\eta^\ast}\,\d y
        \\ &\quad\leq \theta_n^{\frac{-2\beta+\eta^\ast+1}{\alpha}} \int_{\theta_n^{-1/\alpha}}^1 \left| \frac{\sin(\beta \pi)  x^{2\alpha - \beta} + |\sin((\alpha - \beta)\pi)| x^{\alpha - \beta}}{\pi \sin(\alpha \pi)^2}\right|^2 y^{\eta^\ast} \, \d y
        \\ &\quad= \left| \frac{\sin(\beta \pi)}{\pi \sin(\alpha \pi)^2} \right|^2 \theta_n^{\frac{-2\beta+\eta^\ast+1}{\alpha}} \frac{1 - \theta_n^{-(1+4\alpha - 2\beta + \eta^*)/\alpha}}{1+4\alpha - 2\beta + \eta^*}
        \\ &\qquad + 2\frac{\sin(\beta \pi)\sin((\alpha - \beta)\pi)}{\pi^2 \sin(\alpha \pi)^4}\theta_n^{\frac{-2\beta+\eta^\ast+1}{\alpha}} \frac{1 - \theta_n^{-(1 + 3\alpha - 2\beta + \eta^*)/\alpha}}{1 + 3\alpha - 2\beta + \eta^*}
        \\ &\qquad + \left| \frac{\sin((\alpha-\beta) \pi)}{\pi \sin(\alpha \pi)^2}\right|^2 \theta_n^{\frac{-2\beta+\eta^\ast+1}{\alpha}} \frac{1 - \theta_n^{- (1+2\alpha - 2\beta + \eta^*)/\alpha}}{1 + 2\alpha - 2\beta\1_{\{\alpha \neq \beta\}} + \eta^*}
        \\ &\quad\leq \left| \frac{\sin(\beta \pi)}{\pi \sin(\alpha \pi)^2} \right|^2 \frac{\theta_n^{\frac{-2\beta+\eta^\ast+1}{\alpha}}+\theta_n^{-4}}{1+4\alpha - 2\beta + \eta^*}
        \\ &\qquad + 2\frac{|\sin(\beta \pi)\sin((\alpha - \beta)\pi)|}{\pi^2 \sin(\alpha \pi)^4}\frac{\theta_n^{\frac{-2\beta+\eta^\ast+1}{\alpha}}+\theta_n^{-3}}{1 + 3\alpha - 2\beta + \eta^*}
        \\ &\qquad + \left| \frac{\sin((\alpha-\beta) \pi)}{\pi \sin(\alpha \pi)^2}\right|^2 \frac{\theta_n^{\frac{-2\beta+\eta^\ast+1}{\alpha}}+\theta_n^{-2}}{1 + 2\alpha - 2\beta\1_{\{\alpha \neq \beta\}} + \eta^*}
        \\ &\quad\leq \left( \frac{|\sin(\beta \pi)|}{\pi \sin(\alpha \pi)^2}  + \frac{|\sin((\alpha-\beta) \pi)|}{\pi \sin(\alpha \pi)^2} \right)^2 \frac{\theta_n^{\frac{-2\beta+\eta^\ast+1}{\alpha}}+\theta_n^{-2}}{1 + 2\alpha - 2\beta\1_{\{\alpha \neq \beta\}} + \eta^*}.
    \end{align*}
    Likewise, we obtain for the second term
    \begin{align*}
        \theta_n^{\frac{-2\beta+\eta^\ast+1}{\alpha}} \int_1^\infty |\xi_{\alpha,\beta}(y;1)|^2y^{\eta^\ast}\,\d y
        &\leq  \theta_n^{\frac{-2\beta+\eta^\ast+1}{\alpha}} \frac{\left(\frac{|\sin(\beta \pi)|}{\pi \sin(\alpha \pi)^2} + \frac{|\sin((\alpha-\beta)\pi)|}{2\pi (1 + \cos(\alpha \pi))}\right)^2}{2\beta - 1 - \eta^*}.
    \end{align*}
\end{proof}

We are now prepared to study \eqref{eq: mild formulation additive noise} in terms of the corresponding Markovian lift. First, we consider the case where $W$ is a $Q$-Wiener process on $H$ such that $Q$ is trace-class. In such a case, let us take
\[
    H = V = H_b = H_{\sigma}
\]
and denote by $\mathcal{H}_{\delta, \eta}$ the corresponding scale of Hilbert spaces defined in Section 5.1 with reference measure $\mu(\d x) = \d x$. Recall that $S(t)y(x) = \e{-tx}y(x)$ and that $\Xi y = \int_{\R_+} y(x)\, \d x$. In this setting, let us suppose that $G(t) = g(t) + A\int_0^t E^{\alpha, \alpha}(t-s)g(s)\, \d s$ appearing in \eqref{eq: mild formulation additive noise} is of the form
\[
    G(t) = \int_{\R_+} \e{-xt}\xi(x)\, \d x = \Xi S(t)\xi, \qquad t > 0,
\]
where $\xi \in L^p(\Omega,\F_0,\P;\mathcal{H}_{\delta, \eta})$ and $\delta, \eta$ are specified below. We use $\mathcal{G}_p(\delta,\eta)$ to denote the collection of all such admissible driving forces $G$. The following example illustrates a possible choice for $g$ covered by our assumptions.
\begin{example}
    Suppose that $g(t) = \frac{ t^{\gamma - 1}}{\Gamma(\gamma)}g_0$ for some $g_0 \in H$ and $\gamma \geq 0$. Then
    \[
       G(t) = \sum_{n=1}^{\infty} t^{\gamma - 1}E_{\alpha, \gamma}(-\theta_n t^{\alpha}) \langle g_0, e_n^H \rangle_H\, e_n^H.
    \]
    In particular, we obtain $G(t) = \Xi S(t)\xi_g$ with
    \begin{align*}
        \xi_g = \sum_{n=1}^{\infty} \xi_{\alpha,\gamma}(\cdot; \theta_n) \langle g_0, e_n^H \rangle_H\, e_n^H.
    \end{align*}
\end{example}

More generally, given the nature of the Markovian lift, it is also feasible to study the initial conditions directly of the form $G(t) = \Xi S(t)\xi$ and think about $\xi$ being the initial condition. In this setting, the corresponding Markovian lift of \eqref{eq: mild formulation additive noise} takes the form
\begin{align}\label{eq: mild formulation additive noise lift}
    X_t = S(t)\xi + \int_0^t S(t-s)\xi^{\alpha, \alpha}\, b(\Xi X_s)\, \d s
    + \int_0^t S(t-s)\xi^{\alpha, \beta}\, Q^{\frac{1}{2}}\, \d \widetilde{W}_s
\end{align}
for some $\xi$ to be specified below. The next theorem summarises our results of Sections \ref{sec:limit_dist_inv_meas} and \ref{sec:limit_theorems} applied to this particular Markovian lift. Results for $u(\cdot; G)$ may then be obtained through the relation $\Xi X_t = u(\cdot; G)$.

\begin{theorem}\label{theorem  cm lift trace class}
Suppose that $W$ is a $Q$-Wiener process on $H$ with $Q$ trace-class, that $b\colon H \longrightarrow H$ is Lipschitz continuous with constant $C_{b, \mathrm{lip}}$ and linear growth constant $C_{b, \mathrm{lin}} = \sup_{x\in H}\frac{\|b(x)\|_H}{1 + \|x\|_H}$, $\alpha \in (\frac12,1)$, and $\beta \in (\frac{1}{2}, \frac12 +\alpha)$. Let $p\in(2,\infty)$ and $\varepsilon_0, \varepsilon_1 > 0$ satisfy
\[
    \frac{2}{p} < \varepsilon_0 < \frac{2\beta-1}{3} \ \text{ and } \ 0 < \varepsilon_1 < \frac{2}{3}\left(\alpha - \beta \1_{\{\alpha \neq \beta\}}+\frac12 \1_{\{b\equiv0\}}\right),
\]
and define $\delta, \eta$ by
\begin{equation*}\label{eq: delta eta condition}
    \delta = 2\alpha - 2\beta \1_{\{\alpha \neq \beta\}} - 1 + \1_{\{b\equiv0\}} - 2\varepsilon_1 \ \text{ and }\ \eta = 2\beta - 2\varepsilon_0.
\end{equation*}
Then for each $\xi \in L^p(\Omega, \F_0, \P; \mathcal{H}_{\delta, \eta})$ there exists a unique solution of \eqref{eq: mild formulation additive noise lift} in $\mathcal{H}_{\delta, \eta}$. In particular, setting $G = \Xi S(\cdot)\xi \in \mathcal{G}_p(\delta, \eta)$, $u(\cdot; G) = \Xi X$ is the unique solution of \eqref{eq: mild formulation additive noise}. If $b \neq 0$, suppose additionally that we may choose $\varepsilon_0, \varepsilon_1 > 0$ such that
\begin{align*}
    \max \left\{ C_{b, \mathrm{lip}},\ C_{b, \mathrm{lin}}\right\} K_0(\alpha, \beta, \varepsilon_0, \varepsilon_1) K_1(\alpha, \varepsilon_0, \varepsilon_1) < 1
\end{align*}
with constants $K_0 = K_0(\alpha, \beta, \varepsilon_0, \varepsilon_1)$ and $K_1 = K_1(\alpha, \varepsilon_0, \varepsilon_1)$ given by
\begin{align*}
    K_0(\alpha, \beta, \varepsilon_0, \varepsilon_1) &\coloneqq \left( \frac{1}{2\alpha - 2\beta \1_{\{\alpha \neq \beta\}} - 2\varepsilon_1} + \frac{1}{2\beta - 1 - 2\varepsilon_0}\right)^{1/2},
    \\  K_1(\alpha, \varepsilon_0, \varepsilon_1) &\coloneqq\frac{2}{\pi \sin(\alpha \pi)} \bigg( \theta_1^{-2} \varepsilon_1^{-\frac{1}{2}} + \varepsilon_0^{-\frac{1}{2}}\bigg) \left(1+\frac{2}{\varepsilon_1} \right) \left(\frac{2 + \varepsilon_1}{2 \mathrm{e}}\right)^{2 + \varepsilon_1}.
\end{align*}
Then, the following assertions hold:
\begin{enumerate}
    \item[(a)] Equation \eqref{eq: mild formulation additive noise lift} admits a unique limiting distribution $\pi \in \mathcal{P}_p(\mathcal{H}_{\delta, \eta})$ with respect to the Wasserstein $p$-distance. This limit distribution is also the unique invariant measure.
    \item[(b)] Let $\widetilde{\xi} \sim \pi$ and set $\widetilde{G} = \Xi S(\cdot)\widetilde{\xi}$. Then $u(\cdot,\widetilde{G})$ is the unique stationary process corresponding to \eqref{eq: mild formulation additive noise}.
    \item[(c)] If $p\in(4,\infty)$, the Law of Large Numbers holds in the mean-square sense with rate of convergence
    \[
    \vartheta < \frac{1}{2} \begin{cases}
        \min\left\{1,\ \frac{1}{2} + \alpha - \beta \1_{ \{\alpha \neq \beta\}}\right\}, & b = 0,
        \\ \min\left\{1,\ \alpha - \beta \1_{ \{\alpha \neq \beta\}},\ \log\left(C_{b, \mathrm{lip}}^{-1} K_0^{-1} K_1^{-1}\right)\right\}, & b \neq 0.
    \end{cases}
    \]
\end{enumerate}
\end{theorem}
\begin{proof}
    Denote by $\mathcal{H}_{\delta, \eta}$ the scale of Hilbert spaces defined in Section \ref{section markovia lift cm} with $\mu(\d x) = \d x$ as introduced above. Hence we may take any $\eta_* > 1$ and $\delta_* > -1$. It follows from Theorem \ref{theorem cm lift} combined with the first inequality in Lemma \ref{theorem cm frac spaces} and representation \eqref{eq: xi mun}, that Assumptions \ref{assumption SEE} and \ref{assumption long-time}.(b) and (c) are satisfied for $q = \infty$, $q' = 1$ since $Q$ is trace-class, and
    \[
        \mathcal{H} = \mathcal{H}_{\delta, \eta^*},\ \ \mathcal{V} = \mathcal{H}_{\delta, \eta}, \ \ \mathcal{V}_0 = \mathcal{H}_{\delta_*, \eta},\ \ \lambda = \frac{\delta - \delta_*}{2}, \ \ \rho = \frac{(\eta - \eta^*)_+}{2}
    \]
    where $\max\{\eta_*, \eta^*\} \leq \eta < 1 + \eta^*$, $\delta \in (\delta_*, \delta^*)$, such that $\delta^* - \delta > 2-\1_{\{b\equiv0\}}$. In view of Lemma \ref{theorem cm frac spaces} remark that $\delta^*, \eta^*$ neccessarily satisfy $\delta^* < 1 + 2\alpha - 2\beta \1_{\{ \alpha \neq \beta \} }$ and $\eta^* < 2\beta - 1$. Thus, let us take $\eta_* = 1 + \varepsilon_0$, $\eta^* = 2\beta - 1 - \varepsilon_0$, $\delta_* = -1 + \varepsilon_1$, and $\delta^* = 1 + 2\alpha - 2\beta \1_{\{\alpha \neq \beta\}} - \varepsilon_1$. Then $\delta = \delta^* -2 + \1_{\{b\equiv0\}} - \varepsilon_1$, and the above conditions are satisfied with
    \[
        \rho = \frac{(1 - \varepsilon_0)_+}{2} \in \left[0, 1/2 \right) \ \text{ and } \ \lambda = \alpha - \beta \1_{\{\alpha \neq \beta\}} + \frac12 \1_{\{b\equiv0\}} - \frac{3\varepsilon_1}{2} > 0.
    \]
    Moreover, since $\varepsilon_0 > \frac2p$, we also obtain $\rho + \frac{1}{p} < \frac{1}{2}$, see \eqref{eq: continuity}. Finally, by assumption $b$ is Lipschitz continuous with constant $C_{b, \mathrm{lip}}$, thus also Assumption \ref{assumption long-time}.(a) is satisfied. The existence and uniqueness of solutions follow from Theorem \ref{label:existence_uniqueness_VSPDE}.

    Concerning limit distributions, our assertions follow from Theorem \ref{theorem_limit_distribution}.(c) provided that \eqref{eq:small_nonlinearity} is satisfied. To verify the latter, following the proof of Lemma \ref{lemma: CM SSE}.(b) we get
    \begin{align*}
        \| \Xi \|_{L(\mathcal{H}_{\delta, \eta}, H)} \leq \left(\int_{\R_+} \frac{\d x}{w_{\delta,\eta}(x)} \right)^{1/2}
        &= \left( \frac{1}{1 + \delta} + \frac{1}{\eta - 1}\right)^{1/2}
        \\ &= \left( \frac{1}{2\alpha - 2\beta \1_{\{\alpha \neq \beta\}} - 2\varepsilon_1} + \frac{1}{2\beta - 1 - 2\varepsilon_0}\right)^{1/2}.
    \end{align*}
    Similarly, we also obtain for $\mathcal{V}_0 = \mathcal{H}_{\delta_*, \eta}$
    \[
        \| \Xi \|_{L(\mathcal{H}_{\delta_*, \eta}, H)} \leq \left( \frac{1}{1 + \delta_*} + \frac{1}{\eta - 1}\right)^{1/2} = \left( \frac{1}{2 + 2\alpha - 2\beta \1_{\{\alpha \neq \beta\}} - \varepsilon_1} + \frac{1}{2\beta - 1 - 2\varepsilon_0}\right)^{1/2}
    \]
    Moreover, if $b\neq0$, using \eqref{eq: estimate S(t) xi_b} we obtain for our particular choice of $H_b = H$ and $\mathcal{V} = \mathcal{H}_{\delta, \eta}$
     \begin{align*}
        \int_0^{\infty}\| S(t)\xi^{\alpha, \alpha}\|_{L(H_b, \mathcal{V})}\, \mathrm{d}t
        &\leq 2 \max\{ 1,\ C(\eta - \eta^*)\} C(\delta^* - \delta) \| \xi^{\alpha,\alpha}\|_{L(H, \mathcal{H}_{\delta^*, \eta^*})}\int_0^{\infty} (1 \vee t)^{-\frac{\delta^* - \delta}{2}}\, \d t
    \end{align*}
    For the integral, we obtain
    \[
        \int_0^{\infty} (1 \vee t)^{-\frac{\delta^* - \delta}{2}}\, \d t = 1 + \frac{1}{\frac{\delta^* - \delta}{2} - 1}
        = \frac{\delta^* - \delta}{\delta^* - \delta - 2}
        = 1 + \frac{2}{\varepsilon_1}.
    \]
    For the remaining terms, noting that $C(x) = 2^{-x} x^x \e{-x}$ for $x \geq 0$ is strictly decreasing on $[0,2]$, we get $C(\eta - \eta^*) \leq C(0) = 1$ since $\eta - \eta^* \leq 1$ and $C(\delta^* - \delta) = C(2 + \varepsilon_1)$.      Finally, using Lemma \ref{theorem cm frac spaces}, we get
    \begin{align*}
        \| \xi^{\alpha, \alpha}\|_{L(H, \mathcal{H}_{\delta^*, \eta^*})} &\leq \frac{1}{\pi \sin(\alpha \pi)} \left(\frac{ \theta_1^{-2}}{\sqrt{1 + 2\alpha - \delta^*}}
        + \frac{ 1 }{\sqrt{2\alpha - 1 - \eta^*}} \right)
        \\ &=  \frac{1}{\pi \sin(\alpha \pi)} \left(\frac{ \theta_1^{-2}}{\sqrt{ \varepsilon_1}}
        + \frac{ 1 }{\sqrt{\varepsilon_0}} \right).
    \end{align*}
     Hence \eqref{eq:small_nonlinearity} is satisfied by assumption. The assertion about the limit distribution now follows from Theorem \ref{theorem_limit_distribution}.(c), where uniqueness of the limit distribution follows from $S_{\infty} = 0$ due to $\mu(\{0\}) = 0$. The Law of Large Numbers, including the convergence rate,  is a consequence of Corollary \ref{theorem LLN lift}. For the convergence rate, notice that $\rho_{\textrm{add}}=C_{b,\textrm{lip}}\norm{\Xi}_{L(H_{\delta_\ast,\eta},H)}\norm{S(\cdot)\xi^{\alpha,\alpha}}_{L(H,\mathcal{H}_{\delta_\ast,\eta})}$ and so
     \[
     \norm{\rho_{\textrm{add}}}_{L^1(\R_+)}\leq C_{b,\textrm{lip}}\left( \frac{1}{2\alpha - 2\beta \1_{\{\alpha \neq \beta\}} - 2\varepsilon_1} + \frac{1}{2\beta - 1 - 2\varepsilon_0}\right)^{1/2}K_1.
     \]
     Thus, the monotonicity of $\log(\cdot)$ and the upper bound
     $
     \vartheta < \frac{1}{2}\min\{1,\lambda,\log(1/\norm{\rho_{\textrm{add}}}_{L^1(\R_+))}\}
     $
     yield the desired convergence rate.
\end{proof}

Below we continue with the case where the covariance operator of Gaussian noise $W$ is given by
\begin{align}\label{eq: covariance operator}
    Q = (-A)^{-\gamma},\qquad \gamma \in \R.
\end{align}
Remark that $\gamma = 0$ contains the case where $W$ is a cylindrical Wiener process. Below, we obtain a similar result to Theorem \ref{theorem  cm lift trace class} under an additional summability condition. Finally, let us take
$V = H = H_b$ and $H_\sigma = H^{\gamma}$.
\begin{theorem}
    Let $W$ be a Wiener process with covariance operator \eqref{eq: covariance operator}. Suppose that the assumptions of Theorem \ref{theorem  cm lift trace class} are satisfied, and that additionally
    \begin{equation}\label{eq: theorem cm L2 summability}
        \sum_{n=1}^\infty \theta_n^{-2\gamma - \frac{\varepsilon_0}{\alpha}}<\infty.
    \end{equation}
    Then the assertions (a) -- (c) of Theorem \ref{theorem  cm lift trace class} hold.
\end{theorem}
\begin{proof}
    In light of Theorem \ref{theorem  cm lift trace class}, it suffices to verify that the assumptions of Theorem \ref{theorem cm lift}, with $q=q'=2$, are satisfied. In particular, it remains to show \eqref{eq: kernels condition 1} is satisfied. Indeed, by Lemma \ref{theorem cm frac spaces}, we find that for the chosen $\eta^\ast,\delta^\ast$
    \[
    \int_{\R_{+}}\norm{\xi^{\alpha,\beta}(x)}_{L_2(H^{\gamma},H)}w_{\delta^\ast,\eta^\ast}(x)\,\d x \lesssim \sum_{n=1}^\infty\theta_n^{-2\gamma -\frac{\varepsilon_0}{\alpha}} + \sum_{n=1}^\infty\theta_n^{-2\gamma - 2}
    \]
    where the right-hand side is finite provided that \eqref{eq: theorem cm L2 summability} holds as $\varepsilon_0<2\alpha$.
\end{proof}

To illustrate this result, let us consider the Dirichlet Laplace operator for $(A, D(A))$, in the following example.

\begin{example}
     Let $\mathcal{O} \subset \R^d$ be a bounded domain with $C^1$-boundary, and set $H = L^2(\mathcal{O})$. Then $(A,D(A)) = (\Delta, H_0^1(\mathcal{O}) \cap H^2(\mathcal{O}))$ is diagonalisable with an orthonormal basis $(e_n^H)_{n \geq 1}$ and sequence of eigenvalues $(\theta_n)_{n \geq 1}$  Without loss of generality, we suppose that the latter are increasing to infinity. By Weyl's law, we find for their asymptotics
    \[
        \theta_n \sim c(d, \mathcal{O})n^{2/d}, \qquad n \to \infty,
    \]
    where $c(d, \mathcal{O}) > 0$ denotes some constant. Hence, the summability condition \eqref{eq: theorem cm L2 summability} becomes $\sum_{n=1}^\infty n^{\frac{-4}{d}\gamma-\frac2d\frac{\varepsilon_0}{\alpha}}<\infty$ and is satisfied whenever
    \[
    2\alpha\left(\frac{d}{4} - \gamma\right) < \varepsilon_0.
    \]
\end{example}

Finally, notice that for fractional kernels, the convergence rate in the Law of Large Numbers is too small to obtain the Central Limit Theorem. The next remark outlines that for fractional gamma kernels, the optimal rate of convergence, and hence also the Central Limit Theorem, can be obtained.

\begin{remark}
 For given $\lambda > 0$, let us consider the Volterra kernels
\[
    k(t) = \frac{t^{\alpha - 1}}{\Gamma(\alpha)}\e{-\lambda t} \ \text{ and } \
    h(t) = \frac{t^{\beta - 1}}{\Gamma(\beta)}\e{- \lambda t}.
\]
Then $E^{\alpha, \beta}$ needs to be replaced by
\[
    E^{\alpha, \beta, \lambda}(t) = \int_{\R_+} \e{-xt}\xi^{\alpha, \beta, \lambda}(x)\, \d x
\]
where $\xi^{\alpha, \beta,\lambda}(x) = \sum_{n=1}^{\infty}\xi_{\alpha,\beta,\lambda}(x; \theta_n) (e_n^H \otimes e_n^H)$ and
\[
    \xi_{\alpha,\beta, \lambda}(x; \theta_n) = \1_{(\lambda, \infty)}(x)\xi_{\alpha, \beta}(x-\lambda; \theta_n).
\]
Hence, we may obtain similar bounds to Lemma \ref{theorem cm frac spaces} with the only difference that $\delta^*$ can be now chosen arbitrarily large, see also Remark \ref{remark: regularization}. The latter is sufficient to verify the conditions of Theorem \ref{theorem CLT general} and hence derive a Central Limit Theorem.
\end{remark}

\begin{remark}
    Remark that in all examples above, we may also choose $\mu(\d x) = \delta_0(\d x) + \d x$ as a reference measure. The latter necessarily gives $S_{\infty} \neq 0$, and hence limit distributions will be parameterised by $S_{\infty} \xi = \xi(0)$ where $\xi$ denotes the initial condition. For such a choice of lift based on $\mu$, invariant measures are not unique.
\end{remark}

\section{Markovian lift on weighted Sobolev space}\label{section HJM lift}

\subsection{General framework}\label{subsec:general_framework}

In this section, we provide a Markovian lift based on translations of the Volterra kernels. The latter covers, e.g., the fractional kernel $k(t) = t^{\alpha - 1}/ \Gamma(\alpha)$ in the full regime of parameters $\alpha \in (0,2)$ beyond the completely monotone case $\alpha \in (0,1)$ studied in Section \ref{section markovia lift cm}. Such a lift was, e.g., described in \cite{MR4503737} for Volterra kernels that have time regularity $W_{\mathrm{loc}}^{1,2}$ with integrable weak derivative as $t \to \infty$, see also \cite{MR3574705}. Abstract conditions that go beyond this case were also discussed in \cite{MR4181950} for the finite-dimensional setting. Below, we provide a modification of this lift that allows us to weaken both assumptions with a particular focus on polynomial rates of convergence.

Let $V \hookrightarrow H$ be separable Hilbert spaces, see \eqref{eq: inclusion}. For $\delta,\eta\in\R$ let us define the modified Filipovi\'c space $\mathcal{H}_{\delta,\eta}$ consisting of absolutely continuous functions $y\colon(0,\infty)\longrightarrow V$ with finite norm
	\[
	\vertiii{y}_{\delta,\eta}^2\coloneqq \norm{y(1)}_V^2 + \int_{0}^\infty \norm{y'(x)}_V^2 w_{\delta,\eta}(x)\,\d x
	\]
	where $y'$ denotes the weak derivative of $y$ and $w_{\delta, \eta} \colon(0,\infty)\longrightarrow(0,\infty)$ is the increasing weight function
	\[
	   w_{\delta,\eta}(x) = x^{\eta}\1_{(0,1]}(x) + x^{\delta} \1_{(1,\infty)}(x).
	\]
    Similarly to \cite[Section 3]{MR3574705}, one can show that $\mathcal{H}_{\delta, \eta}$ is a separable Hilbert space. Note that $(\mathcal{H}_{\delta,\eta})_{\delta,\eta\in\R}$ satisfies $\mathcal{H}_{\delta,\eta}\subset \mathcal{H}_{\delta',\eta}$ for $\delta'<\delta$ and $\mathcal{H}_{\delta,\eta'}\subset\mathcal{H}_{\delta,\eta}$ for $\eta'<\eta$. Also here, $\eta$ captures the time regularity, and $\delta$ the decay rate as $t \to \infty$. In this space, point evaluations and translations play a central role. Their properties are summarised in the next lemma.

    \begin{lemma}\label{lemma: HJM lift technical lemma}
        For $z \geq 0$, let $\Xi_zy = y(z)$ be the point evaluation, and let $(S(t))_{t\geq0}$ be the semigroup of shift operators on $\mathcal{H}_{\delta,\eta}$ given by
        \[
            S(t)y(x) = y(x+t),\quad y\in\mathcal{H}_{\delta,\eta},\ x\in\R_+.
        \]
        Then the following assertions hold:
        \begin{enumerate}
            \item[(i)] If $z \in (0,\infty)$, then $\Xi_z\colon \mathcal{H}_{\delta, \eta} \longrightarrow V$ is a bounded linear operator.
            \item[(ii)] If $z = 0$ and $\eta < 1$, then $\Xi_0\colon \mathcal{H}_{\delta, \eta} \longrightarrow V$ is a bounded linear operator given by
            \[
                \Xi_0y = y(0) \coloneqq y(1) - \int_0^1 y'(x)\, \mathrm{d}x.
            \]

            \item[(iii)] If $z = \infty$ and $\delta > 1$, then $\Xi_{\infty}\colon \mathcal{H}_{\delta, \eta} \longrightarrow V$ is a bounded linear operator given by
            \[
                \Xi_{\infty}y = y(\infty) \coloneqq y(1) + \int_1^{\infty} y'(x)\, \mathrm{d}x.
            \]

            \item[(iv)] $(S(t))_{t\geq0}$ is strongly continuous on $\mathcal{H}_{\delta, \eta}$ whenever $\delta, \eta \geq 0$. Moreover, let $\eta' \geq \eta$ and $\delta' \geq \delta$, then $S(t)\in L(\mathcal{H}_{\delta',\eta'},\mathcal{H}_{\delta,\eta})$ and for each $T > 0$ there exists $C(T) > 0$ such that
            \[
            \norm{S(t)}_{L(\mathcal{H}_{\delta',\eta'},\mathcal{H}_{\delta,\eta})}\lesssim C(T) \left(1 + t^{-(\eta'-\eta)/2}\right), \qquad t \in (0,T].
            \]
            In particular $C(T)$ can be chosen independently of $T$ whenever $\delta'>\eta$.

            \end{enumerate}
            In particular, Assumption \ref{assumption SEE} is satisfied for any choice of $0 \leq \delta\leq\delta'$ and $\eta \in [0,1)$ and $\eta \leq \eta' < 1 + \eta$ with
            \[
                \mathcal{H}=\mathcal{H}_{\delta',\eta'},\quad \mathcal{V}=\mathcal{H}_{\delta,\eta},\quad \rho=\frac{\eta'-\eta}{2}
            \]
            and bounded linear projection operator $\Xi \coloneqq \Xi_0\colon\mathcal{V}\longrightarrow V$.
    \end{lemma}
    \begin{proof}
        Suppose that $z \geq 0$. Then using $y(z) = y(1) + \int_{1}^{z}y'(x)\, \mathrm{d}x$, we obtain from the Cauchy-Schwarz inequality:
    \begin{align*}
        \|\Xi_z(y) \|_V = \|y(z)\|_V
        &\leq \|y(1)\|_V + \int_{\min\{1,z\}}^{\max\{1,z\}}\|y'(x)\|_V \, \mathrm{d}x
        \\ &\leq \|y(1)\|_V + \left( \int_{\min\{1,z\}}^{\max\{1,z\}} \|y'(x)\|_V^2\hspace{0.02cm} w_{\delta, \eta}(x)\, \mathrm{d}x \right)^{1/2}\left(\int_{\min\{1,z\}}^{\max\{1,z\}} \frac{\mathrm{d}x}{w_{\delta, \eta}(x)} \right)^{1/2}
        \\ &\leq \vertiii{y}_{\delta, \eta} \left( 1 + \left(\int_{\min\{1,z\}}^{\max\{1,z\}} \frac{\mathrm{d}x}{w_{\delta, \eta}(x)} \right)^{1/2} \right).
    \end{align*}
    The right-hand side is finite if $z > 0$, regardless of the choice of $\delta, \eta$. If $z = 0$, then the right-hand side is finite whenever $\eta < 1$. Finally, for $z = \infty$, the right-hand side is finite whenever $\delta > 1$.

    For the last assertion, let us first show that $S(t)$ is a bounded linear operator. Let $y \in \mathcal{H}_{\delta, \eta}$, then using $\| \Xi_{1+t}\|_{L(\mathcal{H}_{\delta, \eta}, V)} \leq 1 + \left(\int_1^{1+T} \frac{\mathrm{d}x}{w_{\delta, \eta}(x)}\right)^{1/2}$ for $t \in [0,T]$, we obtain
    \begin{align*}
        \vertiii{S(t)y}_{\delta, \eta}^2 &= \norm{\Xi_{1+t}y}_{V}^2 + \int_0^{\infty}\|y'(t+x)\|_V^2 w_{\delta, \eta}(x)\, \mathrm{d}x
        \\ &\leq \sup_{t \in [0,T]}\| \Xi_{1+t}\|_{L(\mathcal{H}_{\delta, \eta}, V)}^2 \vertiii{y}_{\delta, \eta}^2 + \int_t^{\infty}\|y'(x)\|_V^2 w_{\delta, \eta}(x-t)\, \mathrm{d}x
        \\ &\leq \sup_{t \in [0,T]}\| \Xi_{1+t}\|_{L(\mathcal{H}_{\delta, \eta}, V)}^2 \vertiii{y}_{\delta, \eta}^2 + \int_t^{\infty}\|y'(x)\|_V^2 w_{\delta, \eta}(x)\, \mathrm{d}x
        \\ &\leq \left( 1 + \sup_{t \in [0,T]}\| \Xi_{1+t}\|_{L(\mathcal{H}_{\delta, \eta}, V)}^2\right) \vertiii{y}_{\delta, \eta}^2
    \end{align*}
    where we have used that $w_{\delta, \eta}(x-t) \leq w_{\delta, \eta}(x)$ since $\eta, \delta \geq 0$. Thus, it suffices to verify the strong continuity
    \[
        \lim_{t \to 0}\vertiii{S(t)y - y}_{\delta, \eta} = 0, \qquad \forall y \in \mathcal{D},
    \]
    where $\mathcal{D} \subset \mathcal{H}_{\delta, \eta}$ is dense. Let us take
    \[
        \mathcal{D} = \left\{ y \in C^2(\R_+; V) \ : \ y' \in C_c^1(\R_+; V) \right\}.
    \]
    Similarly to \cite[Section 3]{MR3574705}, it can be shown that $\mathcal{D} \subset \mathcal{H}_{\delta, \eta}$ is dense. For $y \in \mathcal{D}$, let us write
    \begin{align*}
        \vertiii{S(t)y - y}_{\delta, \eta}^2 &= \| y(1+t) - y(1)\|_V^2 + \int_0^{\infty}\|y'(t+x) - y'(x)\|_V^2 w_{\delta, \eta}(x)\, \mathrm{d}x.
    \end{align*}
    For the first term, we obtain
    \begin{align*}
        \| y(1+t) - y(1)\|_V \leq \int_1^{1+t} \| y'(x)\|_V\, \mathrm{d}x \to 0
    \end{align*}
    by dominated convergence since $\int_1^{1+t} \| y'(x)\|_V\, \mathrm{d}x \leq \left(\int_{1}^{1+T} w_{\delta, \eta}(x)^{-1}\mathrm{d}x \right)^{1/2} \vertiii{y}_{\delta, \eta} < \infty$. For the second term, let us first note that
    $y'(t+x) - y'(x) = \int_x^{t+x}y''(\widetilde{x})\, \mathrm{d}\widetilde{x}$. Let $R > 0$ be large enough such that $y'(t+x) = y'(x) = 0$. Then we obtain for $t \in [0,1]$
    \begin{align*}
        \int_0^{\infty}\|y'(t+x) - y'(x)\|_V^2 w_{\delta, \eta}(x)\, \mathrm{d}x
        &\leq \int_0^{R}\left( \int_x^{t+x} \| y''(\widetilde{x})\|_V \, \mathrm{d}\widetilde{x} \right)^2 w_{\delta, \eta}(x)\, \mathrm{d}x
        \\ &\leq \sqrt{t} \int_0^{R} \int_x^{t+x} \| y''(\widetilde{x})\|_V^2 w_{\delta, \eta}(x)\, \, \mathrm{d}\widetilde{x} \mathrm{d}x
        \\ &\leq \sqrt{t} \left(\int_0^R w_{\delta,\eta}(x)\, \mathrm{d}x \right) \left( \int_0^{1 + R} \| y''(\widetilde{x})\|_V^2\, \mathrm{d}\widetilde{x} \right) \to 0.
    \end{align*}
    This proves the desired strong continuity.
    For the regularising property of the semigroup, we let $0 \leq \delta \leq \delta'$, and $0 \leq \eta \leq \eta'$. Then using
    \[
        \int_{1}^{1+t} \frac{\mathrm{d}x}{w_{\delta', \eta'}(x)} = \begin{cases} \ln(1+t), & \delta' = 1
        \\ \frac{(1+t)^{1 - \delta'} - 1 }{1 - \delta'}, & \delta' \neq 1
        \end{cases} < \infty,
    \]
    we find for $y\in\mathcal{H}_{\delta',\eta'}$, $t\in[0,T]$,
    \begin{align*}
        \vertiii{S(t)y}_{\delta, \eta}^2 &= \norm{\Xi_{1+t}y}_{V}^2 + \int_t^{\infty}\|y'(x)\|_V^2 w_{\delta, \eta}(x-t)\, \mathrm{d}x
        \\ &\lesssim \vertiii{y}_{\delta', \eta'}^2 \left( 1 + \int_{1}^{1+t} \frac{\mathrm{d}x}{w_{\delta', \eta'}(x)}  \right) + \int_t^{1 + t}\|y'(x)\|_V^2 (x-t)^{\eta}\, \mathrm{d}x
        \\ &\qquad \qquad \qquad \qquad  + \int_{1 + t}^{\infty} \|y'(x)\|_V^2 (x-t)^{\delta}\, \mathrm{d}x.
    \end{align*}

        Let us estimate the remaining two integrals. For the first integral, assume $t > 1$. Then using $(x-t)^{\eta} \leq x^{\eta} \leq x^{\delta'}$ when $\delta' \geq \eta$, and $(x-t)^{\eta} \leq x^{\eta} \leq (1+T)^{\eta-\delta'} x^{\delta'}$ for $\delta' < \eta$, we obtain
    \begin{align*}
        \int_t^{1 + t}\|y'(x)\|_V^2 (x-t)^{\eta}\, \mathrm{d}x
        &\leq (1+T)^{(\eta-\delta')_+} \int_{t}^{1 + t}\|y'(x)\|_V^2 x^{\delta'}\, \mathrm{d}x
        \leq (1+T)^{(\eta-\delta')_+}\vertiii{y}_{\delta', \eta'}^2.
    \end{align*}
    When $t \in (0,1]$, we may use $(x-t)^{\eta} \leq x^{\eta} = x^{\eta'}x^{-(\eta' - \eta)} \leq x^{\eta'}t^{-(\eta' - \eta)}$ to find
    \begin{align*}
        \int_t^{1 + t}\|y'(x)\|_V^2 (x-t)^{\eta}\, \mathrm{d}x
        &\leq \int_{t}^{1}\|y'(x)\|_V^2 x^{\eta}\, \mathrm{d}x +
        \int_{1}^{1 + t}\|y'(x)\|_V^2 x^{\eta}\, \mathrm{d}x
        \\ &\leq (1 \wedge t)^{-(\eta' - \eta)} \int_t^{1}\|y'(x)\|_V^2 x^{\eta'}\, \mathrm{d}x + 2^{(\eta-\delta')_+}\int_{1}^{1 + t}\|y'(x)\|_V^2 x^{\delta'}\, \mathrm{d}x
        \\ &\leq \left( 2^{(\eta-\delta')_+} + (1 \wedge t)^{-(\eta' - \eta)} \right)\vertiii{y}_{\delta', \eta'}^2.
    \end{align*}

    Finally, using $(x-t)^{\delta} \leq x^{\delta} \leq (1+t)^{-(\delta' - \delta)} x^{\delta'}$, the last term can be bounded by
    \begin{align*}
        \int_{1 + t}^{\infty} \|y'(x)\|_V^2 (x-t)^{\delta}\, \mathrm{d}x
        &\lesssim (1+t)^{-(\delta' - \delta)} \int_{1 + t}^{\infty} \|y'(x)\|_V^2 x^{\delta'}\, \mathrm{d}x
        \\ &\lesssim (1 \vee t)^{-(\delta' - \delta)}\vertiii{y}_{\delta', \eta'}^2.
    \end{align*}
    This proves the assertion with $\rho = (\eta' - \eta)_+/2$.
    \end{proof}

    Below we provide sufficient conditions on $E_b, E_{\sigma}$ such that Assumptions \ref{assumption SEE} and \ref{assumption long-time} are satisfied. Let $E_b\colon (0,\infty) \longrightarrow L(H_b, V)$ and $E_\sigma \colon (0,\infty) \longrightarrow L_q(H_{\sigma},V)$ be absolutely continuous such that there exist $\eta^*, \delta^* \in \R$ with
    \begin{align}\label{eq: lift condition 2}
    \begin{cases}
        \quad \int_1^\infty \left(\norm{E'_b(x)}_{L(H_b, V)}^2 + \norm{E'_\sigma(x)}_{L_q(H_{\sigma}, V)}^2\right) x^{\delta^*}\,\d x <\infty
        \\ \quad \int_0^1 \left(\norm{E'_b(x)}_{L(H_b, V)}^2 + \norm{E'_\sigma(x)}_{L_q(H_{\sigma}, V)}^2\right) x^{\eta^*}\,\d x <\infty.
        \end{cases}
    \end{align}
    As before, note that we may always replace $\delta^*$ by a smaller value, and $\eta^*$ by a larger value. Hence, we are interested in the largest choice for $\delta^*$, and the smallest possible choice for $\eta^*$. Let us define for $a\in\{b,\sigma\}$ the action of $E_a$ on $h\in H_a$ via $(E_ah)(x)=E_a(x)h$ where $x>0$. Then $E_b\in L(H_b,\mathcal{H}_{\delta^*,\eta^*})$ and $E_{\sigma} \in L_q(H_{\sigma}, \mathcal{H}_{\delta^*, \eta^*})$. In the following, we denote by $\iota_{\delta,\eta}$ the natural embedding from $V$ into $\mathcal{H}_{\delta,\eta}$ which we will omit when it is clear from the context.

    \begin{theorem}\label{theorem HJJM lift}
        Let $E_b,E_{\sigma}$ be absolutely continuous with \eqref{eq: lift condition 2} such that $\eta^* \in[0,2)$ and $\delta^* \geq 0$. Then the following assertions hold:
        \begin{enumerate}
            \item[(a)] For each $\eta \in [0,1)$ and $\delta > 1$, $S(t) \longrightarrow S_{\infty}\coloneqq\iota_{\delta,\eta}\Xi_\infty$, strongly on $\mathcal{H}_{\delta, \eta}$, and
            \[
                \|S(t) - S_{\infty}\|_{L(\mathcal{H}_{\delta, \eta}, \mathcal{H}_{\delta', \eta})} \leq \max\left\{ 1, \frac{1}{(\delta'-1)^{1/2}}\right\} (1 \vee t)^{- \frac{\delta - \delta'}{2}}
            \]
            holds for all $1 < \delta' < \delta$.

            \item[(b)] If $\delta^* > 3$, $\lim_{t\to\infty}E_b(t)=0$, and $\lim_{t\to\infty}E_\sigma(t)=0$, then Assumptions \ref{assumption long-time}.(b) and (c) are satisfied for any $\eta\in[0,1)$ such that $\eta\leq\eta^\ast<1+\eta$ and $\delta_* \in (1,\delta^*)$, $\delta \in (\delta_*, \delta^*)$ such that $\delta^* - \delta > 2$
            with
            \begin{equation}\label{eq:HJM lift parameters}
            \mathcal{H}=\mathcal{H}_{\delta,\eta^\ast},\quad \mathcal{V}=\mathcal{H}_{\delta,\eta},\quad \mathcal{V}_0 = \mathcal{H}_{\delta_\ast,\eta},\quad \lambda=\frac{\delta-\delta_\ast}{2},\quad \rho=\frac{\eta^\ast-\eta}{2}
            \end{equation}
            and projection operator $S_\infty$.

            \item[(c)] Suppose that $b\equiv0$ and $E_b\equiv0$. If $\delta^\ast>2$ and $\lim_{t\to\infty}E_\sigma(t)=0$, then Assumptions \ref{assumption long-time}.(b) and (c) are satisfied with \eqref{eq:HJM lift parameters} where $\eta\in[0,1)$ satisfies $\eta\leq\eta^\ast<1+\eta$ and $\delta_* \in (1,\delta^*)$, $\delta \in (\delta_*, \delta^*)$ satisfy $\delta^* - \delta > 1$.
        \end{enumerate}
    \end{theorem}
    \begin{proof}
        (a) For $y\in\mathcal{H}_{\delta,\eta}$, we find
        \[
            \vertiii{S(t)y-S_\infty}_{\delta,\eta}^2= \norm{y(1+t)-y(\infty)}_V^2 +\int_0^\infty \norm{y'(x+t)}_V^2 w_{\delta,\eta}(x)\,\d x.
        \]
        The first term satisfies
       \begin{equation*}\label{eq: P convergence shift}
           \norm{y(1+t)-y(\infty)}_V^2 = \norm*{\int_{1+t}^\infty y'(x)\,\d x}_V^2 \leq \left(\int_{1+t}^\infty \norm{y'(x)}_V^2 x^{\delta} \,\d x \right)\left(\int_1^\infty x^{-\delta}\,\d x\right).
       \end{equation*}
       For the second we use $(x-t)^{\eta} \leq x^{\delta}$ for $1 \leq t \leq x \leq 1 + t$ to find
       \begin{align*}
            \int_0^\infty \norm{y'(x+t)}_V^2 w_{\delta,\eta}(x)\,\d x
            &= \int_t^{1+t} \norm{y'(x)}_V^2 (x-t)^{\eta}\,\d x
            + \int_{1+t}^\infty \norm{y'(x)}_V^2 (x-t)^{\delta} \,\d x
            \\ &\leq 2\int_{t}^\infty \norm{y'(x)}_H^2 x^{\delta}\,\d x.
        \end{align*}
        By dominated convergence, the right-hand side tends to zero as $t\to\infty$. This proves $\lim_{t \to \infty}S(t) = S_\infty=\iota_{\delta,\eta}\Xi_\infty$ strongly on $\mathcal{H}_{\delta, \eta}$.

        Next, we derive the convergence rate bound \eqref{eq: rate S to P} on $\mathcal{V}_0 = \mathcal{H}_{\delta', \eta}$, $\mathcal{V} = \mathcal{H}_{\delta, \eta}$. Take $y\in\mathcal{H}_{\delta,\eta}$ and note that $\vertiii{S(t)y-\iota_{\delta,\eta}y(\infty)}_{\delta',\eta}^2 = \norm{y(1+t)-y(\infty)}_V^2 + \int_0^\infty \norm{y'(x+t)}_V^2 w_{\delta',\eta}(x)\,\d x$. For the first term, we obtain for $t \geq 0$
    \begin{align*}
        \norm{y(1+t)-y(\infty)}_V^2 &= \norm*{\int_{1+t}^\infty y'(x)\,\d x}_V^2
        \\ &\leq \left(\int_{1+t}^\infty \norm{y'(x)}_V^2 x^{\delta}x^{-(\delta-\delta')}\,\d x\right)\left( \int_{1}^\infty x^{-\delta'}\,\d x\right)
        \\ &\leq \frac{(1+t)^{-(\delta-\delta')}}{\delta' - 1}\left(\int_{1}^\infty \norm{y'(x)}_V^2 x^{\delta}\,\d x\right).
    \end{align*}
    For the second term we use $(x-t)^{\eta} \leq x^{\delta} x^{-(\delta - \eta)} \leq t^{-(\delta - \eta)}x^{\delta} \leq t^{-(\delta - \delta')}x^{\delta}$ when $1 \leq t \leq x \leq 1 + t$ to find
    \begin{align*}
        \int_0^1 \norm{y'(x+t)}_V^2 x^{\eta}\,\d x
        &= \int_t^{1+t} \norm{y'(x)}_V^2(x-t)^{\eta}\,\d x
        \\ &\leq t^{-(\delta - \delta')}\int_t^{1+t}\norm{y'(s)}_V^2 x^{\delta}\,\d x
        \leq t^{-(\delta - \delta')}\int_{1}^\infty \norm{y'(x)}_V^2 x^{\delta}\,\d x.
    \end{align*}
    Using $(x-t)^{\delta'} \leq x^{\delta'} \leq t^{-(\delta - \delta')} x^{\delta}$ for $x \geq t \geq 1$, we arrive at
    \begin{align*}
        \int_1^\infty \norm{y'(x+t)}_V^2 x^{\delta'}\,\d x
        = \int_{1+t}^\infty \norm{y'(x)}_V^2 (x-t)^{\delta'}\,\d x
        \leq t^{-(\delta - \delta')}\int_1^\infty \norm{y'(x)}_V^2 x^{\delta}\,\d x.
    \end{align*}
    Finally, when $t \in [0,1]$ we obtain from $w_{\delta',\eta}(x-t) \leq w_{\delta',\eta}(x) \leq w_{\delta,\eta}(x)$ the bound
    \begin{align*}
        \int_{0}^{\infty}\|y'(x+t)\|_V^2 w_{\delta',\eta}(x)\, \mathrm{d}x
        \leq \int_t^{\infty} \|y'(x)\|_V^2 w_{\delta',\eta}(x)\, \mathrm{d}x
        \leq \int_0^{\infty} \|y'(x)\|_V^2 w_{\delta,\eta}(x)\, \mathrm{d}x.
    \end{align*}
    Combining all inequalities proves the second inequality in assertion (a).

    (b) Using assertion (a), we obtain for $a\in\{b,\sigma\}$
    \begin{align}
        \norm{S(t)E_a}_{L(H_a,\mathcal{H}_{\delta,\eta})} &\leq \norm{S(t/2) - S_\infty }_{L(\mathcal{H}_{\delta^\ast,\eta},\mathcal{H}_{\delta,\eta})} \norm{S(t/2)E_a}_{L(H_a,\mathcal{H}_{\delta^\ast,\eta})} \nonumber
        \\ &\leq \max\left\{1,\frac{1}{(\delta-1)^{1/2}}\right\}   (1 \vee t)^{- \frac{\delta^* - \delta}{2}} \norm{S(t/2)E_a}_{L(H_a,\mathcal{H}_{\delta^\ast,\eta})},\label{eq: reg property HJM}
    \end{align}
     where the first inequality follows from $S_\infty S(t/2) E_a = 0$. By noting that $\eta^\ast<2<\delta^\ast$ and using Lemma \ref{lemma: HJM lift technical lemma}, we find for the last term
    \begin{align*}
        \norm{S(t/2)E_a}_{L(H_a,\mathcal{H}_{\delta^\ast,\eta})}\leq \norm{S(t/2)}_{L(\mathcal{H}_{\delta^\ast,\eta^\ast},\mathcal{H}_{\delta^\ast,\eta})}\norm{E_a}_{L(H_a,\mathcal{H}_{\delta^\ast,\eta^\ast})}\lesssim (1+(1\lor t)^{-\frac{\eta^\ast-\eta}{2}}).
    \end{align*}
     This proves that
    \[
        \| S(t)E_b \|_{L(H_b, \mathcal{H}_{\delta, \eta})},\ \| S(t)E_{\sigma} \|_{L_q(H_{\sigma}, \mathcal{H}_{\delta, \eta})} \lesssim (1 \vee t)^{- \frac{\delta^* - \delta}{2}}
    \]
    which yields \eqref{eq:assumption long-timme - integrability} provided that $\delta^\ast-\delta>2$.

    (c) When $b\equiv0$ and $E_b\equiv0$, then we only need that $\norm{S(\cdot)E_\sigma}_{L_q(H_\sigma,\mathcal{H}_{\delta,\eta})}^{2}$ is integrable, whence $\delta^\ast-\delta>1$ is sufficient.
\end{proof}

\begin{example}
    Let $E \in L(H_b, V)$ and $\widetilde{E} \in L_q(H_{\sigma}, V)$, and define for $\alpha, \beta \in (\frac{1}{2}, \frac{3}{2})$
    \[
        E_b(t) = \frac{t^{\alpha-1}}{\Gamma(\alpha)}\ E \ \text{ and } \
        E_{\sigma}(t) = \frac{t^{\beta - 1}}{\Gamma(\beta)}\ \widetilde{E}.
    \]
    Then \eqref{eq: lift condition 2} is satisfied for any choice $\delta^*, \eta^*$ such that
    \[
        \eta^* > 3-2 (\alpha \wedge \beta)  \ \text{ and } \ \delta^* < 3 - 2(\alpha \vee \beta).
    \]
    Consequently, by Lemma \ref{lemma: HJM lift technical lemma}, Assumption \ref{assumption SEE} is satisfied for
    \[
        \mathcal{H} = \mathcal{H}_{\delta, \eta^*}, \qquad \mathcal{V} = \mathcal{H}_{\delta, \eta}, \qquad \rho = (1-(\alpha\land\beta)+\varepsilon_0)_+
    \]
    where $\eta^* = 3 - 2(\alpha \wedge \beta) + \varepsilon_0$, $\eta=\eta^\ast-(\eta^\ast-1+\varepsilon_0)_+$, $\delta^* = 3 - 2(\alpha \wedge \beta) - \varepsilon_1$, with $\varepsilon_0 \in (0, (\alpha\land\beta)-\frac12)$, $\varepsilon_1\in(0,3-2(\alpha\lor\beta))$,  $\delta \in [0, \delta^*]$ are arbitrary. Remark that Assumption \ref{assumption long-time} does not hold.
\end{example}

\subsection{Fractional kernels in the mild formulation}

Similarly to Section \ref{sec:frac_kernels_mild}, let us consider the case where $(A,D(A))$ admits an orthonormal basis $(e_n^H)_{n\geq 1}$ of eigenvectors such that $Ae_n^H = - \theta_n e_n^H$, see \eqref{eq: A diagonal}, where $(\theta_n)_{n\geq 1}$ denotes the increasing sequence of nonnegative eigenvalues. Let $W$ be a Gaussian process given as in \eqref{eq: W}, and recall that we denotes its covariance operator by $Q = \sum_{n = 1}^{\infty} \lambda_n (e_n^H \otimes e_n^H)$. Below, we study the stochastic Volterra equation with fractional kernels and either additive or multiplicative noise of the form \eqref{VSPDE} with
\[
    k(t) = \frac{t^{\alpha - 1}}{\Gamma(\alpha)} \ \text{ and } \ h(t) = \frac{t^{\beta - 1}}{\Gamma(\beta)}
\]
where $\alpha, \beta > \frac{1}{2}$. For simplicity, we additionally set $b\equiv0$. The corresponding mild formulation takes the form
\begin{equation}\label{eq: HJM example mild frac}
    u(t;G) = G(t) +  \int_0^t E^{\alpha,\beta}(t-s)\sigma(u(s;G))\,\d W_s
\end{equation}
with $E_b\equiv0$ and $E_\sigma=E^{\alpha,\beta}$. Due to Remark \ref{remark resolvents frac}, $E^{\alpha,\beta}$ is given by
\[
    E^{\alpha,\beta}(t) = \sum_{n=1}^{\infty} e_n(t; \alpha,\beta) (e_n^H \otimes e_n^H),
\]
where $e_n(t; \alpha, \beta) = t^{\beta - 1}E_{\alpha, \beta}(- \theta_n t^{\alpha})$. Recall that $H^{\varkappa} = D((-A)^{\varkappa})$ denotes the fractional domain of $(A, D(A))$ defined in \eqref{eq:H-kappa space}. Below, we verify condition \eqref{eq: lift condition 2} with $q=2$.

\begin{lemma}\label{lemma HJM mild}
Suppose that \eqref{eq: A diagonal} holds. Let $\alpha\in(0,2)$, $\beta\in\R$ and $\gamma ,\varkappa \in \R$ be such that $\beta-\alpha\notin\mathbb{Z}\setminus\N$ and $\sum_{n=1}^\infty\theta_n^{2(\gamma-\varkappa)-4}\1_{\{\alpha=\beta\}} + \theta_n^{2(\gamma-\varkappa)-2}\1_{\{\alpha\neq\beta\}}<\infty$. Take $\delta^\ast,\eta^\ast\in\R$ such that
\[
\delta^\ast,\eta^\ast \in (3 - 2\beta, 3 - 2(\beta - \alpha)).
\]
Then $E^{\alpha, \beta}\colon (0, \infty) \longrightarrow L_2(H^{\varkappa}, H^{\gamma})$ is absolutely continuous and satisfies
\begin{align*}
    & \int_0^{\infty} \| (E^{\alpha, \beta})'(x)\|_{L_2(H^{\varkappa}, H^{\gamma})}^2 w_{\delta^*, \eta^*}(x)\, \d x
    \\ & \qquad \leq \left( \int_{0}^\infty |x^{\beta-2}E_{\alpha,\beta-1}(-x^\alpha)|^2 w_{\delta^\ast,\eta^\ast}(x) \,\d x\right)\sum_{n=1}^\infty \theta_n^{2(\varkappa - \gamma) + \frac{3-2\beta-\eta^\ast}{\alpha}}.
\end{align*}
\end{lemma}
\begin{proof}
    Recall the Poincar\'e asymptotics of the Mittag-Leffler function for $\alpha\in(0,2)$ and $\beta\in\R$, $E_{\alpha,\alpha}(-x^\alpha)\sim\frac{\sin(\alpha,\pi)}{\alpha\pi\Gamma(\alpha)}x^{-2\alpha}$ and $E_{\alpha,\beta}(- x^\alpha)\sim -\frac{x^{-\alpha}}{\Gamma(\beta-\alpha)}$ as $x\to\infty$. Since $(e_n^{H^{\varkappa}})_{n\geq1}=(\theta_n^{-\varkappa} e_n^H)_{n\geq1}$ is an orthonormal basis of $H^{\varkappa}$, we obtain for each $t > 0$
    \begin{align*}
        \| E^{\alpha, \beta}(t)\|_{L_2(H^{\varkappa}, H^{\gamma})}^2
        &= \sum_{n=1}^{\infty}\theta_n^{2(\gamma - \varkappa)} \left| t^{\beta - 1}E_{\alpha, \beta}(- \theta_n t^{\alpha})\right|^2
        \\ &\lesssim t^{2\beta - 2 - 4\alpha}\1_{\{\alpha = \beta\}}\sum_{n=1}^{\infty}\theta_n^{2(\gamma - \varkappa) - 4}
        + t^{2\beta - 2 - 2\alpha}\1_{\{\alpha \neq \beta\}}\sum_{n=1}^{\infty}\theta_n^{2(\gamma - \varkappa) - 2}.
    \end{align*}
    This shows that $E^{\alpha, \beta}\colon (0,\infty) \longrightarrow L_2(H^{\varkappa}, H^{\gamma})$. Using the particular form, differentiation term by term yields
     \[
        \frac{\d}{\d x}E^{\alpha, \beta}(x) = \sum_{n=1}^\infty e_n'(x;\alpha,\beta)(e_n^H\otimes e_n^H)
    \]
    where $e_n'(x;\alpha,\beta)=x^{\beta-2}E_{\alpha,\beta-1}(-\theta_n x^{\alpha})$. For  $x \in [0,1]$ we obtain
    \begin{align*}
        \int_0^1 |x^{\beta-2}E_{\alpha,\beta-1}(-\theta_n x^\alpha)|^2x^{\eta^\ast}\,\d x \leq \theta_n^{\frac{3-2\beta-\eta^\ast}{\alpha}} \int_0^\infty |y^{\beta-2}E_{\alpha,\beta-1}(-y^\alpha)|^2y^{\eta^\ast}\,\d y.
    \end{align*}
    Using again the Poincaré asymptotics, we conclude that the above integral is finite for $\eta^\ast\in(3-2\beta,3-2(\beta-\alpha))$. Similarly, when $x > 1$ we obtain
    \begin{align*}
        \int_1^\infty |x^{\beta-2}E_{\alpha,\beta-1}(-\theta_nx^\alpha)|^2x^{\delta^\ast}\,\d x \leq \theta_n^{\frac{3-2\beta-\delta^\ast}{\alpha}} \int_{0}^\infty |y^{\beta-2}E_{\alpha,\beta-1}(-y^\alpha)|^2y^{\delta^\ast}\,\d y.
    \end{align*}
    Also here, the integral is finite when $\delta^* \in (3 - 2\beta, 3 - 2(\beta - \alpha))$. Thus, we obtain for $\delta^\ast,\eta^\ast$ as above
    \begin{align*}
        & \int_0^\infty \norm{ (E^{\alpha,\beta})'(x)}_{L_2(H^{\varkappa},H^\gamma)}w_{\delta^\ast,\eta^\ast}(x)\,\d x
        \\ &\qquad =  \sum_{n=1}^\infty \theta_{n}^{2(\gamma - \varkappa)}\int_0^\infty |x^{\beta-2}E_{\alpha,\beta-1}(-\theta_n x^\alpha)|^2w_{\delta^\ast,\eta^\ast}(x)\,\d x
        \\ &\qquad \leq \left(\int_{0}^\infty |x^{\beta-2}E_{\alpha,\beta-1}(-x^\alpha)|^2 w_{\delta^\ast,\eta^\ast}(x) \,\d x \right) \sum_{n=1}^\infty \theta_n^{2(\gamma - \varkappa) + \frac{3-2\beta-\eta^\ast}{\alpha}}.
    \end{align*}
\end{proof}

In the following, we work in the following setup:
\[
V=H=H_b \quad\text{and}\quad H_\sigma=H^\gamma,
\]
for $\gamma\in\R$ arbitrary but fixed, and denote by $\mathcal{H}_{\delta,\eta}$ the corresponding scales of Hilbert spaces defined in Subsection~\ref{subsec:general_framework}. Recall that $S(t)y(x)=y(t+x)$ and that $\Xi y= y(0)$. In this setting, let us suppose that $G(t)=g(t)+A\int_0^t E^{\alpha,\alpha}(t-s)g(s)\,\d s$ appearing in \eqref{eq: HJM example mild frac} is of the form $G=\xi$ for some $\xi\in L^p(\Omega,\F_0,\P;\mathcal{H}_{\delta,\eta})$ and $\delta,\eta$ are specified below. We use $\mathcal{G}_{p}(\delta,\eta)$ to denote the collection of all such admissible driving forces $G$.

\begin{theorem}\label{theorem frac HJM lift}
Suppose that $W$ is a cylindrical Wiener process with covariance operator $Q$. Suppose that $\sigma\colon H\longrightarrow L_2(Q^{\frac12}H,H^\gamma)$ is Lipschitz continuous with constant $C_{\sigma,\mathrm{lip}}$  and linear growth constant $C_{\sigma,\mathrm{lin}}=\sup_{x\in H}\frac{\norm{\sigma(x)}_{L_2(Q^{1/2}H,H^\gamma)}}{1+\norm{x}_H}$, $\alpha\in(0,2)$, $\beta\in(\frac{1}{2},\frac{1}{2}+\alpha)$.
    Let $p\in(2,\infty)$ and $\varepsilon_0,\varepsilon_1>0$ satisfy
    \[
   \frac2p <\varepsilon_0<1, \quad\text{and}\quad 0< \varepsilon_1 < \frac{1-2(\beta-\alpha)}{2},
    \]
    and suppose that $\sum_{n=1}^\infty \theta_n^{-2\gamma+\frac{1-2\beta+\varepsilon_0}{\alpha}}<\infty$. Define $\delta,\eta$ by
    \[
    \eta = 1-\varepsilon_0 \ \text{and} \ \delta =  2- 2(\beta-\alpha) -2\varepsilon_1.
    \]
    Then, for each $\xi\in L^p(\Omega,\F_0,\P;\mathcal{H}_{\delta,\eta})$, there exists a unique solution of \eqref{eq: HJM example mild frac} in $\mathcal{H}_{\delta,\eta}$. In particular, setting $G=\Xi S(\cdot)\xi\in\mathcal{G}_p(\delta,\eta)$, $u(\cdot;G)=\Xi X$ is the unique solution of \eqref{eq: mild formulation}. Suppose that additionally, for some $p\in(2,\infty)$ and $\varepsilon_0,\varepsilon_1 >0$,
    \begin{align}
        &4^{1-\frac1p}C_{\sigma,\mathrm{lin}}\left(1 + \sqrt{\frac{1}{\varepsilon_0}}\right)\max\left\{1,\frac{1}{\delta-1}\right\}^{\frac12}  K_0K_1 \left(\frac{2}{2-\varepsilon_0} + \frac{2}{\varepsilon_1}\right)^{1/2} c_p^{\frac1p}<1 \label{eq: bounded moments HJM lift ex}\\
        &2^{1-\frac{1}{p}}C_{\sigma,\mathrm{lip}} \left(1 + \sqrt{\frac{1}{\varepsilon_0}}\right) \max\left\{1,\frac{2}{1-2(\beta-\alpha)-2\varepsilon_1}\right\}^{\frac12}K_0K_1\left(\frac{2}{2 - \varepsilon_0} + \frac{4}{1 - 2(\beta - \alpha)}\right)^{1/2}   c_p^{\frac{1}{p}} < 1,\nonumber
    \end{align}
    where $c_p$ is explicitly given in \eqref{eq: constant BDG} and, the constants $K_0=K_0(\alpha,\beta,\varepsilon_0,\varepsilon_1,\gamma)$ and $K_1=K_1(\alpha,\beta,\varepsilon_1)$ are defined by
    \begin{align*}
        K_0(\alpha,\beta,\varepsilon_0,\varepsilon_1,\gamma)&\coloneqq\left(\|E^{\alpha,\beta}(1) \|_{L_2(H^\gamma, H)}^2 + \int_{0}^\infty |x^{\beta-2}E_{\alpha,\beta-1}(-x^\alpha)|^2 w_{3-2(\beta-\alpha)-\varepsilon_1,1-\varepsilon_0}(x) \,\d x\right)^{1/2}\\ &\qquad\cdot\left(\sum_{n=1}^\infty \theta_n^{-2\gamma + \frac{2(\varepsilon_0-\beta)}{\alpha}}\right)^{1/2},\\
        K_1(\alpha,\beta,\varepsilon_1)&\coloneqq \left(3+\frac{1}{2-2(\beta-\alpha)-\varepsilon_1}\right)^{1/2}.
    \end{align*}
    Then the following assertions hold:
    \begin{enumerate}
        \item[(a)] Equation \eqref{eq: HJM example mild frac} admits a limiting distribution $\pi_{G} \in\mathcal{P}_p(\mathcal{H}_{\delta,\eta})$ with respect to the Wasserstein $p$-distance. This limit distribution is an invariant measure, which is parameterised by $G(\infty)$.
        \item[(b)] Let $\widetilde{\xi}\sim\pi_G$ and set $\widetilde{G}=\Xi S(\cdot)\widetilde{\xi}$. Then $u(\cdot,\widetilde{G})$  is a stationary process corresponding to \eqref{eq: mild formulation}.
    \end{enumerate}
    If additionally \eqref{eq: bounded moments HJM lift ex} is satisfied for $2p$ instead of $p$ and $\xi\in L^{2p}(\Omega,\F_0,\P;\mathcal{H}_{\delta,\eta})$, the Law of Large Numbers holds in the mean-square sense with rate of convergence
    \[
    \vartheta < \frac12 \min\left\{1,\frac{1}{p}\log\left(1/\norm{\rho_{\mathrm{b=0}}^{(p)}}_{L^1(\R_+)}\right),\frac1p\left(\frac{1}{2}-(\beta-\alpha)-\varepsilon_1\right)\right\},
    \]
    where $\rho_{\mathrm{b=0}}^{(p)}$ is defined in \eqref{eq: rho b zero}.
\end{theorem}

\begin{proof}
    Denote by $\mathcal{H}_{\delta,\eta}$ the scale of Hilbert spaces defined in Subsection~\ref{subsec:general_framework}. It follows from Theorem~\ref{theorem HJJM lift} (c) that Assumptions~\ref{assumption SEE} and~\ref{assumption long-time} (b) and (c) are satisfied for $q=q'=2$, and
    \[
    \mathcal{H} = \mathcal{H}_{\delta,\eta^\ast}, \ \mathcal{V} = \mathcal{H}_{\delta,\eta}, \ \mathcal{V}_0 = \mathcal{H}_{\delta_\ast,\eta}, \ \lambda=\frac{\delta-\delta_\ast}{2}, \ \rho=\frac{\eta^\ast-\eta}{2}
    \]
    where $\eta\in[0,1)$ satisfies $\eta\leq\eta^\ast<1+\eta$ and $\delta_\ast\in(1,\delta^\ast)$, $\delta\in(\delta_\ast,\delta^\ast)$  satisfy $\delta^\ast-\delta>1$. In light of Lemma \ref{lemma HJM mild}, remark that $\delta^\ast,\eta^\ast$ necessarily satisfy $\delta^\ast,\eta^\ast \in (3 - 2\beta, 3 - 2(\beta - \alpha))$. Thus, let us take $\eta^\ast=2-2\varepsilon_0$, $\delta^\ast = 3-2(\beta-\alpha)-\varepsilon_1$ and $\delta_\ast = 1 + \frac{1-2(\beta-\alpha)-2\varepsilon_1}{2}$. Then $\delta=\delta^\ast-1-\varepsilon_1$ and the above conditions are satisfied with
    \[
    \lambda= \frac{1}{4}-\frac{\beta-\alpha}{2}-\frac{\varepsilon_1}{2}>0\quad\text{and}\quad \rho = \frac{1}{2}-\frac{\varepsilon_0}{2}\in \left[0,\frac12 \right).
    \]
    Moreover, since $\varepsilon_0>\frac2p$, we also obtain $\rho+\frac{1}{p}<\frac{1}{2}$, see \eqref{eq: continuity}. Finally, by assumption $\sigma$ is Lipschitz continuous with constant $C_{b,\mathrm{lip}}$, thus also Assumption \ref{assumption long-time} (a) is satisfied. The existence and uniqueness of solutions follow from Theorem \ref{label:existence_uniqueness_VSPDE}. Concerning limit distributions, our assertion follows from Theorem \ref{theorem_limit_distribution} (b) provided that \eqref{eq:bounded_moment_condi_multi b zero} and  \eqref{eq:contraction_estimate_condi_multi b zero} are satisfied for a $p\in(2,\infty)$. To verify the latter, following Lemma \ref{lemma: HJM lift technical lemma} (b), we find
    \begin{align*}
        \norm{\Xi}_{L(\mathcal{H}_{\delta,\eta},V)} \leq 1 +\left(\int_0^1 \frac{\d x}{w_{\delta,\eta}(x)}\right)^{1/2}  = 1 + \sqrt{\frac{1}{1-\eta}}
    \end{align*}
    and analogously, $\norm{\Xi}_{L(\mathcal{H}_{\delta_\ast,\eta},V)}\leq 1+\sqrt{\frac{1}{1-\eta}}$. Moreover, using \eqref{eq: reg property HJM}, we obtain for our particular choice of $H_\sigma=H^\gamma$ and $\mathcal{V}=\mathcal{H}_{\delta,\eta}$

    \begin{multline}\label{eq: estimate L2 condition HJM}
        \int_0^\infty \norm{S(t)E^{\alpha,\beta}}_{L_2(H^\gamma,\mathcal{H}_{\delta,\eta})}^2\,\d t\\ \leq \max\left\{1,\frac{1}{\delta-1}\right\}\norm{E^{\alpha,\beta}}^2_{L_2(H^\gamma,\mathcal{H}_{\delta^\ast,\eta})}\int_0^\infty \left(1\lor  t/2\right)^{-(\delta^\ast-\delta)}\norm{S(t/2)}^2_{L(\mathcal{H}_{\delta^\ast,\eta^\ast},\mathcal{H}_{\delta^\ast,\eta})}\,\d t.
    \end{multline}
    In particular,
    \begin{align*}
        \norm{S(t/2)}^2_{L(\mathcal{H}_{\delta^\ast,\eta^\ast},\mathcal{H}_{\delta^\ast,\eta})}\leq  \left( 3+\int_1^\infty \frac{\d x}{w_{\delta^\ast,\eta}(x)} \right) (1\land t/2)^{-(\eta-\eta^\ast)}.
    \end{align*}
    Moreover, let $(e_n^{H^\gamma})_{n\geq1}$ be an orthonormal basis of $H^\gamma$ and remark that
    \begin{align*}
        \norm{E^{\alpha,\beta}}^2_{L_2(H^\gamma,\mathcal{H}_{\delta^\ast,\eta})} &= \sum_{n=1}^\infty \vertiii{E^{\alpha,\beta}e_n^{H^\gamma}}_{\delta^\ast,\eta}^2\\ &= \sum_{n=1}^\infty \norm{E^{\alpha,\beta}(1)e_n^{H^\gamma}}^2_H + \sum_{n=1}^\infty \int_0^\infty \norm{(E^{\alpha,\beta})'(x)e_n^{H^\gamma}}_H^2 w_{\delta^\ast,\eta}(x)\,\d x \\
        &= \| E^{\alpha,\beta}(1) \|_{L_2(H^\gamma, H)}^2
        + \int_0^\infty \| (E^{\alpha,\beta})'(x) \|_{L_2(H^\gamma, H)}^2
         w_{\delta^\ast,\eta}(x) \, \mathrm{d}x
    \end{align*}
    where the last integral can be estimated using Lemma \ref{lemma HJM mild}. Hence,  \eqref{eq:bounded_moment_condi_multi b zero} is satisfied by assumption.  By replacing $\delta$ by $\delta_\ast$ in \eqref{eq: estimate L2 condition HJM}, we also conclude that \eqref{eq:contraction_estimate_condi_multi b zero} is satisfied by assumption. The Law of Large Numbers, including the convergence rate, is a consequence of Corollary \ref{theorem LLN lift}. Finally, recall that $S_{\infty}G = G(\infty)$ which, depending on the choice of $G$, may be nontrivial as shown in the examples below. This implies the nonuniqueness of invariant measures.
\end{proof}

It is interesting to note that, by using the lift presented in this section, we fail to conclude the uniqueness of limiting and invariant distributions as $S_\infty\neq 0$. For the latter, the following example provides two elements in $\mathcal{H}_{\delta,\eta}$ with non-trivial limit.
\begin{example} Let $e\in V$ be such that $\norm{e}_V=1$ and let $\delta,\eta\in\R$.
    \begin{enumerate}
        \item[(a)] Define $G\colon (0,\infty)\longrightarrow V$ by $G\equiv e$. Then $G'\equiv0$ and hence $G\in\mathcal{H}_{\delta,\eta}$ with $G(\infty)=e\neq0$.
    \item[(b)] Let $p>\max\{1,\frac{\delta+1}{2}\}$ and define $G\colon (0,\infty)\longrightarrow  V$ by
    \[
    G(x) = \begin{cases}G(1), &\text{if }x\in(0,1] \\
    e-\left(\int_x^\infty t^{-p}\,\d t\right)e, & \text{if }x\in(1,\infty)
    \end{cases}
    \]
    where $G(1)= e-\left(\int_1^\infty t^{-p}\,\d t\right)e$. Then $\norm{G(1)}_V^2<\infty$, and  $G'(x)=x^{-p}e$ on $x\in(1,\infty)$ and $0$ elsewhere, so $\int_0^\infty \norm{G'(x)}_V^2 w_{\delta,\eta}(x)\,\d x = \int_1^\infty x^{-2p}x^\delta \,\d x <\infty$. Thus, $G\in \mathcal{H}_{\delta,\eta}$ and $G(\infty)=e\neq0$.
    \end{enumerate}
\end{example}

To illustrate this result, let us consider the case of the Dirichlet Laplacian operator for $(A,D(A))$.

\begin{example}
    Let $\mathcal{O} \subset \R^d$ be a bounded domain with $C^1$-boundary, and set $H = L^2(\mathcal{O})$. Then $(A,D(A)) = (\Delta, H_0^1(\mathcal{O}) \cap H^2(\mathcal{O}))$ is diagonalisable with an orthonormal basis $(e_n^H)_{n \geq 1}$ and sequence of eigenvalues $(\theta_n)_{n \geq 1}$  Without loss of generality, we suppose that the latter are increasing to infinity. By Weyl's law, we find for their asymptotics
    \[
        \theta_n \sim c(d, \mathcal{O})n^{2/d}, \qquad n \to \infty,
    \]
    where $c(d, \mathcal{O}) > 0$ denotes some constant. Hence, the summability condition in Theorem \eqref{theorem frac HJM lift} becomes $\sum_{n=1}^\infty n^{\frac{-4}{d}\varkappa+\frac{2}{d}\frac{1-2\beta+\varepsilon_0}{\alpha}}<\infty$ and is satisfied whenever
    \[
    2\alpha\left(\varkappa-\frac{1-2\beta}{2\alpha}-\frac{d}{4}\right) > \varepsilon_0.
    \]
\end{example}

Similarly to the situation outlined in section \ref{section markovia lift cm}, the convergence rate in the Law of Large Numbers is too small to obtain the Central Limit Theorem. The next remark outlines how the optimal rate of convergence can be obtained via exponential damping.

\begin{remark}
    For given $\lambda>0$, let us consider the Volterra kernels
    \[
    k(t) = \frac{t^{\alpha-1}}{\Gamma(\alpha)}\e{-\lambda t}\quad\text{and}\quad h(t)=\frac{t^{\beta-1}}{\Gamma(\beta)}\e{-\lambda t}.
    \]
    Then $E^{\alpha,\beta}$ needs to be replaced by
    \[
    E^{\alpha,\beta,\lambda}(t) = \sum_{n=1}^\infty e_n(t;\alpha,\beta,\lambda)(e^H_n\otimes e_n^H)
    \]
    where $e_n(t;\alpha,\beta,\lambda)=t^{\beta-1}E_{\alpha,\beta}(-\theta_n t^\alpha)\e{-\lambda t}$. Hence, we may derive similar bounds to Lemma \ref{lemma HJM mild} with the only difference that $\delta^\ast$ can be chosen arbitrarily large. The latter is sufficient to verify the conditions of Theorem \ref{theorem CLT general}  and hence derive a Central Limit Theorem.
\end{remark}


	\appendix

    \section{Proofs from Section 2}

    \subsection{Proof of Theorem \ref{label:existence_uniqueness_VSPDE}}

\begin{proof}
        Fix $\lambda<0$ and $s \in [0,T)$, and define the rescaled semigroup $S_\lambda(t)=\e{\lambda t}S(t)$. Suppose that $X(\cdot; s, \xi)\in L^p(\Omega,\P; C([s,T]; \mathcal{V}))$ is a solution of \eqref{eq: abstract Markov lift time inhomogeneous}. Then $X_\lambda(t; s, \xi)=\e{\lambda (t-s)}X(t; s, \xi)$ solves
		\begin{align}\label{eq:scaled H-SVE lift}
			X_\lambda(t; s, \xi) &= S_\lambda(t-s)\xi + \int_s^t S_\lambda(t-r)\xi_b^{\lambda}(r, \Xi X_\lambda(r; s, \xi))\,\d r
			\\ &\qquad \qquad \qquad + \int_s^t S_\lambda(t-r)\xi_\sigma^{\lambda}(r, \Xi X_\lambda(r; s, \xi))\,\d W_r, \notag
		\end{align}
		where $\xi_b^{\lambda}(r,u) = \e{\lambda (r-s)}\xi_b(r, \e{-\lambda (r-s)}u)$ and $\xi_{\sigma}^{\lambda}(r,u)=\e{\lambda (r-s)}\xi_\sigma(r, \e{-\lambda (r-s)}u)$. Conversely, let $X_\lambda(\cdot, s, \xi)$ be a solution of \eqref{eq:scaled H-SVE lift}, then $X(t; s, \xi)= \e{-\lambda (t-s)}X_\lambda(t; s, \xi)$ solves \eqref{eq: abstract Markov lift time inhomogeneous}. Therefore, it suffices to prove the existence and uniqueness of \eqref{eq:scaled H-SVE lift}.

		Let us define $\mathcal{T}_{\lambda}(X)(t) = S_{\lambda}(t - s)\xi + \mathcal{S}_{\lambda}(X)(t)$, where
		\begin{equation*}\label{eq: T lambda}
			\mathcal{S}_\lambda(X)(t) = \int_s^t S_\lambda(t-r)\xi_b^{\lambda}(r,\Xi X(r))\,\d r
			+ \int_s^t S_\lambda(t-r)\xi_\sigma^{\lambda}(r, \Xi X(r))\,\d W_r.
		\end{equation*}
		Then \eqref{eq:scaled H-SVE lift} reads as $X_\lambda(\cdot;\xi)=\mathcal{T}_\lambda(X_\lambda(\cdot;\xi))$. In the following, we show that $\mathcal{T}_\lambda$ is a contraction on $L^p(\Omega,\P;C([s,T];\mathcal{V}))$. Firstly, we obtain from the first inequality in Proposition \ref{proposition_properties_lift} the bound
        \begin{align*}
            &\ \left\| \int_s^{\cdot} S_\lambda(\cdot-r)\xi_b^{\lambda}(r,\Xi X(r))\,\d r \right\|_{L^p(\Omega;C([s,T];\mathcal{V}))}
            \\ &= \left\| \int_0^{\cdot-s} S_\lambda( \cdot - s - r)\xi_b^{\lambda}(r+s,\Xi X(r+s))\,\d r \right\|_{L^p(\Omega;C([s,T];\mathcal{V}))}
            \\ &= \left\| \int_0^{\cdot} S_\lambda( \cdot - r)\xi_b^{\lambda}(r+s,\Xi X(r+s))\,\d r \right\|_{L^p(\Omega;C([0,T-s];\mathcal{V}))}
            \\ &\leq \left( \int_0^{T-s} \|S_{\lambda}(r)\|_{L(\mathcal{H}, \mathcal{V})}^{\frac{p}{p-1}}\, \mathrm{d}r \right)^{p-1} \left\| \xi_b^{\lambda}(\cdot+s,\Xi X(\cdot+s))\right\|_{L^p(\Omega;L^p([0,T-s];\mathcal{V}))}
            \\ &\leq C(T, p, \rho, \xi_b) \left( 1 + \| \Xi\|_{L(\mathcal{V}, V)} \| X\|_{L^p(\Omega;C([s,T];\mathcal{V}))} \right)
        \end{align*}
        where we have used the Lipschitz continuity from Assumption \ref{assumption Lipschitz} and $C(T,p, \rho, \xi_b) > 0$ is some constant. Similarly, using the second inequality in Proposition \ref{proposition_properties_lift} and setting $W_r^s = W_{r+s} - W_s$, which is again a Wiener process, we find for $0 < \alpha < \frac{1}{2} - \rho$
        \begin{align*}
            &\ \left\| \int_s^{\cdot} S_\lambda( \cdot-r)\xi_\sigma^{\lambda}(r, \Xi X(r))\,\d W_r \right\|_{L^p(\Omega;C([s,T];\mathcal{V}))}
            \\ &\qquad = \left\| \int_0^{\cdot - s} S_\lambda(\cdot - r)\xi_\sigma^{\lambda}(r + s, \Xi X(r + s))\,\d W^s_r \right\|_{L^p(\Omega;C([0,T-s];\mathcal{V}))}
            \\ &\qquad \leq A \left( \int_0^{T-s} r^{-2\alpha} \| S_{\lambda}(r)\|_{L(\mathcal{H}, \mathcal{V})}^2\, \mathrm{d}r \right)^{\frac{1}{2}}
             \cdot \| \xi_\sigma^{\lambda}(\cdot + s, \Xi X(\cdot + s)) \|_{L^{\infty}([0,T-s]; L^p(\Omega; L_2(U, \mathcal{H})))}
            \\ &\qquad \leq C'(T, p, \rho, \alpha, \xi_{\sigma}) \left( 1 + \| \Xi\|_{L(\mathcal{V}, V)} \| X\|_{L^p(\Omega;C([s,T];\mathcal{V}))} \right)
        \end{align*}
        where $A = A(p, \rho, \alpha, T-s)$ and $C'(T, p, \rho, \alpha, \xi_{\sigma})$ is another constant. Since also $\| S_{\lambda}(\cdot - s)\xi \|_{L^p(\Omega; C([s,T]; \mathcal{V}))} \lesssim \| \xi \|_{L^p(\Omega; \mathcal{V})}$, we see that $\mathcal{T}_\lambda$ with fixed $\xi$ leaves $L^p(\Omega;C([s,T];\mathcal{V}))$ invariant.

        Now let $X,Y\in L^p(\Omega;C([s,T];\mathcal{V}))$ and define $\widetilde{\xi}_b(r; X) = \xi_b^{\lambda}(r+s,\Xi X(r+s))$, $\widetilde{\xi}_b(r; Y) = \xi_b^{\lambda}(r+s,\Xi Y(r+s))$, $\widetilde{\xi}_{\sigma}(r; X) = \xi_{\sigma}^{\lambda}(r+s,\Xi X(r+s))$, and $\widetilde{\xi}_{\sigma}(r; Y) = \xi_{\sigma}^{\lambda}(r+s,\Xi Y(r+s))$. Then we obtain from similar arguments to those above with $A = A(p, \rho, \alpha, T-s)$
		\begin{align*}
				&\ \norm{\mathcal{S}_\lambda(X)-\mathcal{S}_\lambda(Y)}_{L^p(\Omega;C([0,T];\mathcal{V}))}
                \\ &\leq \left( \int_0^{T-s} \|S_{\lambda}(r)\|_{L(\mathcal{H}, \mathcal{V})}^{\frac{p}{p-1}}\, \mathrm{d}r \right)^{p-1} \left\| \widetilde{\xi}_b(\cdot; X) - \widetilde{\xi}_b(\cdot; Y)\right\|_{L^p(\Omega;L^p([0,T-s];\mathcal{V}))}
                \\ &+ A \left( \int_0^{T-s} r^{-2\alpha} \| S_{\lambda}(r)\|_{L(\mathcal{H}, \mathcal{V})}^2\, \mathrm{d}r \right)^{\frac{1}{2}} \| \widetilde{\xi}_{\sigma}(r; X) - \widetilde{\xi}_{\sigma}(\cdot; Y) \|_{L^{\infty}([0,T-s]; L^p(\Omega; L_2(U, \mathcal{H})))}
                \\ &\leq C_{\mathrm{lip}}(T) \left( \int_0^{T-s} \|S_{\lambda}(r)\|_{L(\mathcal{H}, \mathcal{V})}^{\frac{p}{p-1}}\, \mathrm{d}r \right)^{p-1} \| \Xi\|_{L(\mathcal{V}, V)} \| X - Y \|_{L^p(\Omega; C([s,T]; \mathcal{V}))}
                \\ &\qquad + C_{\mathrm{lip}}(T) A \left( \int_0^{T-s} r^{-2\alpha} \| S_{\lambda}(r)\|_{L(\mathcal{H}, \mathcal{V})}^2\, \mathrm{d}r \right)^{\frac{1}{2}} \| \Xi\|_{L(\mathcal{V}, V)} \| X - Y \|_{L^p(\Omega; C([s,T]; \mathcal{V}))}.
		\end{align*}
		By Lebesgue's dominated convergence theorem, the constant on the right-hand side becomes arbitrarily small as $\lambda\to-\infty$. Hence, we can choose $\lambda$ sufficiently negative such that $\mathcal{T}_\lambda$ is a contraction.

        Let us denote by $X_\lambda(\cdot; s, \xi)$ the unique fixed point of $\mathcal{T}_{\lambda}$. Then it is the unique solution of \eqref{eq:scaled H-SVE lift}, and setting $X(t; s, \xi)=\e{-\lambda (t-s)}X_\lambda(t; s, \xi)$ we also obtain the unique solution of \eqref{eq: abstract Markov lift time inhomogeneous}. Finally, let $\xi,\widetilde{\xi}\in L^p(\Omega, \F_s, \P;\mathcal{V})$ and denote by $X_\lambda(\cdot;s, \xi),X_\lambda(\cdot; s, \widetilde{\xi})$ the corresponding solutions of \eqref{eq:scaled H-SVE lift}. Using the contraction property, we find
		\begin{align*}
			&\norm{X_\lambda(\cdot; s, \xi)-X_\lambda(\cdot; s, \widetilde{\xi})}_{L^p(\Omega;C([s,T];\mathcal{V}))}
			\\ &\quad \leq \norm{S_\lambda(\cdot - s)(\xi-\widetilde{\xi})}_{L^p(\Omega;C([s,T];\mathcal{V}))} + \norm{\mathcal{S}_\lambda( X_\lambda(\cdot; s, \xi)) - \mathcal{S}_\lambda(X_\lambda(\cdot; s, \widetilde{\xi}))}_{L^p(\Omega;C([s,T];\mathcal{V}))}
			\\ &\quad\leq \sup_{r \in [s,T]} \| S_{\lambda}(r-s)\|_{L(\mathcal{V})} \norm{\xi-\widetilde{\xi}}_{L^p(\Omega;\mathcal{V})} + C(\lambda) \norm{X_\lambda(\cdot; s, \xi)-X_\lambda(\cdot; s, \widetilde{\xi})}_{L^p(\Omega;C([0,T];\mathcal{V}))}
		\end{align*}
        where $C(\lambda) \in (0,1)$ denotes the Lipschitz constant of $\mathcal{S}_{\lambda}$. Inequality \eqref{eq:H-SVE lift initial dependence} then follows from
		\begin{align*}
			\norm{X(\cdot; s, \xi)-X(\cdot; s, \widetilde{\xi})}_{L^p(\Omega;C([s,T];\mathcal{V}))}
			&\leq \e{|\lambda|(T-s)} \norm{X_\lambda(\cdot; s, \xi)-X_\lambda(\cdot; s, \widetilde{\xi})}_{L^p(\Omega;C([s,T];\mathcal{V}))}
			\\ &\leq \frac{\e{|\lambda| (T-s)} \sup_{r \in [0,T-s]} \| S_{\lambda}(r)\|_{L(\mathcal{V})} }{1 - C(\lambda)} \norm{\xi-\widetilde{\xi}}_{L^p(\Omega;\mathcal{V})}
		\end{align*}
		which completes the proof.
	\end{proof}

    \subsection{Proof of Corollary \ref{cor: Markov}}

	\begin{proof}
        Firstly, it follows from \eqref{eq:H-SVE lift initial dependence}, that the map $\mathcal{V} \ni \xi \longmapsto X(t; s, \xi)\in L^p(\Omega,\P; \mathcal{V})$ is continuous for any $t\geq s$. Thus, by the dominated convergence theorem, for any $f\in C_b(\mathcal{V})$, $t\geq s$, and any sequence $(\xi_k)_{k\in\N}\subset \mathcal{V}$ such that $\lim_{k\to\infty}\xi_k=\xi\in \mathcal{V}$, we have $\lim_{k\to\infty} P_{s,t}f(\xi_k) = P_{s,t}f(\xi)$, i.e.~$(P_{s,t})_{t\geq s}$ has the $C_b$-Feller property.

		Let $s \geq 0$ and $\xi \in L^p(\Omega, \F_s, \P; \mathcal{V})$ with $p$ satisfying \eqref{eq: continuity}. Then by Theorem \ref{label:existence_uniqueness_VSPDE} the family of unique solutions $X(\cdot,s;\xi) \in L^p(\Omega, \P; C([s,\infty); \mathcal{V}))$ of \eqref{eq: abstract Markov lift time inhomogeneous} satisfies the flow property $X(t; r, X(r; s, \xi)) = X(t; s, \xi)$ $\P$-a.s.~for $0 \leq s \leq r \leq t$. Hence we obtain for every $f\in B_b(\mathcal{V})$
		\begin{align*}
			\E\left[ f(X(t; s, \xi)) \ | \ \F_r \right]
			&= \E\left[ f(X(t; r, X(r; s, \xi)) \ | \ \F_r \right].
		\end{align*}
		Below we consider $X(t; r, \eta)$ for $\F_r$-measurable initial conditions $\eta\in L^p(\Omega, \F_r, \P;\mathcal{V})$. Assume that $\eta$ is simple, i.e., it only takes a finite number of values, so that we can write $\eta=\sum_{j=1}^N\eta_j\1_{A_j}$ with $\eta_j\in\mathcal{V}$ and $A_j\in\F_r$ such that $A_j\cap A_i=\varnothing$ for $j\neq i$ and $\bigcup_{j=1}^N A_j=\Omega$. Then, one can show that $X(t; r, \eta) = \sum_{j=1}^N X(t,r;\eta_j)\1_{A_j}$. Consequently, since $X(t, r, \eta_j)$ is independent of $\F_r$ and the functions $\1_{A_j}$ are $\F_r$-measurable, we have
		\begin{align*}
			\E[f(X(t; r, \eta))\vert \F_r]
            &= \sum_{j=1}^N \E[f(X(t; r, \eta_j))\1_{A_j}\vert\F_r]
			\\ &= \sum_{j=1}^N \E[f(X(t; r, \eta_j))]\1_{A_j}
			= \E[f(X(t; r, z))]|_{z = \eta}
		\end{align*}
		for any $f\in C_b(\mathcal{V})$. If $\eta\in L^p(\Omega, \F_r, \P;\mathcal{V})$, we may approximate it by $\eta_k$ of the above simple form which yields combined with \eqref{eq:H-SVE lift initial dependence}, $f$ being bounded and continuous, and dominated convergence
		\begin{equation}\label{eq: inhom. markov}
			\E[f(X(t; r, \eta))\vert\F_r] = \E[f(X(t; r, z))]|_{z = \eta} = (P_{r,t}f)(\eta).
		\end{equation}
        This proves the assertion.
	\end{proof}

	\section{Convolution tail estimates}

	\begin{lemma}\label{lemma tail convolution}
		Let $T>0$ and $\lambda\in(0,1)$ be fixed and $0\leq f,g\in  L^1(\R_+)$. Then
		\[
		\int_T^\infty (f\ast g)(t)\,\d t \ \leq \ 2\norm{f}_{L^1(\R_+)}\int_{\lambda T}^\infty g(s)\,\d s + \norm{g}_{L^1(\R_+)}\int_{(1-\lambda)T}^\infty f(t)\,\d t.
		\]
	\end{lemma}
	\begin{proof}
		Using Fubini's theorem and the substitution $t'=t-s$, we find
		\begin{align*}
			\int_T^\infty (f\ast g)(t)\,\d t &=\int_T^\infty \int_0^t f(t-s)g(s)\,\d s\,\d t \\
			&=\int_0^T \int_T^\infty f(t-s)g(s)\,\d t\,\d s + \int_T^\infty \int_s^\infty f(t-s)g(s)\,\d t\,\d s\\
			&= \int_0^T \left(\int_{T-s}^\infty f(t')\,\d t'\right)g(s)\,\d s + \int_T^\infty \left(\int_0^\infty f(t')\,\d t'\right)g(s)\,\d s \\
			&= I + \norm{f}_{L^1(\R_+)}\int_{T}^\infty g(s)\,\d s.
		\end{align*}
		Moreover,
		\begin{align*}
			I &= \int_0^{\lambda T} \left(\int_{T-s}^\infty f(t')\,\d t'\right)g(s)\,\d s + \int_{\lambda T}^T \left(\int_{T-s}^\infty f(t')\,\d t'\right)g(s)\,\d s\\
			&\leq \norm{g}_{L^1(\R_+)}\int_{(1-\lambda)T}^\infty f(t')\,\d t'  + \norm{f}_{L^1(\R_+)}\int_{\lambda T}^\infty g(s)\,\d s.
		\end{align*}
		Noting that $\lambda T< T$ and collecting all estimates yields the asserted.
	\end{proof}

	\begin{lemma}\label{lemma convergence rate tail r}
		Let $\rho\in L^1(\R_+)$ be a non-negative Volterra kernel satisfying $\norm{\rho}_{L^1(\R_+)}<1$ and let $r$ be the unique non-negative solution of the linear Volterra equation $r=\rho+r\ast\rho$. Then $r\in L^1(\R_+)$ and satisfies for every $T>0$ and $\kappa \in(0,1)$ fixed
		\begin{equation*}
			\int_T^\infty r(t)\,\d t \ \leq \ \norm{r}_{L^1(\R_+)} T^{-\log\left(1/\norm{\rho}_{L^1(\R_+)}\right)} + \left(\frac{1+2\norm{r}_{L^1(\R_+)}}{1-\norm{\rho}_{L^1(\R_+)}}\right)\int_{\kappa T^{1+\log(1-\kappa)}}^\infty \rho(t)\,\d t.
		\end{equation*}
	\end{lemma}
	\begin{proof}
		Using $r=\rho+r\ast\rho$ and Lemma \ref{lemma tail convolution} we find
		\begin{align*}
			\int_T^\infty r(t)\,\d t &= \int_T^\infty \rho(t)\,\d t + \int_T^\infty (r\ast\rho)(t)\,\d t \\
			&\leq \int_T^\infty \rho(t)\,\d t + 2\norm{r}_{L^1(\R_+)}\int_{\kappa T}^\infty\rho(t)\,\d t + \norm{\rho}_{L^1(\R_+)}\int_{(1-\kappa)T}^\infty r(t)\,\d t \\
			&\leq (1+2\norm{r}_{L^1(\R_+)})\int_{\kappa T}^\infty \rho(t)\,\d t+  \norm{\rho}_{L^1(\R_+)}\int_{(1-\kappa)T}^\infty r(t)\,\d t.
		\end{align*}
		By iterating $N$-times, $N\in\N$, we obtain
		\begin{align*}
			\int_T^\infty r(t)\,\d t &\leq \norm{\rho}_{L^1(\R_+)}^N \int_{(1-\kappa)^N T}^\infty r(t)\,\d t + (1+2\norm{r}_{L^1(\R_+)})\sum_{k=1}^{N-1}\norm{\rho}_{L^1(\R_+)}^k \int_{\kappa(1-\kappa)^{k-1}T}^\infty \rho(t)\,\d t \\
			&\leq \norm{r}_{L^1(\R_+)}\norm{\rho}_{L^1(\R_+)}^N + \frac{1+2\norm{r}_{L^1(\R_+)}}{1-\norm{\rho}_{L^1(\R_+)}}\int_{\kappa(1-\kappa)^{N-1} T}^\infty\rho(t)\,\d t.
		\end{align*}
		Suppose that $\log(T)\leq N\leq \log(T)+1$. Then
		\begin{align*}
			\int_T^\infty r(t)\,\d t &\leq \norm{r}_{L^1(\R_+)}\exp\left(N\log(\norm{\rho}_{L^1(\R_+)})\right) + \frac{1+2\norm{r}_{L^1(\R_+)}}{1-\norm{\rho}_{L^1(\R_+)}} \int_{\kappa(1-\kappa)^{N-1}}^\infty r(t)\,\d t \\
			&\leq \norm{r}_{L^1(\R_+)} T^{-\log(1/\norm{\rho}_{L^1(\R_+)})} + \frac{1+2\norm{r}_{L^1(\R_+)}}{1-\norm{\rho}_{L^1(\R_+)}} \int_{\kappa T^{1+\log(1-\kappa)}}^\infty \rho(t)\,\d t.
		\end{align*}
	\end{proof}


	\bibliographystyle{plainurl}
	\bibliography{literature}

\end{document}